\documentclass[11pt,letterpaper]{article}
\usepackage{microtype}
\usepackage{authblk}
\usepackage{amsthm}
\usepackage{amsfonts}
\usepackage{amsmath}
\usepackage{amssymb}

\usepackage{mathtools}
\usepackage[mathscr]{eucal}
\usepackage{url}
\usepackage[left=2cm,right=2cm,top=2cm,bottom=2cm]{geometry}
\usepackage[frozencache]{minted}
\usepackage{cite}
\usepackage{multicol}
\usepackage{tocloft}
\usepackage{enumerate}
 
\usepackage{tikz-cd}
\usetikzlibrary{positioning}
\usetikzlibrary{fit}
\usetikzlibrary{cd}
\usetikzlibrary{arrows}
\usetikzlibrary{calc}
\usetikzlibrary{decorations.markings}
\tikzset{ed/.style={auto,inner sep=2pt,font=\scriptsize}} %edges
\tikzset{>=stealth'}
\tikzset{vert/.style={draw,circle, minimum size=6mm, inner sep=0pt, fill=white}}
\tikzset{vertbig/.style={draw,circle, minimum size=8mm, inner sep=0pt, fill=white}}
\tikzset{->-/.style={decoration={
      markings,
      mark=at position #1 with {\arrow{>}}},postaction={decorate}}}

\usepackage[hidelinks]{hyperref} % Should be imported LAST

\theoremstyle{plain}
\newtheorem{theorem}{Theorem}[subsection]
\newtheorem{proposition}[theorem]{Proposition}
\newtheorem{lemma}[theorem]{Lemma}
\newtheorem{corollary}[theorem]{Corollary}
\newtheorem{conjecture}[theorem]{Conjecture}

\theoremstyle{definition}
\newtheorem{definition}[theorem]{Definition}
\newtheorem{example}[theorem]{Example}
\newtheorem{remark}[theorem]{Remark}

\newcommand{\C}{\mathscr{C}}
\newcommand{\homC}{\underline{\C}}
\newcommand{\D}{\mathscr{D}}
\newcommand{\E}{\mathscr{E}}
\newcommand{\M}{\mathscr{M}}
\newcommand{\N}{\mathscr{N}}
\newcommand{\T}{\mathscr{T}}
\renewcommand{\S}{\mathscr{S}}

\newcommand{\bN}{\mathbb{N}}

\newcommand{\lenslib}{\texttt{lens}}

\newcommand{\Pastro}{\Phi}
\newcommand{\Set}{\mathbf{Set}}
\newcommand{\Cat}{\mathbf{Cat}}
\newcommand{\Prof}{\mathbf{Prof}}

\newcommand{\SymmMonCat}{\mathbf{SymmMonCat}}
\newcommand{\StrictSymmMonCat}{\mathbf{StrictSymmMonCat}}

\newcommand{\Tele}{\mathbf{Tele}}
\newcommand{\StrictTele}{\mathbf{StrictTele}}
\newcommand{\Tamb}{\mathbf{Tamb}}

\newcommand{\Strong}{\mathbf{Strong}}

\newcommand{\App}{\mathbf{App}}
\newcommand{\Traversable}{\mathbf{Traversable}}

\newcommand{\Optic}{\mathbf{Optic}}
\newcommand{\Twoptic}{\mathbf{Optic}^2}
\newcommand{\Lawful}{\mathbf{Lawful}}
\newcommand{\Lens}{\mathbf{Lens}}
\newcommand{\Prism}{\mathbf{Prism}}
\newcommand{\Setter}{\mathbf{Setter}}
\newcommand{\Traversal}{\mathbf{Traversal}}

\newcommand{\switched}{\mathbin{\tilde{\otimes}}}

\newcommand{\conc}{\mathbb{C}}
\newcommand{\conctwice}{\mathbb{C}^2}

\newcommand{\id}{\mathrm{id}}
\newcommand{\op}{\mathrm{op}}

\DeclareMathOperator{\copr}{copr}
\newcommand{\inl}{\mathrm{inl}}
\newcommand{\inr}{\mathrm{inr}}
\DeclareMathOperator{\im}{im}
\newcommand{\act}{\cdot}

\DeclareMathOperator*{\colim}{\mathrm{colim}}
\newcommand{\teletimes}{\mathbin{\boxtimes}}
\newcommand{\defeq}{\mathrel{\vcentcolon=}}

\newcommand*\circled[1]{\tikz[baseline={([yshift=-0.65ex]current bounding box.center)}]{
   \node[shape=circle,draw,inner sep=1pt] (char) {#1};}}
\newcommand{\actL}{{\circled{\tiny$\mathsf{L}$}}}
\newcommand{\actR}{{\circled{\tiny$\mathsf{R}$}}}

\newcommand{\rep}[2]{{\ensuremath \left\langle #1 \mid #2 \right\rangle}}
\newcommand{\repthree}[3]{{\ensuremath \langle #1 \mid #2 \mid #3 \rangle}}

\newcommand{\fget}{\textsc{Get}}
\newcommand{\fput}{\textsc{Put}}

\newcommand{\freview}{\textsc{Review}}
\newcommand{\fcreate}{\textsc{Create}}
\newcommand{\fmatching}{\textsc{Matching}}
\newcommand{\funzip}{\textsc{Unzip}}
\newcommand{\fover}{\textsc{Over}}

\newcommand{\mget}{\textsc{MGet}}
\newcommand{\mput}{\textsc{MPut}}
\newcommand{\munzip}{\textsc{Munzip}}

\newcommand{\inside}{\mathsf{inside}}
\newcommand{\outside}{\mathsf{outside}}
\newcommand{\once}{\mathsf{once}}
\newcommand{\twice}{\mathsf{twice}}

\newcommand{\hto}{\ensuremath{\,\mathaccent\shortmid\rightarrow\,}}

\makeatletter
\providecommand{\leftsquigarrow}{%
  \mathrel{\mathpalette\reflect@squig\relax}%
}
\newcommand{\reflect@squig}[2]{%
  \reflectbox{$\m@th#1\rightsquigarrow$}%
}
\makeatother

\title{Categories of Optics}
\author{Mitchell Riley}
\affil{Wesleyan University \\ \texttt{mvriley@wesleyan.edu}}
\date{\vspace{-5ex}}
\begin{document}
\maketitle

\begin{abstract}
Bidirectional data accessors such as lenses, prisms and traversals are all instances of the same general `optic' construction. We give a careful account of this construction and show that it extends to a functor from the category of symmetric monoidal categories to itself. We also show that this construction enjoys a universal property: it freely adds counit morphisms to a symmetric monoidal category. Missing in the folklore is a general definition of `lawfulness' that applies directly to any optic category. We provide such a definition and show that it is equivalent to the folklore profunctor optic laws.
\end{abstract}

\setcounter{tocdepth}{1}
\setlength\cftparskip{-8pt}
\microtypesetup{protrusion=false}
\tableofcontents
\microtypesetup{protrusion=true}

\section{Introduction}

In its most concrete form, a \emph{lens} $S \hto A$ is a pair of maps ${\fget : S \to A}$ and ${\fput : S \times A \to S}$. From an engineering standpoint, such a lens allows us to ``zoom in'' on $S$ to focus on a small part $A$, manipulate $A$ in some way, then ``zoom out'' and have our changes reflected in $S$~\cite{CombinatorsForBidirectionalTreeTransformations}.

So that our lenses better adhere to this intuitive idea of ``zooming in'', we often want them satisfy some conditions known as the \emph{lens laws}:
\begin{center}
\begin{minipage}[b]{0.33333\textwidth}
\begin{center}
\[
\begin{tikzcd}
S \times A \ar[rr, "\fput"] \ar[dr, "\pi_2", swap] && S \ar[dl, "\fget"] \\
& A
\end{tikzcd}
\]
\hspace{0.8cm}$\fput\fget$
\end{center}
\end{minipage}%
\begin{minipage}[b]{0.33333\textwidth}
\begin{center}
\[
\begin{tikzcd}
S \ar[rr, "{[\id_S, \fget]}"] \ar[dr, "\id_S", swap] && S \times A \ar[dl, "\fput"] \\
& S
\end{tikzcd}
\]
\hspace{-0.6cm}$\fget\fput$
\end{center}
\end{minipage}%
\begin{minipage}[b]{0.33333\textwidth}
\begin{center}
\[
\begin{tikzcd}
S \times A \times A \ar[r, "\fput \times A"] \ar[d, "\pi_{1, 3}", swap] & S \times A\ar[d, "\fput"] \\
S \times A \ar[r, "\fput", swap] & A
\end{tikzcd}
\]
\quad$\fput\fput$
\end{center}
\end{minipage}%
\end{center}
We call such lenses lawful. The $\fput\fget$ law states that any update to $A$ is represented faithfully in $S$. The $\fget\fput$ law states that if $A$ is not changed then neither is $S$; and finally, the $\fput\fput$ law states that any update to $A$ completely overwrites previous updates.

Lenses form a category, with the composition of two lenses $(\fget_1, \fput_1) : T \hto S$ and $(\fget_2, \fput_2) : S \hto A$ as indicated:
\begin{align*}
\fget &: T \xrightarrow{\fget_1} S \xrightarrow{\fget_2} A \\
\fput &: T \times A \xrightarrow{[\id_T, \fget_1] \times A} T \times S \times A \xrightarrow{T \times \fput_2} T \times S \xrightarrow{\fput_1} T
\end{align*}
If the two input lenses are lawful then the composite is as well, so we find there is a subcategory of lawful lenses.

Lenses were discovered to be just one of a hierarchy of data accessors, including prisms, setters, traversals and more. These are collectively called \emph{optics} and have been best explored in the widely uesd Haskell \lenslib{} library: see~\cite{LensLibrary}. Each optic variant has a concrete description as a certain collection of maps, with attendant laws under which we consider them well-behaved, similar to the pair $(\fget, \fput)$ above and the lens laws. We begin in Section~\ref{sec:optics} by defining the \emph{category of optics} for a symmetric monoidal category in a sufficiently general way to encompass almost all the optic variants in use in the wild, using lenses as a running example. The category of lenses is precisely the result of this construction when applied to a symmetric monoidal category where the tensor is given by binary product. Section~\ref{sec:lawful-optics} defines the equivalent of the lens laws for a general category of optics. Then in Section~\ref{sec:examples} we see that these generic definitions specialise correctly to the other basic varieties of optic, including the laws. % We include in Section~\ref{sec:mixed-optics} a more speculative discussion of ``mixed optics'', which include indexed and coindexed lenses and traversals.

When implementing optics, library authors often use a form known as the \emph{profunctor encoding}, which at first glance is completely different to that given in Section~\ref{sec:optics}. (The Haskell \lenslib{} library itself actually uses a variant called the \emph{van Laarhoven encoding}, for reasons of efficiency and backwards compatibility.) As this was being written, Milewski~\cite{ProfunctorOpticsPost} and Boisseau and Gibbons~\cite{YouNeeda} independently described the isomorphism between optics and their profunctor encoding. In Section~\ref{sec:profunctor-optics} we review this isomorphism and verify that the folklore profunctor optic laws are equivalent to lawfulness as defined here.

More recently, concrete lenses have found use in compositional game theory~\cite{CompositionalGameTheory}. The $\fget$ function is thought of as mapping observations on the state of play to choices of what move to make. The $\fput$ function computes the utility of the moves that the players choose. There is interest in generalising this to a probabilistic setting, but it is not yet clear what the right replacement for concrete lenses is.

Much of what is known about optics is folklore, and careful verification of some of their categorical properties has been lacking, especially when working in categories other than $\Set$ (or $\Set$-like categories such as $\mathbf{Hask}$). The aim of the present paper is to fill this gap, with the hope that a better understanding of the general structure of these categories will make it easier to generalise optics to new and exotic settings. This is particularly important with the advent of linear types in Haskell, enabling a new branch of the lens family tree, and also with the new applications to game theory.

\subsection{Contributions}
  \begin{itemize}
  \item A careful account of the folklore optic construction in an arbitrary symmetric monoidal category $\C$, which we show extends to a functor $\Optic : \SymmMonCat \to \SymmMonCat$ (Section~\ref{sec:optics}),
  \item A universal property of the $\Optic$ construction as freely adding counits to a category of `dualisable morphisms' (Section~\ref{sec:teleological-categories}),
  \item A definition of lawfulness for a general optic category that specialises in the correct way to known cases and allows us to derive concrete laws for new kinds of optic (Section~\ref{sec:lawful-optics}),
  \item Commentary on the optic variants used most frequently in the wild (Section~\ref{sec:examples}),
  \item A proof that lawfulness as defined here is equivalent to the folklore profunctor optic laws (Section~\ref{sec:profunctor-optics}).
  \end{itemize}

\subsection{(Co)ends and Yoneda Reduction}

In this paper we will make frequent use of the (co)end calculus. For a comprehensive introduction to ends and coends, see~\cite{CoendCofriend}. We write $\copr_X : F(X, X) \to \int^{X \in \M} F(X, X)$ for the structure maps of a coend. The most important results for us regarding ends and coends are:

\begin{lemma}[Coend as coequaliser]\label{lemma:calculate-coend}
If $\E$ is cocomplete and $\M$ is small, the coend of $P : \M^\op \times \M \to \E$ can be calculated as the coequaliser in the diagram
\[
  \begin{tikzcd}
    \displaystyle \coprod_{M \to N} P(N, M) \ar[r,shift left=.75ex]  \ar[r,shift right=.75ex] & \displaystyle\coprod_{M \in \M} P(M, M) \ar[r] & \displaystyle\int^{M \in \M} P(M, M)
  \end{tikzcd}
\]
\qed
\end{lemma}

\begin{lemma}[Ninja Yoneda Lemma/Yoneda Reduction]\label{lem:yoneda-reduction}
For every functor $K : \C^\op \to \Set$ and $H : \C \to \Set$, we have the following natural isomorphisms:
\begin{align*}
KX &\cong \int^{C \in \C} KC \times \C(X,C) &
KX &\cong \int_{C \in \C} \Set(\C(C,X), KC) \\
HX &\cong \int^{C \in \C} HC \times \C(C,X)  &
HX &\cong \int_{C \in \C} \Set(\C(X,C), HC)
\end{align*}
where the isomorphisms are given by inclusion with the identity morphism $\C(X, X)$ for the left two, and evaluation at the identity morphism on the right.
\qed
\end{lemma}

\begin{theorem}[Fubini Theorem]
For a functor $F : \C^\op \times \C \times \D^\op \times \D \to \E$, there are canonical isomorphisms
\begin{align*}
\int^{C \in \C} \int^{D \in \D} F(C,C,D,D) \cong \int^{(C,D) \in \C \times \D} F(C,C,D,D) \cong \int^{D \in \D} \int^{C \in \C} F(C,C,D,D)
\end{align*}
\qed
\end{theorem}

\begin{lemma}[Mute coends]\label{lem:mute-coend}
Consider a functor $F : \C \to \E$ as a functor $\C^\op \times \C \to \E$ that ignores its contravariant argument. Then \[ \int^{C \in \C} F(C) \cong \colim F. \] \qed
\end{lemma}

\section{Optics}\label{sec:optics}

We begin by defining the category of optics for a symmetric monoidal category. This category was first defined in~\cite[Section 6]{Doubles} as the `double' of a monoidal category. There it was used for a completely different purpose---to investigate the relationship between Tambara modules and the `center' of a monoidal category. Our definition is almost identical, the only differences being that we have flipped the direction of the morphisms to match the existing work on lenses and restricted our attention to the unenriched setting.

Our definition of optic has as domain and codomain \emph{pairs} of objects of $\C$, one of which behaves covariantly and the other contravariantly. For example, our lenses will be pairs of maps $\fget : S \to A$ and $\fput : S \times A' \to S'$. This generality is important for the applications to game theory, and in fact helps in calculations by making the covariant and contravariant positions more difficult to confuse. Readers more familiar with lenses should ignore the primes.

  In this section we work with a fixed symmetric monoidal category $(\C, \otimes, I)$, with associator $\alpha$ and unitors $\lambda$ and $\rho$. To avoid getting lost in the notation we will use the standard cheat of omitting associativity morphisms and trust that the dedicated reader could insert them everywhere they are needed.

\begin{definition}
  Given two pairs of objects of $\C$, say $(S, S')$ and $(A, A')$, an \emph{optic} $p : (S, S') \hto (A, A')$ is an element of the set
  \begin{align*}
    \Optic_\C((S, S'), (A, A')) := \int^{M \in \C} \C(S, M \otimes A) \times \C(M \otimes A', S')
  \end{align*}
\end{definition}

Because this coend takes place in $\Set$, we can use Lemma~\ref{lemma:calculate-coend} to describe $\Optic_\C((S, S'), (A, A'))$ explicitly. It is the set of pairs $(l, r)$, where $l : S \to M \otimes A$ and $r : M \otimes A' \to S'$, quotiented by the equivalence relation generated by relations of the form
\begin{align*}
  ((f \otimes A) l, r) \sim (l, r (f \otimes A'))
\end{align*}
for any $l : S \to M \otimes A$, $r : N \otimes A' \to S'$ and $f : M \to N$.

For a pair of maps $l : S \to M \otimes A$ and $r : M \otimes A' \to S'$, we write $\rep{l}{r} : (S, S') \hto (A, A')$ for their image in $\Optic_\C((S, S'), (A, A'))$, and say that the object $M$ is the \emph{residual} for this representative. Optics will always be written with a crossed arrow $\hto$ to distinguish them from morphisms of $\C$.

  The residual $M$ should be thought of as a kind of `scratch space'; information from $S$ that we need to remember to construct $S'$. The quotienting imposed by the coend means we cannot inspect this temporary information, indeed, given an optic $S \hto A$ there is not even a canonical choice for the object $M$ in general.
    
  Elements of $\Optic_\C((S, S'), (A, A'))$ have an appealing
  interpretation as string diagrams with a ``hole'' missing. We draw the
  pair $\rep{l}{r}$ as
  
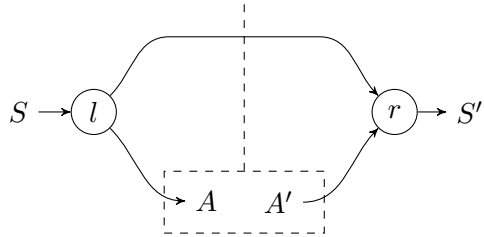
\begin{center}
    \begin{tikzpicture}
\node[vert] (l) at (0, 0) {$l$};
\node[vert] (r) at (4, 0) {$r$};

\node (S) [left of=l] {$S$};
\node (A) [below right = 0.7 and 1 of l] {$A$};
\node (S') [right of=r] {$S'$};
\node (A') [below left = 0.7 and 1 of r] {$A'$};

\draw[->] (S) -- (l);
\draw[->] (l) to[out=south east,in=west] (A);

\draw[<-] (S') -- (r);
\draw[<-] (r) to[out=south west,in=east] (A');

\draw[->] (l) to[out=north east, in=west] ++(1,1)
 to ++(2,0)
 to[out=east, in=north west] (r)
;

\node[draw,dashed,fit=(A) (A'), inner xsep = 8pt] (box) {};
\draw[dashed] (box.90) -- +(0,2.25);
\end{tikzpicture}
  \end{center}
  reading left to right, so the portion of the diagram to the left of the line represents $l$ and the right portion $r$. The relation expressed by the coend can be drawn graphically as:
  
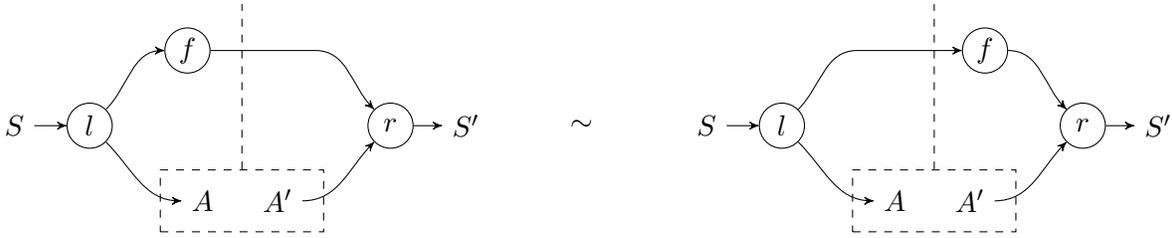
\begin{center}
    \begin{tikzpicture}
\begin{scope}[on grid]

\node[vert] (l) at (0, 0) {$l$};
\node[vert] (r) at (4, 0) {$r$};

\node (S) [left of=l] {$S$};
\node (A) [below right = 1 and 1.5 of l] {$A$};
\node (S') [right of=r] {$S'$};
\node (A') [below left = 1 and 1.5 of r] {$A'$};

\draw[->] (S) -- (l);
\draw[->] (l) to[out=south east,in=west] (A);

\draw[<-] (S') -- (r);
\draw[<-] (r) to[out=south west,in=east] (A');

\node[vert, above right = 1 and 1.3 of l] (f) {$f$};
\draw[->] (l) to[out=north east, in=west] (f);
\draw[->] (f)
 to[out=east, in=west] +(1.7,0)
 to[out=east, in=north west] (r)
;

\node[draw,dashed,fit=(A) (A'), inner xsep = 8pt] (box) {};
\draw[dashed] (box.90) -- +(0,2.2);

\end{scope}
\end{tikzpicture}
    \hspace{0.7cm} \raisebox{1.35cm}{$\sim$} \hspace{1cm}
    \begin{tikzpicture}
\begin{scope}[on grid]

\node[vert] (l) at (0, 0) {$l$};
\node[vert] (r) at (4, 0) {$r$};

\node (S) [left of=l] {$S$};
\node (A) [below right = 1 and 1.5 of l] {$A$};
\node (S') [right of=r] {$S'$};
\node (A') [below left = 1 and 1.5 of r] {$A'$};

\draw[->] (S) -- (l);
\draw[->] (l) to[out=south east,in=west] (A);

\draw[<-] (S') -- (r);
\draw[<-] (r) to[out=south west,in=east] (A');

\node[vert, above left = 1 and 1.3 of r] (f) {$f$};
\draw[->] (l) 
 to[out=north east, in=west] +(1,1)
 to[out=east, in=west] (f)
;
\draw[->] (f)
 to[out=east, in=north west] (r)
;

\node[draw,dashed,fit=(A) (A'), inner xsep = 8pt] (box) {};
\draw[dashed] (box.90) -- +(0,2.2);

\end{scope}
\end{tikzpicture}
  \end{center}
  We will therefore omit the vertical cut between $l$ and $r$ in most subsequent diagrams; any choice yields a representative of same optic.

A common use of the coend relation is to introduce or cancel isomorphisms. Given $l : S \to M \otimes A$ and $r : M \otimes A \to S$, for any isomorphism $f : M \to N$ we have
\begin{align*}
\rep{l}{r} = \rep{(f^{-1} \otimes A)(f \otimes A)l}{r} = \rep{(f \otimes A)l}{r(f^{-1} \otimes A)}
\end{align*}
Diagrammatically, this is the equality
  
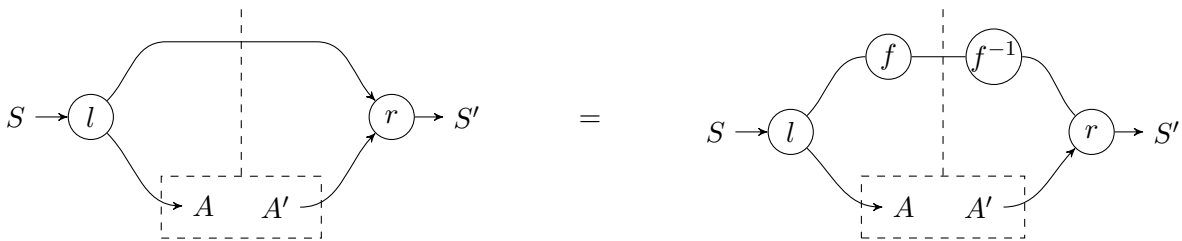
\begin{center}
    \begin{tikzpicture}
\node[vert] (l) at (0, 0) {$l$};
\node[vert] (r) at (4, 0) {$r$};

\node (S) [left of=l] {$S$};
\node (A) [below right = 0.7 and 1 of l] {$A$};
\node (S') [right of=r] {$S'$};
\node (A') [below left = 0.7 and 1 of r] {$A'$};

\draw[->] (S) -- (l);
\draw[->] (l) to[out=south east,in=west] (A);

\draw[<-] (S') -- (r);
\draw[<-] (r) to[out=south west,in=east] (A');

\draw[->] (l) to[out=north east, in=west] ++(1,1)
 to ++(2,0)
 to[out=east, in=north west] (r)
;

\node[draw,dashed,fit=(A) (A'), inner xsep = 8pt] (box) {};
\draw[dashed] (box.90) -- +(0,2.25);
\end{tikzpicture}
    \hspace{0.8cm} \raisebox{1.5cm}{$=$} \hspace{1cm}
    \begin{tikzpicture}
\begin{scope}[on grid]
\node[vert] (l) at (0, 0) {$l$};
\node[vert] (r) at (4, 0) {$r$};

\node (S) [left of=l] {$S$};
\node (A) [below right = 1 and 1.5 of l] {$A$};
\node (S') [right of=r] {$S'$};
\node (A') [below left = 1 and 1.5 of r] {$A'$};

\node[vert] (f) [above right = 1 and 1.3 of l] {$f$};
\node[vert] (fi) [above left = 1 and 1.3 of r] {$f^{-1}$};

\draw[->] (S) -- (l);
\draw[->] (l) to[out=south east,in=west] (A);

\draw[<-] (S') -- (r);
\draw[<-] (r) to[out=south west,in=east] (A');

\draw (l) to[out=north east, in=west] (f);
\draw (f) to[out=east, in=west] (fi);
\draw (fi) to[out=east, in=north west] (r);

\node[draw,dashed,fit=(A) (A'), inner xsep = 8pt] (box) {};
\draw[dashed] (box.90) -- +(0,2.25);
\end{scope}
\end{tikzpicture}
  \end{center}
  
  \begin{example}
  ~\begin{enumerate}[(1)]
  \item  For any three objects $M, A, A' \in \C$, there is the \emph{tautological} optic \[t_{M,A,A'} : (M \otimes A, M \otimes A') \hto (A, A')\] given by $\rep{\id_{M \otimes A}}{\id_{M \otimes A'}}$.
  This would be drawn as follows:
  
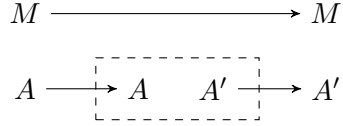
\begin{center}
  \begin{tikzpicture}
\begin{scope}[on grid]
\node (M) at (0, 1) {$M$};
\node (M2) at (4, 1) {$M$};
\node (A) at (0, 0) {$A$};
\node (A2) at (1.5, 0) {$A$};
\node (A') at (2.5, 0) {$A'$};
\node (A'2) at (4, 0) {$A'$};

\draw[->] (M) -- (M2);
\draw[->] (A) -- (A2);
\draw[->] (A') -- (A'2);

\node[draw,dashed,fit=(A2) (A'), inner xsep = 8pt] (box) {};
\end{scope}
\end{tikzpicture}
  \end{center}
  
  \item We also have the \emph{identity} optic $\id_{(S, S')} : (S, S') \hto (S, S')$, given by $\rep{\lambda^{-1}_S}{\lambda_{S'}}$, where $\lambda_S : I \otimes S \to S$ is the left unitor for $S$ and similarly for $S'$.

The identity optic is drawn as
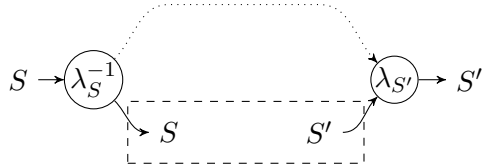
\begin{center}
  \begin{tikzpicture}
\begin{scope}[on grid]
\node[vert] (l) at (0, 0) {$\lambda_S^{-1}$};
\node[vert] (r) at (4, 0) {$\lambda_{S'}$};

\node (S) [left of=l] {$S$};
\node (A) [below right = 0.7 and 1 of l] {$S$};
\node (S') [right of=r] {$S'$};
\node (A') [below left = 0.7 and 1 of r] {$S'$};

\draw[->] (S) -- (l);
\draw[->] (l) to[out=south east,in=west] (A);

\draw[<-] (S') -- (r);
\draw[<-] (r) to[out=south west,in=east] (A');

\draw[->, dotted] (l) to[out=north east, in=west] ++(1,1)
 to ++(2,0)
 to[out=east, in=north west] (r)
;

\node[draw,dashed,fit=(A) (A'), inner xsep = 8pt] (box) {};
\end{scope}
\end{tikzpicture}
\end{center}
This dashed line above the diagram represents the unit object. It is common in string diagrams to omit unitors and the unit object unless they are necessary to make sense of the diagram. We therefore prefer to draw the identity morphism as:
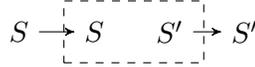
\begin{center}
  \begin{tikzpicture}
\node (Sin) {$S$};
\node (Sout) [right of=Sin] {$S$};
\node (Spout) [right of=Sout] {$S'$};
\node (Spin) [right of=Spout] {$S'$};

\draw[->] (Sin) -- (Sout);
\draw[->] (Spout) -- (Spin);

\node[draw,dashed,fit=(Sout) (Spout), inner xsep = 4pt] (box) {};
\end{tikzpicture}
\end{center}
  \end{enumerate}
\end{example}

Optics compose as follows. The easiest interpretation is graphical: composition corresponds to substituting the first optic for the hole of the second:
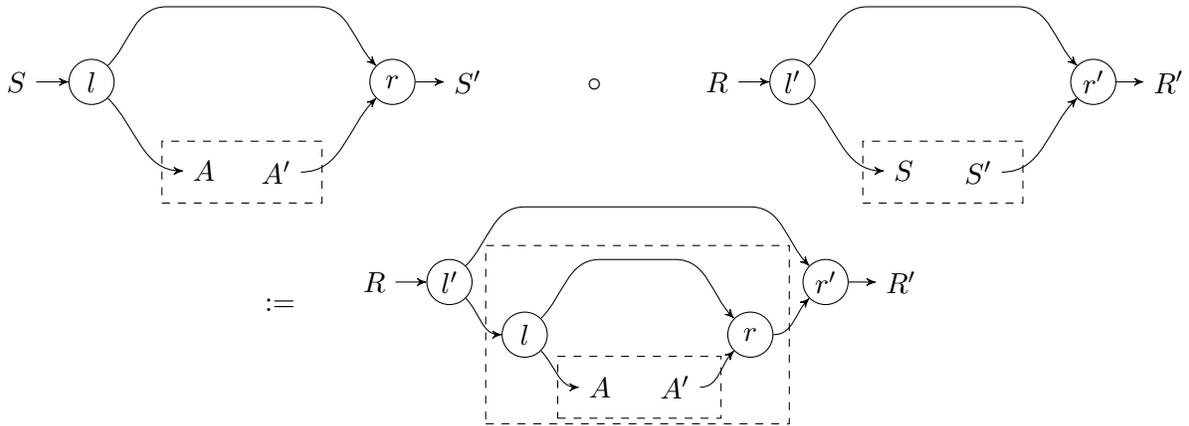
\begin{center}
  \begin{tikzpicture}
\node[vert] (l) at (0, 0) {$l$};
\node[vert] (r) at (4, 0) {$r$};

\node (S) [left of=l] {$S$};
\node (A) [below right = 0.7 and 1 of l] {$A$};
\node (S') [right of=r] {$S'$};
\node (A') [below left = 0.7 and 1 of r] {$A'$};

\draw[->] (S) -- (l);
\draw[->] (l) to[out=south east,in=west] (A);

\draw[<-] (S') -- (r);
\draw[<-] (r) to[out=south west,in=east] (A');

\draw[->] (l) to[out=north east, in=west] ++(1,1)
 to ++(2,0)
 to[out=east, in=north west] (r)
;

\node[draw,dashed,fit=(A) (A'), inner xsep = 8pt] (box) {};
\end{tikzpicture}
  \hspace{0.9cm} \raisebox{1.5cm}{$\circ$} \hspace{1cm}
  \begin{tikzpicture}
\node[vert] (l) at (0, 0) {$l'$};
\node[vert] (r) at (4, 0) {$r'$};

\node (R) [left of=l] {$R$};
\node (S) [below right = 0.7 and 1 of l] {$S$};
\node (R') [right of=r] {$R'$};
\node (S') [below left = 0.7 and 1 of r] {$S'$};

\draw[->] (R) -- (l);
\draw[->] (l) to[out=south east,in=west] (S);

\draw[<-] (R') -- (r);
\draw[<-] (r) to[out=south west,in=east] (S');

\draw[->] (l) to[out=north east, in=west] ++(1,1)
 to ++(2,0)
 to[out=east, in=north west] (r)
;

\node[draw,dashed,fit=(S) (S'), inner xsep = 8pt] (box) {};
\end{tikzpicture} \\
  \raisebox{1.5cm}{$:=$}\qquad
  \begin{tikzpicture}
\begin{scope}[on grid]

\node[vert] (l') at (0, 0) {$l'$};
\node[vert, below right = 0.7 and 1 of l'] (l) {$l$};
\node[vert] (r') at (5, 0) {$r'$};
\node[vert, below left = 0.7 and 1 of r'] (r) {$r$};

\node (A) [below right = 0.7 and 1 of l] {$A$};
\node (A') [below left = 0.7 and 1 of r] {$A'$};

\node (R) [left of=l'] {$R$};
\node (R') [right of=r'] {$R'$};

\draw[->] (R) -- (l');
\draw[<-] (R') -- (r');

\draw[->] (l') to[out=north east, in=west] ++(1,1)
 to ++(3,0)
 to[out=east, in=north west] (r')
;

\draw[->] (l) to[out=north east, in=west] ++(1,1)
 to ++(1,0)
 to[out=east, in=north west] (r)
;

\draw[->] (l') to[out=south east,in=west] (l);
\draw[->] (r) to[out=east, in=south west] (r');

\draw[->] (l) to[out=south east,in=west] (A);
\draw[<-] (r) to[out=south west,in=east] (A');

\node[draw,dashed,fit=(l) (r), inner xsep = 6pt, inner ysep = 25pt] (box2) {};
\node[draw,dashed,fit=(A) (A'), inner xsep = 8pt] (box) {};
\end{scope}
\end{tikzpicture}
\end{center}

More formally, we wish to construct a map
\begin{align*}
  &\left(\int^{M \in \C} \C(S, M \otimes A) \times \C(M \otimes A', S')\right) \times \left(\int^{N \in \C} \C(R, N \otimes S) \times \C(N \otimes S', R')\right) \\
  &\quad \to \int^{M \in \C} \C(R, M \otimes A) \times \C(M \otimes A', R').
\end{align*}
The product in $\Set$ preserves colimits, so in particular coends. Using this fact and the Fubini theorem for coends, the domain is isomorphic to
\begin{align*}
  \int^{(M, N) \in \C \times \C} \C(S, M \otimes A) \times \C(M \otimes A', S') \times \C(R, N \otimes S) \times \C(N \otimes S', R').
\end{align*}
So by the universal property of coends, it suffices to construct maps
\begin{align*}
  & \C(S, M \otimes A) \times \C(M \otimes A', S') \times \C(R, N \otimes S) \times \C(N \otimes S', R') \\ &
                                                                                                              \quad \to \int^{M \in \C} \C(R, M \otimes A) \times \C(M \otimes A', R').
\end{align*}
natural in $M$ and $N$. For these we use the composites
\begin{align*}
  &\C(S, M \otimes A) \times \C(M \otimes A', S') \times \C(R, N \otimes S) \times \C(N \otimes S', R')\\
  \to \,& \C(N \otimes S, N \otimes M \otimes A) \times \C(N \otimes M \otimes A', N \otimes S') \times \C(R, N \otimes S) \times \C(N \otimes S', R') && \text{(functoriality of $N \otimes  -$)} \\
  \to \,& \C(R, N \otimes  M \otimes A) \times \C(N \otimes M \otimes A', R') && \text{(composition in $\C$)} \\
  \to \,&\int^{P \in \M} \C(R, P \otimes A) \times \C(P \otimes A', R') && \text{($\copr_{N \otimes M}$)}
\end{align*}
Written equationally, suppose $\rep{l'}{r'} : (R, R') \hto (S, S')$ and $\rep{l}{r} : (S, S') \hto (A, A')$ are optics with $M$ the residual for $\rep{l'}{r'}$. The composite $(R, R') \hto (A, A')$ is then: \[\rep{l}{r} \circ \rep{l'}{r'} := \rep{(M \otimes l)l'}{r'(M \otimes r)}.\]

\begin{proposition}\label{prop:optic-is-cat}
  The above data form a category $\Optic_\C$.
\end{proposition}
\begin{proof}
  In~\cite[Section 6]{Doubles} this is proven abstractly by exhibiting this category as the Kleisli category for a monad in the bicategory $\Prof$. We prefer a direct proof.
  
  Suppose we have representatives of three optics
\begin{align*}
  \rep{l_1}{r_1} &: (R, R) \hto (S, S') \\
  \rep{l_2}{r_2} &: (S, S') \hto (A, A') \\
  \rep{l_3}{r_3} &: (A, A') \hto (B, B'),
\end{align*}
that have residuals $M$, $N$ and $P$ respectively. We must choose these representatives simultaneously but, as in the definition of composition, this is allowed by the Fubini theorem. Then:
  \begin{align*}
    (\rep{l_3}{r_3} \circ \rep{l_2}{r_2}) \circ \rep{l_1}{r_1}
    &= \rep{(N \otimes l_3)l_2}{r_2(N \otimes r_3)} \circ \rep{l_1}{r_1} \\
    &= \rep{(M \otimes ((N \otimes l_3)l_2))l_1}{r_1(M \otimes (r_2(N \otimes r_3)))} \\
    &= \rep{(M \otimes N \otimes l_3)(M \otimes l_2)l_1}{r_1(M \otimes r_2)(M \otimes N \otimes r_3)} \\
    &= \rep{l_3}{r_3} \circ (\rep{(M \otimes l_2)l_1}{r_1(M \otimes r_2)}) \\
    &= \rep{l_3}{r_3} \circ (\rep{l_2}{r_2} \circ \rep{l_1}{r_1})
  \end{align*}

  For the unit laws, suppose we have $\rep{l}{r} : (S, S') \hto (A, A')$ with representative $M$. We calculate:
  \begin{align*}
    \id_{A, A'} \circ \rep{l}{r}
    &= \rep{\lambda^{-1}_A}{\lambda_{A'}} \circ \rep{l}{r} \\
    &= \rep{(M \otimes \lambda^{-1}_A) l}{r (M \otimes  \lambda_{A'})} \\
    &= \rep{(\rho^{-1}_M \otimes  A) l}{r (\rho_M \otimes A')} \\
    &= \rep{l}{r (\rho_M \otimes A') (\rho^{-1}_M \otimes A')} \\
    &= \rep{l}{r} \\
    \rep{l}{r} \circ \id_{S, S'}
    &= \rep{l}{r} \circ \rep{\lambda^{-1}_S}{\lambda_{S'}}  \\
    &= \rep{(I \otimes l)\lambda^{-1}_S}{\lambda_{S'} (I \otimes r)} \\
    &= \rep{(\lambda^{-1}_M \otimes S)l}{r (\lambda_{M} \otimes S')} \\
    &= \rep{l}{r (\lambda_{M} \otimes S')(\lambda^{-1}_M \otimes S')} \\
    &= \rep{l}{r}
  \end{align*}
  In both cases we have used the coend relation to cancel an isomorphism appearing on both sides of an optic.
\end{proof}

  Note that the homsets of $\Optic_\C$ are given by a coend indexed by a possibly large category. If $\C$ is small then these coends always exist, but if $\C$ is not small their existence is not guaranteed by the cocompleteness of $\Set$. Because of this we should be careful to only discuss optic categories where we know that the coends exist by some other means, e.g., by exhibiting an isomorphism of $\Optic_\C((S, S'), (A, A'))$ with a set. For all of the examples we give later we provide such a isomorphism.

\begin{proposition}
If $\C$ is a category with finite products, then $\Lens \defeq \Optic_\C$ is the category of lenses described in the introduction (so long as we restrict to optics of shape $(S, S) \hto (A, A)$). 
\end{proposition}
\begin{proof}
We see that optics correspond to pairs of $\fget$ and $\fput$ functions via the following isomorphisms:
\begin{align*}
  \Lens((S, S'), (A, A'))
  &= \int^{M \in \C} \C(S, M \times A) \times \C(M \times A', S') \\
  &\cong \int^{M \in \C} \C(S, M) \times \C(S, A) \times \C(M \times A', S') && \text{(universal property of product)} \\
  &\cong \C(S, A) \times \C(S \times A', S') && \text{(Yoneda reduction)}
\end{align*}
This last step deserves some explanation. We are applying the isomorphism $KX \cong \int^{C \in \C} \C(X,C) \times KC$ of Lemma~\ref{lem:yoneda-reduction} to the case $X = S$ and $K = \C(S, A) \times \C(- \times A', S')$.

Explicitly the isomorphism states that, given an optic $\rep{l}{r} : (S, S') \hto (A, A')$, the corresponding concrete lens is the pair $\fget : S \to A$ and $\fput : S \times A' \to S'$, where $\fget = \pi_2 l$ and $\fput = r (\pi_1 l \times A)$. In the other direction, given $(\fget, \fput)$, the corresponding optic is represented by $\rep{[\id_S, \fget]}{\fput}$.

We leave it to the reader to verify that composition in $\Lens$ corresponds to ordinary composition of concrete lenses by using this isomorphism in both directions. (Of course, there is only one sensible way to compose such a collection of morphisms!)
\end{proof}

%The remainder of this section comprises some useful observations that were not made in~\cite{Doubles}.

\begin{proposition}\label{prop:iota-functor}
  There is a functor $\iota : \C \times \C^\op \to \Optic_\C$, which on objects is given by $\iota(S, S') = (S, S')$ and on morphisms $(f, g) : (S, S') \to (A, A')$ by $\iota(f, g) = \rep{\lambda_A^{-1} f}{g \lambda_{A'}}$.
\end{proposition}
\begin{proof}
  Graphically, this is:
  \begin{center}
    \begin{tikzpicture}
\node (Sin) {$S$};
\node (f) [vert, right of=Sin] {$f$};
\node (Sout) [right of=f] {$S$};
\node (Spout) [right of=Sout] {$S'$};
\node (g) [vert, right = 0.5 of Spout] {$g$};
\node (Spin) [right of=g] {$S'$};

\draw[->] (Sin) -- (f);
\draw[->] (f) -- (Sout);
\draw[->] (Spout) -- (g);
\draw[->] (g) -- (Spin);

\node[draw,dashed,fit=(Sout) (Spout)] (box) {};
\end{tikzpicture}
  \end{center}

  This preserves identities, as the identity on an object $(S, S')$ in $\Optic_\C$ is defined to be exactly $\rep{\lambda^{-1}_S}{\lambda_{S'}}$. To check functoriality, suppose we have $(f, g) : (S, S') \to (A, A')$ and $(f', g') : (A, A') \to (B, B')$ in $\C \times \C^\op$. Then:
  \begin{align*}
    \iota(f', g') \circ \iota(f, g)
    &= \rep{\lambda^{-1}_B f'}{g' \lambda_{B'}} \circ \rep{\lambda^{-1}_A f}{g \lambda_{A'}} \\
    &= \rep{(I\otimes (\lambda^{-1}_B f'))\lambda^{-1}_A f}{g \lambda_{A'} (I\otimes (g' \lambda_{B'}))} && \text{(By definition of $\circ$)}\\
    &= \rep{(I \otimes \lambda^{-1}_B) (I \otimes f')\lambda^{-1}_A f}{g \lambda_{A'} (I \otimes g')(I\otimes \lambda_{B'})} && \text{(Functoriality of $I \otimes -$)}\\
    &= \rep{(I\otimes \lambda^{-1}_B) \lambda^{-1}_B f' f}{g g' \lambda_{B'} (I\otimes \lambda_{B'})} && \text{(Naturality of $\lambda$)}\\
    &= \rep{(\lambda^{-1}_I \otimes B) \lambda^{-1}_B f' f}{g g' \lambda_{B'} (\lambda_I \otimes B')} && \text{(Unitality of action)} \\
    &= \rep{\lambda^{-1}_B f' f}{g g' \lambda_{B'} (\lambda_I \otimes B') (\lambda^{-1}_I \otimes B')} && \text{(Coend relation)}  \\
    &= \rep{\lambda^{-1}_B f'f}{g g' \lambda_{B'}} \\
    &= \iota(f'f, gg')
  \end{align*}
  Graphically, there is not much to do:
  \begin{center}
    \begin{tikzpicture}
\node (Sin) {$S$};
\node (f) [vert, right of=Sin] {$f$};
\node (f') [vert, right of=f] {$f'$};
\node (Sout) [right of=f'] {$B$};
\node (Spout) [right of=Sout] {$B'$};
\node (g') [vert, right = 0.5 of Spout] {$g'$};
\node (g) [vert, right of=g'] {$g$};
\node (Spin) [right of=g] {$S'$};

\draw[->] (Sin) -- (f);
\draw[->] (f) -- (f');
\draw[->] (f') -- (Sout);
\draw[->] (Spout) -- (g');
\draw[->] (g') -- (g);
\draw[->] (g) -- (Spin);

\node[draw,dashed,fit=(Sout) (Spout)] (box) {};
\end{tikzpicture}
    \qquad \raisebox{0.3cm}{$=$} \qquad
    \begin{tikzpicture}
\node (Sin) {$S$};
\node (f) [vertbig, right of=Sin] {$f'f$};
\node (Sout) [right of=f] {$B$};
\node (Spout) [right of=Sout] {$B'$};
\node (g) [vertbig, right = 0.5 of Spout] {$gg'$};
\node (Spin) [right of=g] {$S'$};

\draw[->] (Sin) -- (f);
\draw[->] (f) -- (Sout);
\draw[->] (Spout) -- (g);
\draw[->] (g) -- (Spin);

\node[draw,dashed,fit=(Sout) (Spout)] (box) {};
\end{tikzpicture}
  \end{center}
\end{proof}
%  This functor is not necessarily faithful, see Remark~\ref{lens-iota-not-faithful}.

There are some other easy-to-construct optics; specifically, optics out of  and into the monoidal unit $(I, I)$. Such maps in a monoidal category are sometimes called states and costates~\cite{CategoricalQuantumMechanics}.
\begin{proposition}\label{prop:costates}
  The set of costates $(S, S') \hto (I, I)$ is isomorphic to $\C(S, S')$.
\end{proposition}
\begin{proof}
  \begin{align*}
    \Optic_\C((S, S'), (I, I))
    &= \int^{M \in \C} \C(S, M \otimes I) \times \C(M \otimes I, S') \\
    &\cong \int^{M \in \C} \C(S, M) \times \C(M, S') \\
    &\cong \C(S, S')
  \end{align*}
  by Yoneda reduction, so a state $\rep{l}{r} : (S, S') \hto (I, I)$ corresponds to the morphism $rl : S \to S'$, and a morphism $f : S \to S'$ corresponds to the state $\rep{\rho_S^{-1}}{f \rho_S} : (S, S') \hto (I, I)$
\end{proof}

In particular, for any $S \in \C$, the identity $\id_S$ yields an optic $c_S = \rep{\rho_S^{-1}}{\rho_S} : (S, S) \hto (I, I)$ that we call the \emph{connector}:
%\begin{center}
%  \input{diagrams/connector-full.tikz}
%\end{center}
%Or, again omitting the unitors:
\begin{center}
  \begin{tikzpicture}
\begin{scope}[on grid]
\node (S) at (0, 0) {$S$};
\node (S') at (4, 0) {$S$};

\draw[->] (S) -- (S');

\node (I) [below right = 0.7 and 1.4 of S] {$I$};
\node (I') [below left = 0.7 and 1.4 of S'] {$I$};

\node[draw,dashed,fit=(I) (I'), inner xsep = 4pt] (box) {};
\end{scope}
\end{tikzpicture}
\end{center}

\begin{proposition}\label{prop:states}
  Suppose the monoidal unit $I$ of $\C$ is terminal. Then the set of states $(I, I) \hto (A, A')$ is isomorphic to $\C(I, A)$.
\end{proposition}
\begin{proof}
  First, note that
  \begin{align*}
    \Optic_\C((I,I), (A,A'))
    &= \int^{M \in \C} \C(I, M \otimes A) \times \C(M \otimes A', I) \\
    &\cong \int^{M \in \C} \C(I, M \otimes A).
  \end{align*}
  as $I$ is terminal. The interior of this coend is mute in the contravariant position, so the coend is equal to the colimit of the functor $\C(I, - \otimes A) : \C \to \Set$ by Lemma~\ref{lem:mute-coend}. But $\C$ has terminal object $I$, so 
\begin{align*}
  \int^{M \in \C} \C(I, M \otimes A) 
  &\cong \colim \C(I, - \otimes A) \\
  &\cong \C(I, I \otimes A) \\
  &\cong \C(I, A)
\end{align*}

  Explicitly, a state $f : I \to A$ in $\C$ corresponds to the optic $\rep{\lambda_A^{-1} f}{!_{I \times A'}} : (I, I) \hto (A, A')$, where $!_{I \times A'} : I \times A' \to I$ is the unique map.
\end{proof}

The remainder of this section comprises a proof of the following fact:

\begin{theorem}\label{thm:optic-functor}
  The $\Optic_\C$ construction extends to a functor \[\Optic : \SymmMonCat \to \SymmMonCat,\] where $\SymmMonCat$ denotes the (1-)category of (small) symmetric monoidal categories and strong symmetric monoidal functors.
\end{theorem}

\begin{proposition}\label{prop:change-of-action-monoidal}
  A monoidal functor $F : \C \to \D$ induces a functor $\Optic(F) : \Optic_\C \to \Optic_\D$, given on objects by $\Optic(F)(S, S') = (FS, FS')$ and on morphisms $\rep{l}{r} : (S, S') \to (A, A')$ by
  \begin{align*}
    \Optic(F)(\rep{l}{r}) := \rep{\phi^{-1}_{M,A} (Fl)}{(Fr) \phi_{M,A'}},
  \end{align*}
  where $\phi_{M,A} : FM \otimes FA \to F(M \otimes A)$ and $\phi_I : I \to FI$ denote the structure maps of the monoidal functor. Graphically:
  
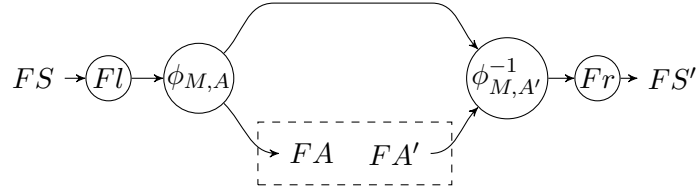
\begin{center}
    \begin{tikzpicture}
\begin{scope}[on grid]
\node[vert] (l) at (0, 0) {$Fl$};
\node[vert] (r) at (6.5, 0) {$Fr$};

\node (S) [left of=l] {$FS$};
\node[vert] (phi) [right = 1.2 of l] {$\phi_{M,A}$};
\node (A) [below right = 1 and 1.5 of phi] {$FA$};
\node (S') [right of=r] {$FS'$};
\node[vert] (phii) [left = 1.2 of r] {$\phi^{-1}_{M, A'}$};
\node (A') [below left = 1 and 1.5 of phii] {$FA'$};

\draw[->] (S) -- (l);
\draw[->] (l) -- (phi);
\draw[->] (phi) to[out=south east,in=west] (A);

\draw[<-] (S') -- (r);
\draw[<-] (r) -- (phii);
\draw[<-] (phii) to[out=south west,in=east] (A');

\draw[->] (phi) to[out=north east, in=west] ++(1,1)
 to ($(phii) + (-1, 1)$)
 to[out=east, in=north west] (phii)
;

\node[draw,dashed,fit=(A) (A'), inner xsep = 8pt] (box) {};
\end{scope}
\end{tikzpicture}
  \end{center}
\end{proposition}
\begin{proof}

  This preserves identities:
  \begin{align*}
  &\Optic(F)(\id_{(S, S')}) \\
  &= \Optic(F)(\rep{\lambda^{-1}_S}{\lambda_{S'}}) &&\text{(Definition of $\id$)} \\
  &= \rep{\phi^{-1}_{I,S} (F\lambda^{-1}_S)}{(F\lambda_{S'}) \phi_{I,S'}} && \text{(Definition of $\Optic(F)$)} \\
  &= \rep{(\phi_I^{-1} \otimes S) \phi^{-1}_{I,S} (F\lambda^{-1}_S)}{(F\lambda_{S'}) \phi_{I,S'}(\phi_I \otimes S) } && \text{(Introducing isomorphism to both sides)} \\
  &= \rep{\lambda^{-1}_{FS}}{\lambda_{FS'}} &&\text{($F$ is a monoidal functor)} \\
  &= \id_{(FS, FS')}
  \end{align*}
  And given two optics $\rep{l}{r} : (S, S') \hto (A, A')$ and $\rep{l'}{r'} : (R, R') \hto (S, S')$ with residuals $M$ and $M'$, it preserves composition:
\begingroup
\allowdisplaybreaks
\begin{align*}
&\Optic(F)(\rep{l}{r} \circ \rep{l'}{r'})  \\
&\qquad \text{(Definition of $\circ$)} \\
&= \Optic(F)(\rep{(M' \otimes l)l'}{r'(M' \otimes r)}) \\
&\qquad \text{(Definition of $\Optic(F)$)} \\
&= \rep{\phi^{-1}_{M' \otimes M,A} F((M' \otimes l)l')}{F(r'(M' \otimes r)) \phi_{M' \otimes M,A'}} \\
&\qquad \text{(Functoriality of $F$)} \\
&= \rep{\phi^{-1}_{M' \otimes M,A} F(M' \otimes l)(Fl')}{(Fr') F(M' \otimes r) \phi_{M' \otimes M,A'}} \\
&\qquad \text{(Introducing isomorphism to both sides)} \\
&= \rep{(\phi^{-1}_{M', M} \otimes A)\phi^{-1}_{M' \otimes M,A} F(M' \otimes l)(Fl')}{(Fr') F(M' \otimes r) \phi_{M' \otimes M,A'}(\phi_{M', M} \otimes A)} \\
&\qquad \text{(Hexagon axiom for $F$)} \\
&= \rep{(FM' \otimes \phi^{-1}_{M,A})\phi^{-1}_{M',M \otimes A}(F(M' \otimes l)) (Fl')}{(Fr') (F(M' \otimes r)) \phi_{M',M \otimes A'}  (FM' \otimes \phi_{M,A'})} \\
&\qquad \text{(Naturality of $\phi$)} \\
&= \rep{(FM' \otimes \phi^{-1}_{M,A})(FM' \otimes Fl)\phi^{-1}_{M',S} (Fl')}{(Fr') \phi_{M',S'} (FM' \otimes Fr) (FM' \otimes \phi_{M,A'})} \\
&\qquad \text{(Functoriality of $\otimes$)} \\
&= \rep{(FM' \otimes \phi^{-1}_{M,A} (Fl))(\phi^{-1}_{M',S} (Fl'))}{((Fr') \phi_{M',S'})(FM' \otimes (Fr) \phi_{M,A'})} \\
&\qquad \text{(Definition of $\circ$)} \\
&= \rep{\phi^{-1}_{M,A} (Fl)}{(Fr) \phi_{M,A'}} \circ \rep{\phi^{-1}_{M',S} (Fl')}{(Fr') \phi_{M',S'}} \\
&\qquad \text{(Definition of $\Optic(F)$)} \\
&= \Optic(F)(\rep{l}{r}) \circ \Optic(F)(\rep{l'}{r'})
\end{align*}
\endgroup
The critical move is adding the isomorphism $(\phi_{M', M} \otimes A)$ to both sides of the coend, so that the hexagon axiom for $F$ may be applied.
\end{proof}

\begin{lemma}\label{lem:iota-commute-with-opticf}
$\iota$ commutes with $\Optic(F)$, in the sense that
\[ \Optic(F)(\iota(f, g)) = \iota(Ff, Fg) \]
\end{lemma}
\begin{proof}
This is a straightforward calculation:
  \begin{align*}
    & \Optic(F)(\iota(f, g)) \\
    &\qquad \text{(Definition of $\iota$)} \\    
    &= \Optic(F)(\rep{\lambda_A^{-1} f}{g \lambda_{A'}}) \\
    &\qquad \text{(Definition of $\Optic(F)$)} \\    
    &= \rep{\phi^{-1}_{I,A} (F(\lambda_A^{-1} f))}{(F(g \lambda_{A'})) \phi_{I,A'}} \\
    &\qquad \text{(Functoriality of $F$)} \\    
    &= \rep{\phi^{-1}_{I,A} (F\lambda_A^{-1}) (Ff)}{(Fg)(F \lambda_{A'}) \phi_{I,A'}} \\
    &\qquad \text{(Introducing $\phi_I$ to both sides)} \\    
    &= \rep{(\phi^{-1}_I \otimes FA) \phi^{-1}_{I,A} (F\lambda_A^{-1}) (Ff)}{(Fg)(F \lambda_{A'}) \phi_{I,A'} (\phi_I \otimes FA)} \\
    &\qquad \text{($F$ is a monoidal functor)} \\    
    &= \rep{\lambda_{FA}^{-1} (Ff)}{(Fg)\lambda_{FA'}} \\
    &\qquad \text{(Definition of $\iota$)} \\    
    &= \iota(Ff, Fg)
  \end{align*}
\end{proof}

\begin{proposition}\label{prop:iota-naturality}
$\iota : \C \times \C^\op \to \Optic_\C$ ``lifts natural isomorphisms'', in the following sense. Given monoidal functors $F, G : \C \to \D$ and a monoidal natural isomorphism $\alpha : F \Rightarrow G$, there is an induced natural isomorphism $\Optic(\alpha) : \Optic(F) \Rightarrow \Optic(G)$ with components:
\begin{align*}
{\Optic(\alpha)}_{(S, S')} &: (FS, FS') \to (GS, GS') \\
{\Optic(\alpha)}_{(S, S')} &:= \iota(\alpha_{S}, \alpha^{-1}_{S'})
\end{align*}
\end{proposition}
\begin{proof}
Suppose $\phi$ and $\psi$ are the structure maps for $F$ and $G$ respectively. We just have to show naturality, i.e.\ that for $p : (S, S') \hto (A, A')$ in $\D$, the equation \[\Optic(\alpha)_{(A, A')} \circ \Optic(F)(p) = \Optic(G)(p) \circ \Optic(\alpha)_{(S, S')}\] holds. Suppose $p = \rep{l}{r}$ with residual $M$. On the left we have:
  
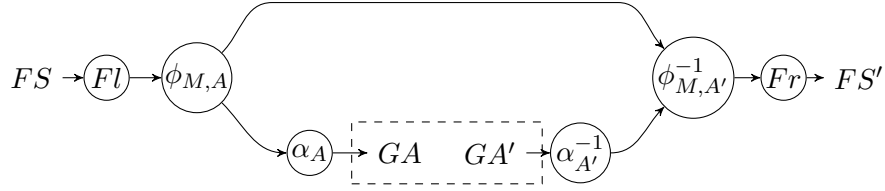
\begin{center}
    \begin{tikzpicture}
\begin{scope}[on grid]
\node[vert] (l) at (0, 0) {$Fl$};
\node[vert] (r) at (9, 0) {$Fr$};

\node (S) [left of=l] {$FS$};
\node[vert] (phi) [right = 1.2 of l] {$\phi_{M,A}$};
\node[vert] (alpha) [below right = 1 and 1.5 of phi] {$\alpha_A$};
\node (A) [right = 1.2 of alpha] {$GA$};
\node (S') [right of=r] {$FS'$};
\node[vert] (phii) [left = 1.2 of r] {$\phi^{-1}_{M, A'}$};
\node[vert] (alphai) [below left = 1 and 1.5 of phii] {$\alpha^{-1}_{A'}$};
\node (A') [left = 1.2 of alphai] {$GA'$};

\draw[->] (S) -- (l);
\draw[->] (l) -- (phi);
\draw[->] (phi) to[out=south east,in=west] (alpha);
\draw[->] (alpha) -- (A);

\draw[<-] (S') -- (r);
\draw[<-] (r) -- (phii);
\draw[<-] (phii) to[out=south west,in=east] (alphai);
\draw[<-] (alphai) -- (A');

\draw[->] (phi) to[out=north east, in=west] ++(1,1)
 to ($(phii) + (-1, 1)$)
 to[out=east, in=north west] (phii)
;

\node[draw,dashed,fit=(A) (A'), inner xsep = 6pt] (box) {};
\end{scope}
\end{tikzpicture}
  \end{center}
  We use the coend relation to place an $\alpha$ on either side:
  
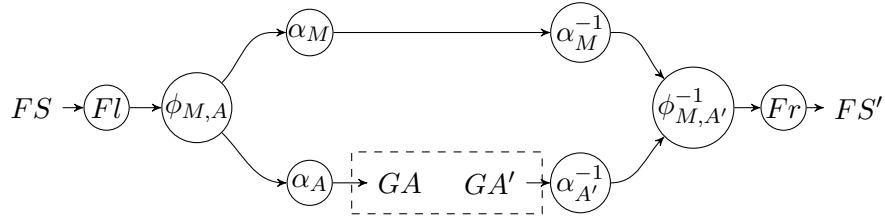
\begin{center}
    \begin{tikzpicture}
\begin{scope}[on grid]
\node[vert] (l) at (0, 0) {$Fl$};
\node[vert] (r) at (9, 0) {$Fr$};

\node (S) [left of=l] {$FS$};
\node[vert] (phi) [right = 1.2 of l] {$\phi_{M,A}$};
\node[vert] (alpha) [below right = 1 and 1.5 of phi] {$\alpha_A$};
\node[vert] (alpha2) [above right = 1 and 1.5 of phi] {$\alpha_M$};
\node (A) [right = 1.2 of alpha] {$GA$};
\node (S') [right of=r] {$FS'$};
\node[vert] (phii) [left = 1.2 of r] {$\phi^{-1}_{M, A'}$};
\node[vert] (alphai) [below left = 1 and 1.5 of phii] {$\alpha^{-1}_{A'}$};
\node[vert] (alphai2) [above left = 1 and 1.5 of phii] {$\alpha^{-1}_M$};
\node (A') [left = 1.2 of alphai] {$GA'$};

\draw[->] (S) -- (l);
\draw[->] (l) -- (phi);
\draw[->] (phi) to[out=south east,in=west] (alpha);
\draw[->] (alpha) -- (A);

\draw[<-] (S') -- (r);
\draw[<-] (r) -- (phii);
\draw[<-] (phii) to[out=south west,in=east] (alphai);
\draw[<-] (alphai) -- (A');

\draw[->] (phi) to[out=north east, in=west] ++(1,1)
 to (alpha2);
\draw[->] (alpha2) -- (alphai2);
\draw[->] (alphai2)
 to ($(phii) + (-1, 1)$)
 to[out=east, in=north west] (phii)
;

\node[draw,dashed,fit=(A) (A'), inner xsep = 6pt] (box) {};
\end{scope}
\end{tikzpicture}
  \end{center}
  And then monoidality of $\alpha$ to commute it past $\phi$.
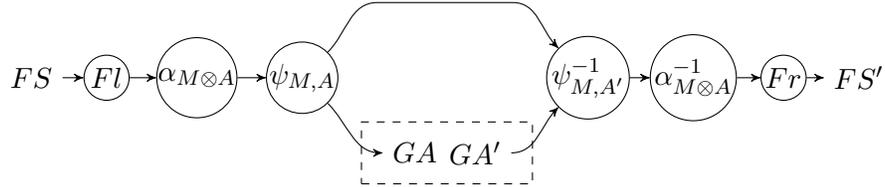
  \begin{center}
    \begin{tikzpicture}
\begin{scope}[on grid]
\node[vert] (l) at (0, 0) {$Fl$};
\node[vert] (r) at (9, 0) {$Fr$};

\node (S) [left of=l] {$FS$};
\node[vert] (alpha) [right = 1.2 of l] {$\alpha_{M \otimes A}$};
\node[vert] (phi) [right = 1.4 of alpha] {$\psi_{M,A}$};
\node (A) [below right = 1 and 1.5 of phi] {$GA$};
\node (S') [right of=r] {$FS'$};
\node[vert] (alphai) [left = 1.2 of r] {$\alpha^{-1}_{M \otimes A}$};
\node[vert] (phii) [left = 1.4 of alphai] {$\psi^{-1}_{M, A'}$};
\node (A') [below left = 1 and 1.5 of phii] {$GA'$};

\draw[->] (S) -- (l);
\draw[->] (l) -- (alpha);
\draw[->] (alpha) -- (phi);
\draw[->] (phi) to[out=south east,in=west] (A);

\draw[<-] (S') -- (r);
\draw[<-] (r) -- (alphai);
\draw[<-] (alphai) -- (phii);
\draw[<-] (phii) to[out=south west,in=east] (A');

\draw[->] (phi) to[out=north east, in=west] ++(1,1)
 to ($(phii) + (-1, 1)$)
 to[out=east, in=north west] (phii)
;

\node[draw,dashed,fit=(A) (A'), inner xsep = 8pt] (box) {};
\end{scope}
\end{tikzpicture}
  \end{center}
  Finally, $\alpha$ commutes with $F l$ and $F r$ by naturality.
  \begin{center}
    \begin{tikzpicture}
\begin{scope}[on grid]
\node[vert] (alpha) at (0, 0) {$\alpha_S$};
\node[vert] (alphai) at (9, 0) {$\alpha_{S'}^{-1}$};

\node (S) [left of=alpha] {$GS$};
\node[vert] (l) [right = 1.2 of alpha] {$Gl$};
\node[vert] (phi) [right = 1.4 of l] {$\psi_{M,A}$};
\node (A) [below right = 1 and 1.5 of phi] {$GA$};
\node (S') [right of=alphai] {$GS'$};
\node[vert] (r) [left = 1.2 of alphai] {$Gr$};
\node[vert] (phii) [left = 1.4 of r] {$\psi^{-1}_{M, A'}$};
\node (A') [below left = 1 and 1.5 of phii] {$GA'$};

\draw[->] (S) -- (alpha);
\draw[->] (alpha) -- (l);
\draw[->] (l) -- (phi);
\draw[->] (phi) to[out=south east,in=west] (A);

\draw[<-] (S') -- (alphai);
\draw[<-] (alphai) -- (r);
\draw[<-] (r) -- (phii);
\draw[<-] (phii) to[out=south west,in=east] (A');

\draw[->] (phi) to[out=north east, in=west] ++(1,1)
 to ($(phii) + (-1, 1)$)
 to[out=east, in=north west] (phii)
;

\node[draw,dashed,fit=(A) (A'), inner xsep = 8pt] (box) {};
\end{scope}
\end{tikzpicture}
  \end{center}
  This is the diagram for $\Optic(G)(p) \circ {\Optic(\alpha)}_{(S, S')}$.
\end{proof}

\begin{theorem}
  $\Optic_\C$ is symmetric monoidal, where $(S, S') \otimes (T, T') = (S \otimes T, S' \otimes T')$, the unit object is $(I, I)$, and the action on a pair of morphisms $\rep{l}{r} : (S, S') \hto (A, A')$ and $\rep{l'}{r'} : (T, T') \hto (B, B')$ is given by:
  
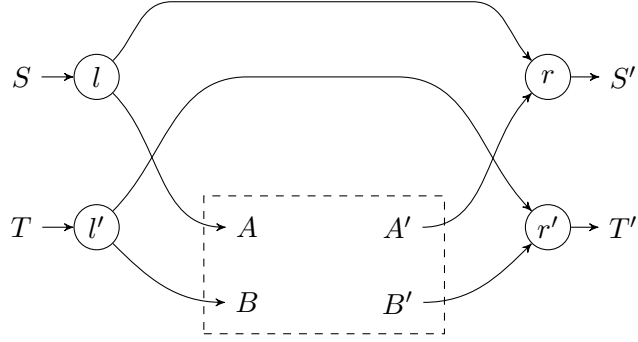
\begin{center}
    \begin{tikzpicture}
\begin{scope}[on grid]

\node[vert] (l) at (0, 0) {$l$};
\node[vert] (r) at (6, 0) {$r$};

\node (S) [left of=l] {$S$};
\node (A) [below right = 2 and 2 of l] {$A$};
\node (S') [right of=r] {$S'$};
\node (A') [below left = 2 and 2 of r] {$A'$};

\draw[->] (S) -- (l);
\draw[->] (l) to[out=south east,in=west] (A);

\draw[<-] (S') -- (r);
\draw[<-] (r) to[out=south west,in=east] (A');

\draw[->] (l) to[out=north east, in=west] ++(1,1)
 to ++(4,0)
 to[out=east, in=north west] (r)
;

\node[vert] (l') at (0, -2) {$l'$};
\node[vert] (r') at (6, -2) {$r'$};

\node (T) [left of=l'] {$T$};
\node (B) [below right = 1 and 2 of l'] {$B$};
\node (T') [right of=r'] {$T'$};
\node (B') [below left = 1 and 2 of r'] {$B'$};

\draw[->] (T) -- (l');
\draw[->] (l') to[out=south east,in=west] (B);

\draw[<-] (T') -- (r');
\draw[<-] (r') to[out=south west,in=east] (B');

\draw[->] (l') 
 to[out=north east, in=west] ++(2,2)
 to ++(2,0)
 to[out=east, in=north west] (r')
;

\node[draw,dashed,fit=(A) (A') (B) (B'), inner xsep = 8pt] (box) {};

\end{scope}
\end{tikzpicture}
  \end{center}
\end{theorem}
\begin{proof}
  Suppose the two optics have residuals $M$ and $N$ respectively. Written equationally, their tensor is:
  \begin{align*}
    \rep{l}{r} \otimes \rep{l'}{r'} &:= \rep{(M \otimes s_{A,N} \otimes B)(l \otimes l')}{(r \otimes r')(M \otimes s_{A',N} \otimes B')}
  \end{align*}
  This does not depend on the choice of representatives, as demonstrated by the equivalence of the following diagrams:
  
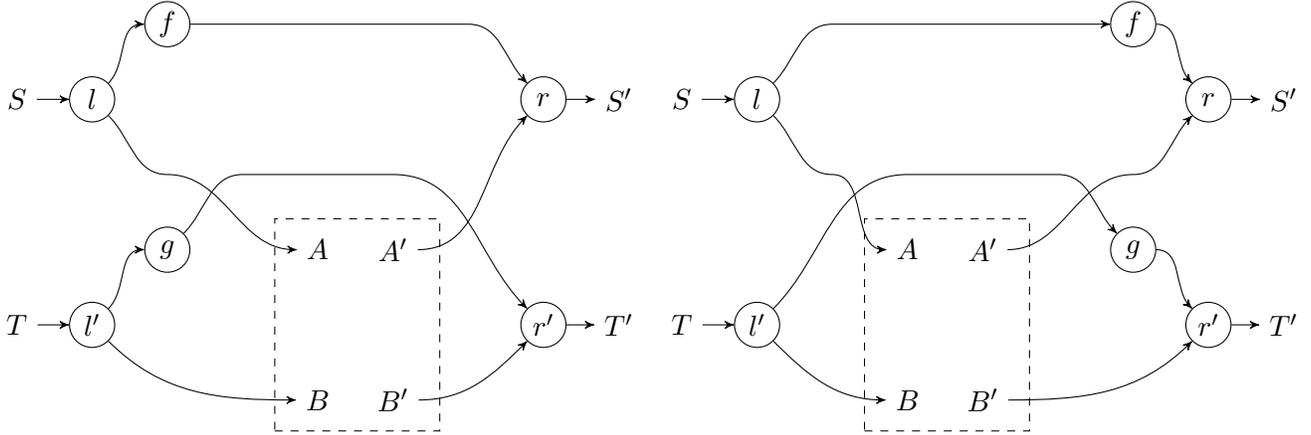
\begin{center}
    \begin{tikzpicture}
\begin{scope}[on grid]

\node[vert] (l) at (-1, 0) {$l$};
\node[vert] (r) at (5, 0) {$r$};

\node[vert] (f) at ($(l) + (1,1)$) {$f$};

\node (S) [left of=l] {$S$};
\node (A) [below right = 2 and 3 of l] {$A$};
\node (S') [right of=r] {$S'$};
\node (A') [below left = 2 and 2 of r] {$A'$};

\draw[->] (S) -- (l);
\draw[->] (l) 
to[out=south east,in=west] ($(l) + (1,-1)$)
to[out=east,in=west] (A);

\draw[<-] (S') -- (r);
\draw[<-] (r) to[out=south west,in=east] (A');

\draw[->] (l) to[out=north east, in=west] (f);
\draw[->] (f) to[out=east, in=west] ($(r) + (-1,1)$)
 to[out=east, in=north west] (r)
;

\node[vert] (l') at (-1, -3) {$l'$};
\node[vert] (r') at (5, -3) {$r'$};

\node[vert] (g) at ($(l') + (1,1)$) {$g$};

\node (T) [left of=l'] {$T$};
\node (B) [below right = 1 and 3 of l'] {$B$};
\node (T') [right of=r'] {$T'$};
\node (B') [below left = 1 and 2 of r'] {$B'$};

\draw[->] (T) -- (l');
\draw[->] (l') to[out=south east,in=west] (B);

\draw[<-] (T') -- (r');
\draw[<-] (r') to[out=south west,in=east] (B');

\draw[->] (l') to[out=north east, in=west] (g);
\draw[->] (g) 
 to[out=north east, in=west] ($(g) + (1,1)$)
 to ($(r') + (-2,2)$)
 to[out=east, in=north west] (r')
;

\node[draw,dashed,fit=(A) (A') (B) (B'), inner xsep = 8pt] (box) {};

\end{scope}
\end{tikzpicture}
    \begin{tikzpicture}
\begin{scope}[on grid]

\node[vert] (l) at (0, 0) {$l$};
\node[vert] (r) at (6, 0) {$r$};

\node[vert] (f) at ($(r) + (-1,1)$) {$f$};

\node (S) [left of=l] {$S$};
\node (A) [below right = 2 and 2 of l] {$A$};
\node (S') [right of=r] {$S'$};
\node (A') [below left = 2 and 3 of r] {$A'$};

\draw[->] (S) -- (l);
\draw[->] (l) 
to[out=south east,in=west] ($(l) + (1,-1)$)
to[out=east,in=west] (A);

\draw[<-] (S') -- (r);
\draw[<-] (r) 
to[out=south west,in=east] ($(r) + (-1,-1)$)
to[out=west,in=east] (A');

\draw[->] (l) 
  to[out=north east, in=west] ($(l) + (1,1)$)
  to[out=east, in=west] (f);
\draw[->] (f) to[out=east, in=north west] (r);

\node[vert] (l') at (0, -3) {$l'$};
\node[vert] (r') at (6, -3) {$r'$};

\node[vert] (g) at ($(r') + (-1,1)$) {$g$};

\node (T) [left of=l'] {$T$};
\node (B) [below right = 1 and 2 of l'] {$B$};
\node (T') [right of=r'] {$T'$};
\node (B') [below left = 1 and 3 of r'] {$B'$};

\draw[->] (T) -- (l');
\draw[->] (l') to[out=south east,in=west] (B);

\draw[<-] (T') -- (r');
\draw[<-] (r') to[out=south west,in=east] (B');

\draw[->] (l') 
  to[out=north east, in=west] ($(l') + (2,2)$)
  to[out=east, in = west] ($(g) + (-1,1)$)
  to[out=east, in = north west] (g);
\draw[->] (g) to[out=east, in=north west] (r');

\node[draw,dashed,fit=(A) (A') (B) (B'), inner xsep = 8pt] (box) {};

\end{scope}
\end{tikzpicture}
  \end{center}
  To check functoriality of $\otimes$, suppose we have optics
  \begin{align*}
    \rep{l_1}{r_1} : (S_1, S_1') &\hto (S_2, S_2') \\
    \rep{l_2}{r_2} : (S_2, S_2') &\hto (S_3, S_3') \\
    \rep{p_1}{q_1} : (T_1, T_1') &\hto (T_2, T_2') \\
    \rep{p_2}{q_2} : (T_2, T_2') &\hto (T_3, T_3').
  \end{align*}
  The string diagram for $(\rep{l_2}{r_2} \circ \rep{l_1}{r_1}) \otimes (\rep{p_2}{q_2} \circ \rep{p_1}{q_1})$ is:
  
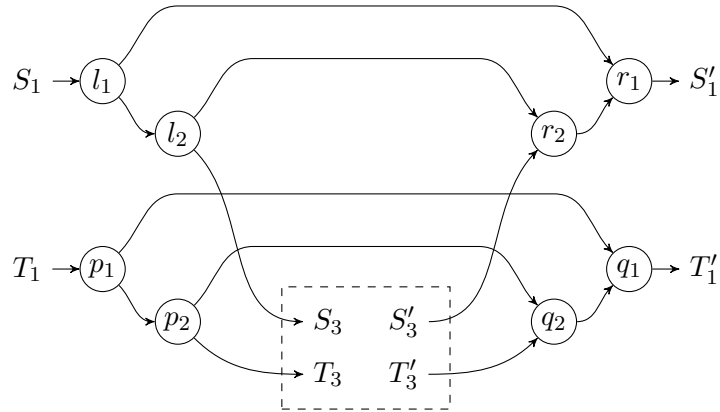
\begin{center}
    \begin{tikzpicture}
\begin{scope}[on grid]

%%% The S part

\node[vert] (l1) at (0, 0) {$l_1$};
\node[vert, below right = 0.7 and 1 of l1] (l2) {$l_2$};
\node[vert] (r1) at (7, 0) {$r_1$};
\node[vert, below left = 0.7 and 1 of r1] (r2) {$r_2$};

\node (S3) [below right = 2.5 and 2 of l2] {$S_3$};
\node (S3') [below left = 2.5 and 2 of r2] {$S_3'$};

\node (S1) [left of=l1] {$S_1$};
\node (S1') [right of=r1] {$S_1'$};

\draw[->] (S1) -- (l1);
\draw[<-] (S1') -- (r1);

\draw[->] (l1) to[out=north east, in=west] ++(1,1)
 to ($(r1) + (-1,1)$)
 to[out=east, in=north west] (r1)
;

\draw[->] (l2) to[out=north east, in=west] ++(1,1)
 to ($(r2) + (-1,1)$)
 to[out=east, in=north west] (r2)
;

\draw[->] (l1) to[out=south east,in=west] (l2);
\draw[->] (r2) to[out=east, in=south west] (r1);

\draw[->] (l2) to[out=south east,in=west] (S3);
\draw[<-] (r2) to[out=south west,in=east] (S3');

%%% The T part

\node[vert] (p1) at (0, -2.5) {$p_1$};
\node[vert, below right = 0.7 and 1 of p1] (p2) {$p_2$};
\node[vert] (q1) at (7, -2.5) {$q_1$};
\node[vert, below left = 0.7 and 1 of q1] (q2) {$q_2$};

\node (T3) [below right = 0.7 and 2 of p2] {$T_3$};
\node (T3') [below left = 0.7 and 2 of q2] {$T_3'$};

\node (T1) [left of=p1] {$T_1$};
\node (T1') [right of=q1] {$T_1'$};

\draw[->] (T1) -- (p1);
\draw[<-] (T1') -- (q1);

\draw[->] (p1) 
 to[out=north east, in=west] ++(1,1)
 to ($(q1) + (-1,1)$)
 to[out=east, in=north west] (q1)
;

\draw[->] (p2) to[out=north east, in=west] ++(1,1)
 to ($(q2) + (-1,1)$)
 to[out=east, in=north west] (q2)
;

\draw[->] (p1) to[out=south east,in=west] (p2);
\draw[->] (q2) to[out=east, in=south west] (q1);

\draw[->] (p2) to[out=south east,in=west] (T3);
\draw[<-] (q2) to[out=south west,in=east] (T3');

\node[draw,dashed,fit=(S3) (S3') (T3) (T3'), inner xsep = 8pt] (box) {};

\end{scope}
\end{tikzpicture}
  \end{center}
  And for $(\rep{l_2}{r_2} \otimes \rep{p_2}{q_2}) \circ (\rep{l_1}{r_1} \otimes \rep{p_1}{q_1})$ is:
  
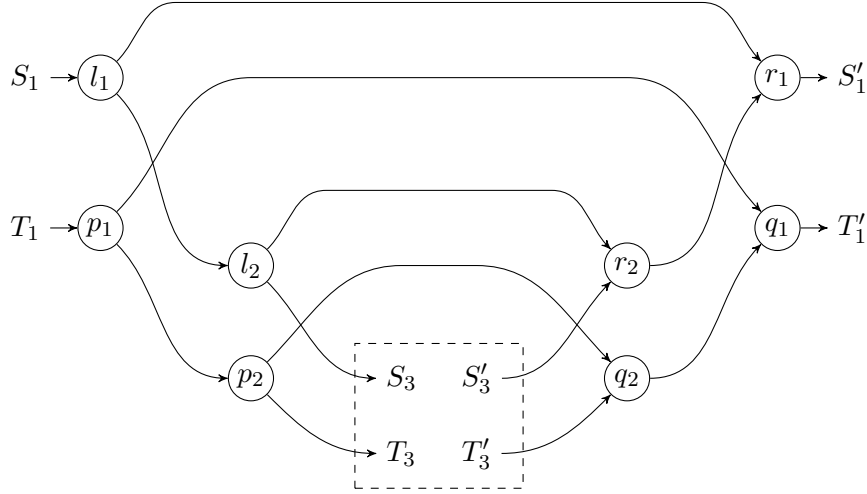
\begin{center}
    \begin{tikzpicture}
\begin{scope}[on grid]

%%% Outer:

\node[vert] (l1) at (0, 0) {$l_1$};
\node[vert] (r1) at (9, 0) {$r_1$};

\node (S1) [left of=l1] {$S_1$};
\node[vert] (l2) [below right = 2.5 and 2 of l1] {$l_2$};
\node (S1') [right of=r1] {$S_1'$};
\node[vert] (r2) [below left = 2.5 and 2 of r1] {$r_2$};

\draw[->] (S1) -- (l1);
\draw[->] (l1) to[out=south east,in=west] (l2);

\draw[<-] (S1') -- (r1);
\draw[<-] (r1) to[out=south west,in=east] (r2);

\draw[->] (l1) to[out=north east, in=west] ++(1,1)
 to ($(r1) + (-1,1)$)
 to[out=east, in=north west] (r1)
;

\node[vert] (p1) at (0, -2) {$p_1$};
\node[vert] (q1) at (9, -2) {$q_1$};

\node (T1) [left of=p1] {$T_1$};
\node[vert] (p2) [below right = 2 and 2 of p1] {$p_2$};
\node (T') [right of=q1] {$T_1'$};
\node[vert] (q2) [below left = 2 and 2 of q1] {$q_2$};

\draw[->] (T1) -- (p1);
\draw[->] (p1) to[out=south east,in=west] (p2);

\draw[<-] (T') -- (q1);
\draw[<-] (q1) to[out=south west,in=east] (q2);

\draw[->] (p1) 
 to[out=north east, in=west] ++(2,2)
 to ($(q1) + (-2,2)$)
 to[out=east, in=north west] (q1)
;

%%% Inner:

\draw[->] (l2) to[out=north east, in=west] ++(1,1)
 to ($(r2) + (-1,1)$)
 to[out=east, in=north west] (r2)
;
\draw[->] (p2) 
 to[out=north east, in=west] ++(2,1.5)
 to ($(q2) + (-2,1.5)$)
 to[out=east, in=north west] (q2)
;

\node(S3) [below right = 1.5 and 2 of l2] {$S_3$};
\node(S3') [below left = 1.5 and 2 of r2] {$S_3'$};

\draw[->] (l2) to[out=south east,in=west] (S3);
\draw[<-] (r2) to[out=south west,in=east] (S3');

\node (T3) [below right = 1 and 2 of p2] {$T_3$};
\node (T3') [below left = 1 and 2 of q2] {$T_3'$};

\draw[->] (p2) to[out=south east,in=west] (T3);
\draw[<-] (q2) to[out=south west,in=east] (T3');

\node[draw,dashed,fit=(S3) (S3') (T3) (T3'), inner xsep = 8pt] (box) {};
the a
\end{scope}
\end{tikzpicture}
  \end{center}
  These two diagrams are equivalent: we can use the naturality of the symmetry morphism to push $l_2$ and $r_2$ past the crossing to be next to $p_2$ and $q_2$ respectively. This creates two extra twists that can be cancelled in the center of the diagram.

  The structure morphisms are all lifted from the structure morphisms in $\C \times \C^\op$:
  \begin{align*}
    \alpha_{(R, R'), (S, S'), (T, T')} &:= \iota(\alpha_{R,S,T}, \alpha_{R',S',T'}^{-1}) \\
    \lambda_{(S, S')} &:= \iota(\lambda_{S}, \lambda_{S'}^{-1}) \\
    \rho_{(S, S')} &:= \iota(\rho_{S}, \rho_{S'}^{-1}) \\
    s_{(S, S'), (T, T')} &:= \iota(s_{S, T}, s_{T', S'})
  \end{align*}

  Note that because $\iota(S, S') = (S, S')$, the equations required  to hold for $\iota$ to be a monoidal functor hold by definition (although we don't yet know that $\Optic_\C$ is monoidal). The pentagon and triangle equations then hold in $\Optic_\C$, as they are the image of the same diagrams in $\C \times \C^\op$ under $\iota$. The only remaining thing to verify is that these structure maps are natural in $\Optic_\C$, but this follows from the previous proposition.
\end{proof}

%  In the Haskell \lenslib{} library, the monoidal product on optics is denoted ``\mintinline{haskell}{alongside}'' for the product and ``\mintinline{haskell}{without}'' for the coproduct.

\begin{proposition}
  For monoidal $F : \C \to \D$, the induced $\Optic(F) : \Optic_\C \to \Optic_\D$ is also monoidal.
\end{proposition}
\begin{proof}
The structure morphisms for monoidality are given by lifting the structure morphisms for $F$:
\begin{align*}
\phi_{(S, S'), (T, T')} &:= \iota(\phi_{S, T}, \phi_{S', T'}) &&: F(S, S') \otimes F(T, T') \hto F((S, S') \otimes (T, T')) \\
\phi &:= \iota(\phi, \phi) &&: (I, I) \hto F(I, I)
\end{align*}
The monoidality axioms follow by lifting the axioms for $F$ and naturality follows by Proposition~\ref{prop:iota-naturality}.
\end{proof}

\begin{proof}[Proof of Theorem~\ref{thm:optic-functor}]
The functor is well defined on its domain: if $\C$ is small then $\Optic_\C$ exists. The only property left to check is functoriality, i.e. that for monoidal functors $F : \C \to \D$ and $G : \D \to \E$ we have
\[ \Optic(G) \circ \Optic(F) = \Optic(G \circ F).\]
On objects this is clear, as $\Optic(F)(S, S') = (FS, FS')$. On a morphism $\rep{l}{r} : (S, S') \hto (A, A')$ in $\C$, we check:
\begin{align*}
(\Optic(G) \circ \Optic(F))(\rep{l}{r})
&= \Optic(G) \left(\rep{\phi^{-1}_{M,A} (Fl)}{(Fr) \phi_{M,A'}}\right) \\
&= \rep{\psi^{-1}_{FM,FA} (G(\phi^{-1}_{M,A} (Fl)))}{(G((Fr) \phi_{M,A'}))\psi_{FM,FA'}} \\
&= \rep{\psi^{-1}_{FM,FA} (G\phi^{-1}_{M,A}) (GFl)}{(GFr) (G\phi_{M,A'})\psi_{FM,FA'}} \\
&=\Optic(G \circ F)(\rep{l}{r})
\end{align*}
where $\phi$ and $\psi$ denote the structure maps for $F$ and $G$ respectively, and in the last step we use that $(G\phi_{M,A'})\psi_{FM,FA'}$ is by definition the structure map for $G \circ F$. Checking that the identity is preserved is similar.
%
%Secondly, we note that $\Optic$ acts functorially on 2-cells. The action on 2-cells is given by Proposition~\ref{prop:iota-naturality}, and functoriality follows by the functoriality of $\iota$.
%
%Finally, we check that horizontal composition is preserved. Horizontal composition in $\SymmMonCat$ is defined in terms of whiskering and vertical composition, so it suffices to verify that whiskering is preserved. But indeed it is:
%\begin{align*}
%{\Optic(F\alpha)}_{(S, S')}
%&= \iota({(F\alpha)}_{S}, {(F\alpha)}^{-1}_{S'}) \\
%&= \iota(F(\alpha_{S}), F(\alpha_{S'}^{-1})) \\
%&= {(\Optic(F)\Optic(\alpha))}_{(S, S')}
%\end{align*}
%by Lemma~\ref{lem:iota-commute-with-opticf}, and
%\begin{align*}
%{\Optic(\alpha G)}_{(S, S')}
%&= \iota(\alpha_{GS}, \alpha_{GS'}^{-1}) \\
%&= \iota({(\alpha G)}_{S}, {(\alpha G)}_{S'}^{-1}) \\
%&= {(\Optic(\alpha)\Optic(G))}_{(S, S')}
%\end{align*}
\end{proof}

This doesn't extend to a strict 2-functor $\SymmMonCat \to \SymmMonCat$, as there is only an action of $\Optic$ on natural \emph{isomorphisms}. It is however functorial on natural isomorphisms, giving a 2-functor on the `homwise-core' of $\SymmMonCat$. We do not explore this any further in the present note.

\begin{proposition}
  If $\C$ is a strict symmetric monoidal category then $\Optic_\C$ is strict, and $\iota : \C \times \C^\op \to \Optic_\C$ is a strict monoidal functor. For $F : \C \to \D$ a strict monoidal functor, the induced functor $\Optic(F) : \Optic_\C \to \Optic_\D$ is also strict.
\end{proposition}
\begin{proof}
  The structure maps of $\Optic_\C$ are given by $\iota$ applied to the structure maps of $\C$. If the latter are identities, then so are the former---the identity morphisms in $\Optic_\C$ are by definition $\iota(\id_S, \id_{S'})$.

  That $\iota$ is strict is clear, as the structure morphisms in $\Optic_\C$ are exactly the structure morphisms in $\C \times \C^\op$ under $\iota$. 
  
  Finally, the structure morphisms of $\Optic(F)$ are lifted from $F$, so if the latter is strict then so is the former.
\end{proof}

\subsection{Teleological Categories}\label{sec:teleological-categories}
In this section we establish a universal property of the $\Optic$ construction. The idea is that every optic $\rep{l}{r} : (S, S') \hto (A, A')$ consists of a morphism $S \to M \otimes A$ and the `formal dual' of a morphism $M \otimes A' \to S'$, composed with a `formal counit' that traces out the object $M$:
\begin{center}
  \begin{tikzpicture}
\node[vert] (l) at (0, 0) {$l$};
\node[vert] (r) at (0, -2) {$r^*$};

\node (S) [left of=l]{$S$};
\node (S') [left of=r]{$S'$};

\node (A) [right =2 of l]{$A$};
\node (A') [right = 2 of r]{$A'$};

\draw[->] (S) -- (l);
\draw[->] (l) -- (A);
\draw[<-] (S') -- (r);
\draw[<-] (r) -- (A');

\draw[->-=0.5] (l) to[out = -20, in = 20, edge label=$M$] (r);

%\node (Sin) {$S$};
%\node (f) [vert, right of=Sin] {$f$};
%\node (Sout) [right of=f] {$S$};
%\node (Spout) [right of=Sout] {$S'$};
%\node (g) [vert, right = 0.5 of Spout] {$g$};
%\node (Spin) [right of=g] {$S'$};
%
%\draw[->] (Sin) -- (f);
%\draw[->] (f) -- (Sout);
%\draw[->] (Spout) -- (g);
%\draw[->] (g) -- (Spin);
\end{tikzpicture}
\end{center}

It will be convenient to equip $\Optic_\C$ with a slightly different symmetric monoidal structure:

\begin{definition}
  The \emph{switched} monoidal product on $\Optic_\C$ is given on objects by
  \begin{align*}
    (S, S') \switched (T, T') := (S \otimes T, T' \otimes S')
  \end{align*}
  And on morphisms $\rep{l}{r} : (S, S') \hto (A, A')$ and $\rep{l'}{r'} : (T, T') \hto (B, B')$ by:
  
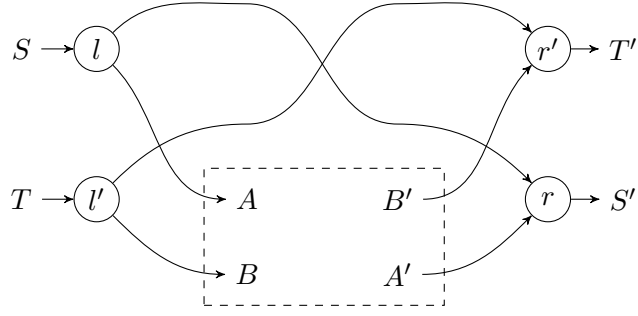
\begin{center}
    \begin{tikzpicture}
\begin{scope}[on grid]

\node[vert] (l) at (0, 0) {$l$};
\node[vert] (r') at (6, 0) {$r'$};
\node[vert] (l') at (0, -2) {$l'$};
\node[vert] (r) at (6, -2) {$r$};

\node (S) [left of=l] {$S$};
\node (A) [below right = 2 and 2 of l] {$A$};
\node (T') [right of=r'] {$T'$};
\node (B') [below left = 2 and 2 of r'] {$B'$};

\node (T) [left of=l'] {$T$};
\node (B) [below right = 1 and 2 of l'] {$B$};
\node (S') [right of=r] {$S'$};
\node (A') [below left = 1 and 2 of r] {$A'$};

\draw[->] (S) -- (l);
\draw[<-] (T') -- (r');
\draw[->] (T) -- (l');
\draw[<-] (S') -- (r);

\draw[->] (l) to[out=south east,in=west] (A);
\draw[<-] (r') to[out=south west,in=east] (B');
\draw[->] (l') to[out=south east,in=west] (B);
\draw[<-] (r) to[out=south west,in=east] (A');

\draw[->] (l') 
 to[out=north east, in=west] ++(2,1)
 to[out=east, in=west] ++(2,1.6)
 to[out=east, in=north west] (r')
;

\draw[<-] (r) 
 to[out=north west, in=east] ++(-2,1)
 to[out=west, in=east] ++(-2,1.6)
 to[out=west, in=north east] (l)
;

\node[draw,dashed,fit=(A) (A') (B) (B'), inner xsep = 8pt] (box) {};

\end{scope}
\end{tikzpicture}
  \end{center}
\end{definition}

The universal property for $\Optic_\C$ given in this section is an argument for this being the ``morally correct'' tensor, although it does seem a little strange. When we later discuss lawful optics, we are forced to use the unswitched tensor to maintain the invariant that our objects are of the form $(X, X)$.

\begin{proposition}
  $(\Optic_\C, \switched, (I, I))$ is a symmetric monoidal category. 
  %that is monoidally equivalent to $\Optic_\C$ with the unswitched tensor.
\end{proposition}
\begin{proof}
  The proof that $(\Optic_\C, \switched, (I, I))$ is symmetric monoidal is nearly identical to that for the unswitched tensor. Note that due to the switching, the structure morphisms are slightly different:
  \begin{align*}
    \alpha_{(R, R'), (S, S'), (T, T')} &:= \iota(\alpha_{R,S,T}, \alpha_{T',S',R'}^{-1}) \\
    \lambda_{(S, S')} &:= \iota(\lambda_{S}, \rho_{S'}^{-1}) \\
    \rho_{(S, S')} &:= \iota(\rho_{S}, \lambda_{S'}^{-1}) \\
    s_{(S, S'), (T, T')} &:= \iota(s_{S, T}, s_{S', T'})
  \end{align*}

%  The categories are monoidally equivalent via the identity functor \[\id : (\Optic_\C, \switched, (I, I)) \to (\Optic_\C, \otimes, (I, I)),\] where the structure isomorphisms making this functor monoidal are given by
%  \begin{align*}
%    \iota(\id_{S \otimes S'}, s_{T, T'}) : \id(S, S') \otimes \id(T, T') \to \id((S, S') \switched (T, T'))
%  \end{align*}
\end{proof}

\begin{remark}
  Just as in the unswitched case, if $\C$ is a strict monoidal category than so is $(\Optic_\C, \switched, (I, I))$.
\end{remark}

We now define the structure on a symmetric monoidal category universally provided by the $\Optic$ construction.

\begin{definition}[Compare {\cite[Definition 5.1]{CoherenceForLenses}}]
  A \emph{teleological category} is a symmetric monoidal category $(\T, \teletimes, I)$, equipped with:
  \begin{itemize}
  \item A symmetric monoidal subcategory $\T_d$ of \emph{dualisable morphisms} containing all the objects of $\T$, with an involutive symmetric monoidal functor ${(-)}^* : \T_d \to \T_d^\op$, where---not finding a standard symbol for such a thing---we mean $\T_d^\op$ to be the category with both the direction of the arrows \emph{and} the order of the tensor flipped: ${(A \teletimes B)}^* \cong B^* \teletimes A^*$. Note that there is therefore also a canonical isomorphism $\phi : I \cong I^*$
  \item A symmetric monoidal extranatural family of morphisms $\varepsilon_X : X \teletimes X^* \to I$, called \emph{counits}, natural with respect to the \emph{dualisable} morphisms.
  \end{itemize}
\end{definition}
Unpacking the definition, $\varepsilon$ being a symmetric monoidal extranatural transformation amounts to the following diagrams in $\T$ commuting:
\[
  \begin{tikzcd}
    X \teletimes Y^* \ar[r, "f \teletimes Y^*"]  \ar[d, "X \teletimes f^*", swap] & Y \teletimes Y^* \ar[d, "\varepsilon_Y"] \\
    X \teletimes X^* \ar[r, "\varepsilon_X", swap] & I
  \end{tikzcd} \hspace{1cm}
  \begin{tikzcd}
    X^* \teletimes X \ar[r, "s"] \ar[d, "\cong" swap]  & X \teletimes X^* \ar[d, "\varepsilon_X"] \\
   X^* \teletimes {(X^*)}^* \ar[r, "\varepsilon_{X^*}", swap] & I
  \end{tikzcd}\]
\[
  \begin{tikzcd}[column sep = large]
    X \teletimes Y \teletimes Y^* \teletimes X^* \ar[r, "X \teletimes \varepsilon_Y \teletimes X^*"] \ar[d, "\cong" swap]  & X \teletimes X^* \ar[d, "\varepsilon_X"] \\
   X \teletimes Y \teletimes (X \teletimes Y)^* \ar[r, "\varepsilon_{X \teletimes Y}", swap] & I
  \end{tikzcd} \hspace{1cm}
  \begin{tikzcd}
    I \teletimes I^* \ar[r,"I \teletimes \phi"] \ar[dr, swap, "\varepsilon_I"] & I \teletimes I \ar[d, "\cong"] \\
    & I
  \end{tikzcd}
\]
where $f : X \to Y$ is dualisable.

%%%% The following is a version with duality given by reflection.
% \begin{definition}[{\cite[Definition 5.1]{CoherenceForLenses}}]
%   A \emph{teleological category is} a symmetric monoidal category $(\C, \otimes, I)$, equipped with:
%   \begin{itemize}
%   \item A wide symmetric monoidal subcategory $\C_d$ of \emph{dualisable morphisms}, with an involutive strong symmetric monoidal functor $(-)^* : \C_d \to \C_d^\op$; and,
%   \item A family of morphisms $\varepsilon_X : X \otimes X^* \to I$, called \emph{counits}, natural with respect to the morphisms in $\C_d$, such that
%     \[
%       \begin{tikzcd}
%         X \otimes Y^* \ar[r, "f \otimes Y^*"]  \ar[d, "X \otimes f^*", swap] & Y \otimes Y^* \ar[d, "\varepsilon_Y"] \\
%         X \otimes X^* \ar[r, "\varepsilon_X", swap] & I
%       \end{tikzcd} \hspace{1cm}
%       \begin{tikzcd}
%         X^* \otimes X \ar[r, "s_{X^*, X}"]  \ar[dr, "\varepsilon_{X^*}", swap] & X \otimes X^* \ar[d, "\varepsilon_X"] \\
%         & I
%       \end{tikzcd}
%     \]
%     \[
%       \begin{tikzcd}
%         X \otimes Y \otimes X^* \otimes Y^* \ar[r, "X \otimes s_{Y,X^*} \otimes X^*"]  \ar[dr, "\varepsilon_{X \otimes Y}", swap] & X \otimes X^* \otimes Y \otimes Y^* \ar[d, "\varepsilon_X \otimes \varepsilon_Y"] \\
%         & I
%       \end{tikzcd}
%     \]
%     commute for all $f : X \to Y$ is in $\C_d$.
%   \end{itemize}
% \end{definition}

Note that because $\T_d$ is symmetric monoidal and has the same collection of objects as $\T$, the symmetric monoidal structure morphisms of $\T$ must be contained in $\T_d$ and so are dualisable.

\begin{example}
  ~\begin{enumerate}[(1)]
  \item Any compact closed category is a teleological category, where every morphism is dualisable and the unit morphisms have been forgotten.
  \item Any symmetric monoidal category with terminal monoidal unit is trivially teleological, setting the dualisable morphisms to be all isomorphisms.
  \end{enumerate}
\end{example}

This definition of teleological category differs from the original given in~\cite{CoherenceForLenses}, in that the duality switches the order of the tensor product. We do this so that compact closed categories are teleological, but the bookkeeping does admittedly become more confusing.

%\begin{example}
%  \todo{The funny graph example from the coherence paper?}
%\end{example}

\begin{definition}
  A \emph{teleological functor} $F : \T \to \S$ is a symmetric monoidal functor that restricts to a functor $F_d : \T_d \to \S_d$ on the dualisable subcategories, commutes with the duality via a monoidal natural isomorphism $d_X : F(X^*) \to {(FX)}^*$, and such that the counits are preserved:
  \[
   \begin{tikzcd}[column sep = large]
    F(X \teletimes X^*) \ar[r, "\phi_{X, X^*}"]  \ar[d, "F\varepsilon_X", swap] & FX \teletimes F(X^*) \ar[r, "FX \teletimes d_X"] & FX \teletimes (FX)^* \ar[d, "\varepsilon_{FX}"] \\
    FI \ar[rr, "\phi_I", swap] & & I
  \end{tikzcd}
  \]
\end{definition}

Together we have $\Tele$, the category of teleological categories and teleological functors. There are evident functors 
\begin{align*}
U &: \Tele \to \SymmMonCat \\
{(-)}_d &: \Tele \to \SymmMonCat 
\end{align*}
that take a teleological category to its underlying symmetric monoidal category and subcategory of dualisable morphisms respectively.

The definition of teleological category suggests a string diagram calculus similar to that for compact closed categories, but where only counits are allowed and only morphisms known to be dualisable may be passed around a counit. We have of course not proven that such a calculus is sound for teleological categories, but we trust that a sceptical reader could verify our arguments equationally.

\begin{proposition}
  $\Optic_\C$ forms a teleological category, where:
  \begin{itemize}
  \item The dualisable morphisms are all morphisms of the form $\iota(f, g)$;
  \item The involution is given on objects by ${(S, S')}^* := (S', S)$, and on morphisms by $\iota{(f, g)}^* := \iota(g, f)$;
  \item The counit $\varepsilon_{(S, S')} : (S, S') \switched {(S, S')}^* = (S \otimes S', S \otimes S') \to (I, I)$ is given by the connector: \[\varepsilon_{(S, S')} := c_{S \otimes S'}.\]
  \end{itemize}
\end{proposition}
\begin{proof}
  That morphisms of the form $\iota(f, g)$ constitute a symmetric monoidal subcategory is clear, they are the image of the symmetric monoidal functor $\iota$.
  
  The functor ${(-)}^*$ is a symmetric monoidal involution, in fact it is strictly so:
  \begin{align*}
    {\left( (S, S') \switched (T, T') \right)}^*
    &= {\left( S \otimes T, T' \otimes S' \right)}^* \\
    &= {\left(T' \otimes S', S \otimes T  \right)} \\
    &= (T', T) \switched (S', S) \\
    &= {(T, T')}^* \switched {(S, S')}^*
  \end{align*}

  To check extranaturality of $\varepsilon$, suppose we have a dualisable optic $\iota(f, g) : (S, S') \hto (T, T')$, so $f : S \to T$ and $g : T' \to S'$. Happily, all the switching in the definitions cancels out! Extranaturality is witnessed by the equality of the string diagrams:
  \begin{center}
    \begin{tikzpicture}
\begin{scope}[on grid]
\node (S) at (0, 0) {$S$};
\node (S') at (4, 0) {$T$};
\node (T) [below = 1 of S] {$T'$};
\node (T') [below = 1 of S'] {$S'$};
\node[vert] (f) [right = 1 of S] {$f$};
\node[vert] (g) [right = 1 of T] {$g$};

\draw[->] (S) -- (f);
\draw[->] (f) -- (S');

\draw[->] (T) -- (g);
\draw[->] (g) -- (T');

\node (I) [below right = 1 and 2 of T] {$I$};
\node (I') [below left = 1 and 1 of T'] {$I$};

\node[draw,dashed,fit=(I) (I'), inner xsep = 4pt] (box) {};
\end{scope}
\end{tikzpicture}
    \qquad \raisebox{1.5cm}{$=$} \qquad
    \begin{tikzpicture}
\begin{scope}[on grid]
\node (S) at (0, 0) {$S$};
\node (S') at (4, 0) {$T$};
\node (T) [below = 1 of S] {$T'$};
\node (T') [below = 1 of S'] {$S'$};
\node[vert] (f) [left = 1 of S'] {$f$};
\node[vert] (g) [left = 1 of T'] {$g$};

\draw[->] (S) -- (f);
\draw[->] (f) -- (S');

\draw[->] (T) -- (g);
\draw[->] (g) -- (T');

\node (I) [below right = 1 and 1 of T] {$I$};
\node (I') [below left = 1 and 2 of T'] {$I$};

\node[draw,dashed,fit=(I) (I'), inner xsep = 4pt] (box) {};
\end{scope}
\end{tikzpicture}
  \end{center}
  Symmetry of $\varepsilon$ by:
  \begin{center}
    \begin{tikzpicture}
\begin{scope}[on grid]
\node (S) at (0, 0) {$S$};
\node (S2) at (4, 0) {$S$};
\node (S') [below = 1 of S] {$S'$};
\node (S'2) [below = 1 of S2] {$S'$};

\draw[->] (S) -- ++(1,0) to[out = east, in = west] ++(1,-1) to[out = east, in = west] ++ (1,1) -- (S2);
\draw[->] (S') -- ++(1,0) to[out = east, in = west] ++(1,1) to[out = east, in = west] ++ (1,-1) -- (S'2);

\node (I) [below right = 1 and 1 of S'] {$I$};
\node (I') [below left = 1 and 1 of S'2] {$I$};

\node[draw,dashed,fit=(I) (I'), inner xsep = 4pt] (box) {};
\end{scope}
\end{tikzpicture}
    \qquad \raisebox{1.5cm}{$=$} \qquad
    \begin{tikzpicture}
\begin{scope}[on grid]
\node (S) at (0, 0) {$S$};
\node (S2) at (4, 0) {$S$};
\node (S') [below = 1 of S] {$S'$};
\node (S'2) [below = 1 of S2] {$S'$};

\draw[->] (S) -- (S2);
\draw[->] (S') -- (S'2);

\node (I) [below right = 1 and 1 of S'] {$I$};
\node (I') [below left = 1 and 1 of S'2] {$I$};

\node[draw,dashed,fit=(I) (I'), inner xsep = 4pt] (box) {};
\end{scope}
\end{tikzpicture}
  \end{center}
  And for monoidality of $\varepsilon$ there is essentially nothing to do in the graphical calculus:
  \begin{center}
    \begin{tikzpicture}
\begin{scope}[on grid]
\node (S) at (0, 0) {$S$};
\node (S2) at (4, 0) {$S$};
\node (T) [below = 0.7 of S] {$T$};
\node (T2) [below = 0.7 of S2] {$T$};
\node (T') [below = 0.7 of T] {$T'$};
\node (T'2) [below = 0.7 of T2] {$T'$};
\node (S') [below = 0.7 of T'] {$S'$};
\node (S'2) [below = 0.7 of T'2] {$S'$};

\draw[->] (S) -- (S2);
\draw[->] (T) -- (T2);
\draw[->] (S') -- (S'2);
\draw[->] (T') -- (T'2);

\node (I) [below right = 1 and 1 of S'] {$I$};
\node (I') [below left = 1 and 1 of S'2] {$I$};

\node[draw,dashed,fit=(I) (I'), inner xsep = 4pt] (box) {};
\end{scope}
\end{tikzpicture}
    \qquad \raisebox{2cm}{$=$} \qquad
    \begin{tikzpicture}
\begin{scope}[on grid]
\node (S) at (0, 0) {$S \otimes T$};
\node (S2) at (4, 0) {$S \otimes T$};
\node (T) [below = 0.7 of S] {$T' \otimes S'$};
\node (T2) [below = 0.7 of S2] {$T' \otimes S'$};

\draw[->] (S) -- (S2);
\draw[->] (T) -- (T2);

\node (I) [below right = 1.5 and 1 of T] {$I$};
\node (I') [below left = 1.5 and 1 of T2] {$I$};

\node[draw,dashed,fit=(I) (I'), inner xsep = 4pt] (box) {};
\end{scope}
\end{tikzpicture}
  \end{center}

  Note that the diagrams that are required to commute in the definition of teleological category all terminate with the unit $I$, so in view of Proposition~\ref{prop:costates} we should not be surprised that they correspond to equality of maps in $\C$.
\end{proof}

\begin{proposition}
  The functor $\Optic : \SymmMonCat \to \SymmMonCat$ of Theorem~\ref{thm:optic-functor} extends to a functor to $\Tele$.
\end{proposition}
\begin{proof}
 We have seen that $\Optic_\C$ is always teleological. We must show that for a symmetric monoidal functor $F : \C \to \D$, the induced functor $\Optic(F) : \Optic_\C \to \Optic_\D$ is teleological. That $\Optic(F)$ preserves the dualisable morphisms is exactly Lemma~\ref{lem:iota-commute-with-opticf}. It also preserves the counits:
  \begin{align*}
    &\Optic(F)(\varepsilon_{(S, S')})  \\
    &= \Optic(F)(c_{S \otimes S'}) && \text{(Definition of the counit)} \\
    &= \Optic(F)(\rep{\rho_S^{-1}}{\rho_{S'}}) && \text{(Definition of the $c$)} \\
    &= \rep{\phi^{-1}_{S,I} (F \rho_S^{-1})}{(F \rho_{S'}) \phi_{S',I}}&& \text{(Definition of the $\Optic(F)$)} \\ 
    &= \rep{(FS \otimes \phi_I^{-1}) \phi^{-1}_{S,I} (F \rho_S^{-1})}{(F \rho_{S'}) \phi_{S',I} (FS \otimes \phi_I)} && \text{(Introduce $\phi_I$ to both sides)} \\
    &= \rep{\rho_{FS}^{-1}}{\rho_{FS'}} && \text{($F$ is monoidal)} \\
    &= \varepsilon_{(FS, FS')} && \text{(Definition of the counit)}
  \end{align*}

%For a natural transformation $\alpha : F \Rightarrow F'$, the induced natural transformation $\Optic(\alpha) : \Optic(F) \Rightarrow \Optic(F')$ has components of the form $\iota(\alpha_{S}, \alpha^{-1}_{S'})$ that are dualisable by definition. Compatibility with the dualisation functor is immediate: for any object $(S, S')$, we have $({\Optic(\alpha)}_{(S, S')})^* = {\Optic(\alpha)}_{(S, S')^*}$ on the nose.
\end{proof}

We will establish the universal property in the somewhat contrived case of \emph{strict} symmetric monoidal categories and \emph{strict} monoidal functors, but anticipate that this result could be weakened to non-strict symmetric monoidal categories at the cost of checking far more coherences.

\begin{definition}
A teleological category is \emph{strict} if it is strict as a symmetric monoidal category and ${(-)}^*$ is a strict monoidal involution, so ${(A \teletimes B)}^* = B^* \teletimes A^*$ and $I^* = I$, and also ${(A^*)}^* = A$. A teleological functor is \emph{strict} if it is strict as a symmetric monoidal functor and strictly preserves the duality and counits.
\end{definition}

We have previously noted that $\Optic_\C$ is strict monoidal if $\C$ is, and that in that case the duality is strict. There are functors
\begin{align*}
\Optic &: \StrictSymmMonCat \to \StrictTele \\
U &: \StrictTele \to \StrictSymmMonCat \\
{(-)}_d &: \StrictTele \to \StrictSymmMonCat 
\end{align*}

The crux is the following proposition that decomposes every optic in a canonical way.

\begin{proposition}\label{prop:optic-decompose}
  Suppose $\rep{l}{r} : (S, S') \hto (A, A')$ has residual $M$. Then
  \begin{align*}
    \rep{l}{r} = ((A, I) \switched \varepsilon_{(M, I)} \switched (I, A'))(j(s_{M,A}l) \switched j{(rs_{A',M})}^*)
  \end{align*}
  where $j : \C \to \Optic_\C$ is the functor $j(A) := \iota(A, I)$.
\end{proposition}
The symmetries in the above expression could have been avoided if $\Optic$ had been defined as $\int^{M \in \C} \C(S, A \otimes M) \times \C(A' \otimes M, S')$, but it is too late to change the convention now!
\begin{proof}
  First note that because $\C$ is strict monoidal, the counit $\varepsilon_{(M, I)} : (M \otimes I, M \otimes I) = (M, M) \hto (I, I)$ is equal to the connector $c_M : (M, M) \hto (I, I)$.

  Then, up to strictness of the monoidal unit, we are composing the two optics
  
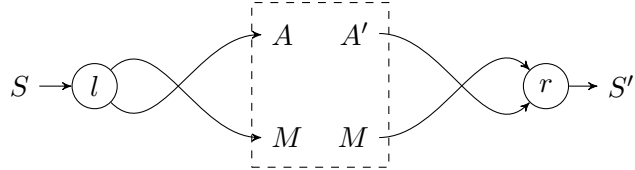
\begin{center}
    \begin{tikzpicture}

\node (l) [vert] at (0, 0) {$l$};
\node (S) [left of=l] {$S$};
\node (A) [above right = 0.2 and 2 of l] {$A$};
\node (M) [below right = 0.2 and 2 of l] {$M$};

\node (r) [vert] at (6, 0) {$r$};
\node (S') [right of=r] {$S'$};
\node (A') [above left = 0.2 and 2 of r] {$A'$};
\node (M') [below left = 0.2 and 2 of r] {$M$};

\draw[->] (S) -- (l);
\draw[->] (r) -- (S');

\draw[->] (l) to[out=south east, in=west] (A);
\draw[->] (l) to[out=north east, in=west] (M);

\draw[<-] (r) to[out=south west, in=east] (A');
\draw[<-] (r) to[out=north west, in=east] (M');

\node[draw,dashed,fit=(A) (A') (M) (M')] (box) {};
\end{tikzpicture}
  \end{center}
  and
  
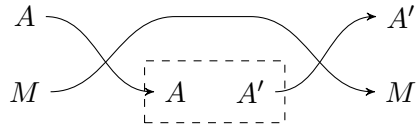
\begin{center}
    \begin{tikzpicture}
\begin{scope}[on grid]

\node (A) at (0, 0) {$A$};
\node (M) [below = 1 of A] {$M$};
\node (A') at (5, 0) {$A'$};
\node (M') [below = 1 of A'] {$M$};

\node (Aout) [below right = 1 and 2 of A] {$A$};
\node (A'out) [below left= 1 and 2 of A'] {$A'$};

\draw[->] (A) to[out=east, in=west] (Aout);
\draw[<-] (A') to[out=west, in=east] (A'out);

\draw[->] (M) to[out=east, in=west] ($(Aout) + (0,1)$)
to[out=east, in=west] ($(A'out) + (0,1)$)
to[out=east, in=west] (M');

\node[draw,dashed,fit=(Aout) (A'out) ] (box) {};

\end{scope}
\end{tikzpicture}
  \end{center}
  so the two pairs of twists cancel, and we are left exactly with the diagram for $\rep{l}{r}$.
\end{proof}

This also holds for monoidal categories that are not necessarily strict, if the unit object and unitors are inserted in the appropriate places.

\begin{proposition}
  Suppose $(\C, \otimes, I)$ is a strict symmetric monoidal category and $(\T, \teletimes, I, {(-)}^*, \varepsilon)$ is a strict teleological category. Given a strict symmetric monoidal functor $F : \C \to \T_d$, there exists a unique strict teleological functor $K : \Optic_\C \to \T$ with the property $Kj = F$.
\end{proposition}
\begin{proof}
  We construct $K$ as follows. Note that any object $(S, S')$ in $\Optic_\C$ can be written uniquely as $j(S) \switched {j(S')}^*$, so we are forced to define $K(S, S') = FS \teletimes {(FS')}^*$. Suppose $\rep{l}{r} : (S, S') \hto (A, A')$ is an optic. By the previous Proposition,
  \begin{align*}
    \rep{l}{r} = ((A, I) \switched \varepsilon_{(M, I)} \switched (I, A'))(j(s_{M,A}l) \switched {j(rs_{A', M})}^*)
  \end{align*}
  So if a $K$ with $Kj = F$ exists, it must hold that
  \begin{align*}
      K\rep{l}{r} 
      &= K((A, I) \switched \varepsilon_{(M, I)} \switched (I, A')) K(j(s_{M,A}l) \switched {j(rs_{A', M})}^*) \\
      &\qquad\text{($K$ is monoidal)} \\
      &= (K(A, I) \teletimes K\varepsilon_{(M, I)} \teletimes K(I, A')) (K(j(s_{M,A}l)) \teletimes K( {j(rs_{A', M})}^*)) \\
      &\qquad\text{($K$ preserves the counit and duality)} \\
      &= (K(A, I) \teletimes \varepsilon_{K(M,I)} \teletimes K(I, A')) (K(j(s_{M,A}l)) \teletimes K( {j(rs_{A', M})})^*) \\
      &\qquad\text{($K$ satisfies $Kj = F$)} \\
      &= (FA \teletimes \varepsilon_{FM} \teletimes {(FA')}^*) (F(s_{M,A}l) \teletimes F(rs_{A', M})^*)
  \end{align*}
  We therefore take \[  K\rep{l}{r} = (FA \teletimes \varepsilon_{FM} \teletimes {(FA')}^*) (F(s_{M,A}l) \teletimes F(rs_{A', M})^*)) \] as our definition of $K$. The diagram for $K\rep{l}{r}$ in $\T$ is as follows:
  
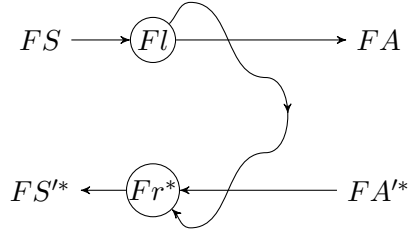
\begin{center}
    \begin{tikzpicture}
\begin{scope}[on grid]

\node[vert] (l) at (0, 0) {$Fl$};
\node[vert] (r) at (0, -2) {$Fr^*$};

\node (S) [left = 1.5 of l]{$FS$};
\node (S') [left = 1.5 of r]{$FS'^*$};

\node (A) [right =3 of l]{$FA$};
\node (A') [right = 3 of r]{$FA'^*$};

\draw[->] (S) -- (l);
\draw[->] (l) -- (A);
\draw[<-] (S') -- (r);
\draw[<-] (r) -- (A');

\draw[->-=0.5, ->] (l) 
to[out=north east, in=west] ($(l) + (0.5,0.5)$)
to[out=east, in=west] ($(l) + (1.5,-0.5)$)
to[out=east, in=east] ($(r) + (1.5,0.5)$)
to[out=west, in=east] ($(r) + (0.5,-0.5)$)
to[out=west, in=south east] (r);
\end{scope}
\end{tikzpicture}
  \end{center}
%  \begin{align*}
%    K\rep{(f \otimes A) l}{r}
%    &= (FA \otimes \varepsilon_{FM} \otimes (FA')^*)(F(s_{M,A}(f \otimes A)l) \otimes (F(rs_{A',M}))^* ) \\
%    &= (FA \otimes \varepsilon_{FM} \otimes (FA')^*)(F((A \otimes f)s_{N,A}l) \otimes (F(rs_{A',M}))^* ) \\
%    % &= (FA \otimes \varepsilon_{FM} \otimes (FA')^*)(FA \otimes Ff \otimes FM \otimes FA')(F(s_{N,A}l) \otimes (F(rs_{A',M}))^* ) \\
%    &= (FA \otimes (\varepsilon_{FM} (Ff \otimes FM)) \otimes (FA')^*)(F(s_{N,A}l) \otimes (F(rs_{A',M}))^* ) \\
%    &= (FA \otimes (\varepsilon_{FN} (FN \otimes (Ff)^*)) \otimes (FA')^*)(F(s_{N,A}l) \otimes (F(rs_{A',M}))^* ) \\
%    &= (FA \otimes \varepsilon_{FN} \otimes (FA')^*)(F(s_{N,A}l) \otimes (F(rs_{A',M}(A' \otimes f)))^* ) \\
%    &= (FA \otimes \varepsilon_{FN} \otimes (FA')^*)(F(s_{N,A}l) \otimes (F(r(f \otimes A')s_{A',N}))^* ) \\
%    &= K\rep{ l}{r (f \otimes A')}
%  \end{align*}
It remains to show that $K$ so defined is indeed a strict teleological functor. There are several things to check:
\begin{itemize}
\item Well-definedness: Suppose we have two optics related by the coend relation:
  \begin{align*}
    \rep{(f \otimes A) l}{r} = \rep{l}{r (f \otimes A')}
  \end{align*}
  Then well-definedness is shown by the equivalence of diagrams
  
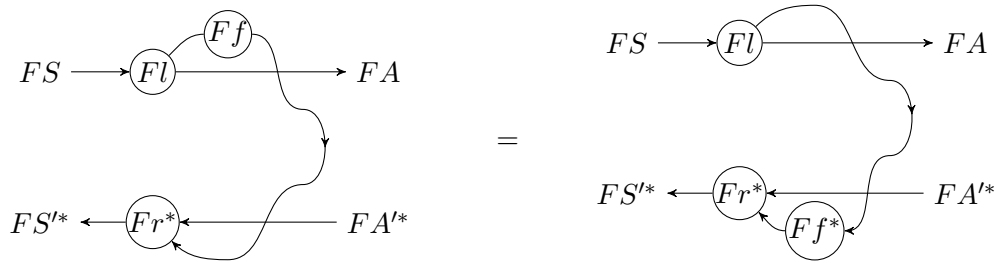
\begin{center}
    \begin{tikzpicture}
\begin{scope}[on grid]

\node[vert] (l) at (0, 0) {$Fl$};
\node[vert] (r) at (0, -2) {$Fr^*$};
\node[vert] (f) at (1, 0.5) {$Ff$};

\node (S) [left = 1.5 of l]{$FS$};
\node (S') [left = 1.5 of r]{$FS'^*$};

\node (A) [right =3 of l]{$FA$};
\node (A') [right = 3 of r]{$FA'^*$};

\draw[->] (S) -- (l);
\draw[->] (l) -- (A);
\draw[<-] (S') -- (r);
\draw[<-] (r) -- (A');

\draw (l) 
to[out=north east, in=west] (f);

\draw[->-=0.4, ->] (f) 
to[out=east, in=west] ($(f) + (1,-1)$)
to[out=east, in=east] ($(r) + (2,0.5)$)
to[out=west, in=east] ($(r) + (1,-0.5)$)
to[out=west, in=south east] (r);
\end{scope}
\end{tikzpicture}
    \qquad \raisebox{1.5cm}{$=$} \qquad
    \begin{tikzpicture}
\begin{scope}[on grid]

\node[vert] (l) at (0, 0) {$Fl$};
\node[vert] (r) at (0, -2) {$Fr^*$};
\node[vert] (fd) at (1, -2.5) {$Ff^*$};

\node (S) [left = 1.5 of l]{$FS$};
\node (S') [left = 1.5 of r]{$FS'^*$};

\node (A) [right =3 of l]{$FA$};
\node (A') [right = 3 of r]{$FA'^*$};

\draw[->] (S) -- (l);
\draw[->] (l) -- (A);
\draw[<-] (S') -- (r);
\draw[<-] (r) -- (A');

\draw[->, ->-=0.6] (l) 
to[out=north east, in=west] ($(l) + (1,0.5)$)
to[out=east, in=west] ($(l) + (2,-0.5)$)
to[out=east, in=east] ($(fd) + (1,1)$)
to[out=west, in=east] (fd);
\draw[->] (fd)
to[out=west, in=south east] (r);
\end{scope}
\end{tikzpicture}
  \end{center}
  using naturality of the symmetry and extranaturality of the counit.
\item Functoriality: We have an equivalence of diagrams
  
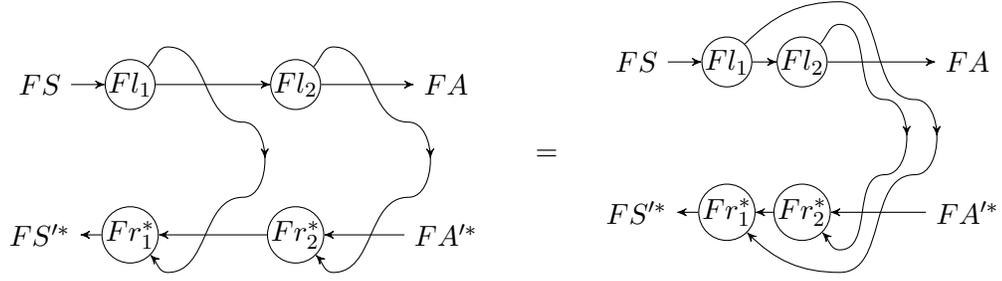
\begin{center}
    \begin{tikzpicture}
\begin{scope}[on grid]

\node[vert] (l) at (0, 0) {$Fl_1$};
\node[vert] (r) at (0, -2) {$Fr_1^*$};

\node (S) [left = 1.2 of l]{$FS$};
\node (S') [left = 1.2 of r]{$FS'^*$};

\node[vert] (l') [right = 2.2 of l] {$Fl_2$};
\node[vert] (r') [right = 2.2 of r] {$Fr_2^*$};

\node (A) [right = 2 of l']{$FA$};
\node (A') [right = 2 of r']{$FA'^*$};

\draw[->] (S) -- (l);
\draw[->] (l) -- (l');
\draw[->] (l') -- (A);
\draw[<-] (S') -- (r);
\draw[<-] (r) -- (r');
\draw[<-] (r') -- (A');

\draw[->-=0.5, ->] (l) 
to[out=north east, in=west] ($(l) + (0.5,0.5)$)
to[out=east, in=west] ($(l) + (1.5,-0.5)$)
to[out=east, in=east] ($(r) + (1.5,0.5)$)
to[out=west, in=east] ($(r) + (0.5,-0.5)$)
to[out=west, in=south east] (r);

\draw[->-=0.5, ->] (l') 
to[out=north east, in=west] ($(l') + (0.5,0.5)$)
to[out=east, in=west] ($(l') + (1.5,-0.5)$)
to[out=east, in=east] ($(r') + (1.5,0.5)$)
to[out=west, in=east] ($(r') + (0.5,-0.5)$)
to[out=west, in=south east] (r');
\end{scope}
\end{tikzpicture}
    \quad \raisebox{1.5cm}{$=$} \quad
    \begin{tikzpicture}
\begin{scope}[on grid]

\node[vert] (l) at (0, 0) {$Fl_1$};
\node[vert] (r) at (0, -2) {$Fr_1^*$};

\node (S) [left = 1.2 of l]{$FS$};
\node (S') [left = 1.2 of r]{$FS'^*$};

\node[vert] (l') [right = 1 of l] {$Fl_2$};
\node[vert] (r') [right = 1 of r] {$Fr_2^*$};

\node (A) [right = 2.2 of l']{$FA$};
\node (A') [right = 2.2 of r']{$FA'^*$};

\draw[->] (S) -- (l);
\draw[->] (l) -- (l');
\draw[->] (l') -- (A);
\draw[<-] (S') -- (r);
\draw[<-] (r) -- (r');
\draw[<-] (r') -- (A');

\draw[->-=0.5, ->] (l) 
to[out=north east, in=west] ($(l) + (1.5,0.8)$)
to[out=east, in=west] ($(l) + (2.5,-0.5)$)
to[out=east, in=east] ($(r) + (2.5,0.5)$)
to[out=west, in=east] ($(r) + (1.5,-0.8)$)
to[out=west, in=south east] (r);

\draw[->-=0.5, ->] (l') 
to[out=north east, in=west] ($(l') + (0.5,0.5)$)
to[out=east, in=west] ($(l') + (1.1,-0.5)$)
to[out=east, in=east] ($(r') + (1.1,0.5)$)
to[out=west, in=east] ($(r') + (0.5,-0.5)$)
to[out=west, in=south east] (r');
\end{scope}
\end{tikzpicture}
  \end{center}
using naturality of the symmetry and monoidality of the counit.
\item Monoidality:
\begin{align*}
K((S, S') \switched (T, T'))
&= K(S \otimes T, T' \otimes S) \\
&= F(S \otimes T) \teletimes {F(T' \otimes S')}^* \\
&= FS \teletimes FT \teletimes {(FS')}^* \teletimes {(FT')}^* \\
&= FS \teletimes {(FS')}^* \teletimes FT \teletimes {(FT')}^* \\
&= K(S, S') \teletimes K(T, T')
\end{align*}
and
\begin{align*}
K(I, I)
&= FI \teletimes {(FI)}^* \\
&= I \teletimes I^* \\
&= I
\end{align*}
%That these obey the required coherences is straightforward. %More difficult is showing that the first isomorphism is natural in $(S, S')$ and $(T, T')$: \todo{todo}
\item Preservation of duals:
\begin{align*}
K({(S, S')}^*)
= K(S', S)
= FS' \teletimes {(FS)}^*
= {(FS \teletimes {(FS')}^*)}^*
= {(K(S, S'))}^*
\end{align*}

\item Preservation of dualisable morphisms:  For a morphism $\iota(f, g)$:
\begin{align*}
K(\iota(f, g))
    &= K(\rep{\lambda_A^{-1} f}{g \lambda_{A'}}) \\
    &= (FA \teletimes \varepsilon_{FI} \teletimes {(FA')}^*)(F(s_{I,A}\lambda_A^{-1} f) \teletimes {(F(g \lambda_{A'}s_{A', I}))}^* ) \\
    &= (FA \teletimes {(FA')}^*)(Ff \teletimes {(Fg)}^* ) \\
    &= Ff \teletimes {(Fg)}^*
\end{align*}
and this is dualisable, as dualisability is preserved by taking the monoidal product and duals.
\item Preservation of counits:
\begin{align*}
K(\varepsilon_{(S, S')})
&= K(c_{S \otimes S'}) \\
&= K(\rep{\rho_{S \otimes S'}^{-1}}{\rho_{S \otimes S'}}) \\
&= (FI \teletimes \varepsilon_{F(S \otimes S')} \teletimes {(FI)}^*)(F(s_{S \otimes S',I}\rho_{S \otimes S'}^{-1}) \teletimes (F(\rho_{S \otimes S'} s_{I, S \otimes S'}))^* ) \\
&= (\varepsilon_{F(S \otimes S')})(F(S \otimes S') \teletimes F{(S \otimes S')}^* ) \\
&= \varepsilon_{F(S \otimes S')} \\
&= \varepsilon_{FS \otimes FS'} \\
&= \varepsilon_{FS}(FS \otimes \varepsilon_{FS'} \otimes {(FS)}^*) \\
&= \varepsilon_{FS}(FS \otimes \varepsilon_{{FS'}^*} \otimes {(FS)}^*) \\
&= \varepsilon_{FS \otimes {FS'}^*} \\
&= \varepsilon_{K(S, S')}
\end{align*}
The critical move is applying the equality $\varepsilon_{FS'} = \varepsilon_{{FS'}^*}$, which follows because $\varepsilon$ is a symmetric monoidal transformation and the duality is strict.
\end{itemize}

%To conclude biadjointness it remains to show that for any symmetric monoidal category $\C$ and teleological category $\T$, restriction along $j : \C \to \Optic_\C$ defines an equivalence of categories \[\Tele(\Optic_\C, \T) \simeq \SymmMonCat(\C, \T_d). \]
%The restriction of a functor $K$ along $j$ indeed has its image in the subcategory $\T_d$, as by definition $j(f) = \iota(f, \id_I)$ is dualisable and $K$ preserves dualisable morphisms. The construction of $K$ given earlier establishes that the restriction functor is essentially surjective, so 
%
%Given two teleological functors $K, L : \Optic_\C \to \T$, the whiskering of a teleological natural transformation $\alpha : K \Rightarrow L$ by $j$ is a well defined natural transformation $\beta : K j \Rightarrow L j$ in $\SymmMonCat(\C, \T_d)$ as the components of teleological natural transformations are required to be dualisable.
%
%We provide an inverse to this whiskering operation as follows. If $\beta : K j \Rightarrow L j$ is a monoidal natural transformation, define $\alpha : K \Rightarrow L$ to be the natural transformation with components:
%\begin{align*}
%K(S, S') = K(j(S) \switched {j(S')}^*) \xrightarrow{\phi} Kj(S) \teletimes K({j(S')}^*) \xrightarrow{\beta_S \teletimes \beta_{S'}^*} Lj(S) \teletimes L({j(S')}^*) \xrightarrow{\phi^{-1}} L(j(S) \switched {j(S')}^*) = L(S, S')
%\end{align*}
%\todo{Not quite right, we should be using $d_X$} We leave showing that this is an inverse to the reader.
%
%\todo{pseudonaturality of that equivalence of categories is missing}
\end{proof}

\begin{theorem}\label{thm:optic-is-free-teleological-cat}
  $\Optic : \StrictSymmMonCat \to \StrictTele$ is left adjoint to the `underlying dualisable morphisms' functor ${(-)}_d : \StrictTele \to \StrictSymmMonCat$.
\end{theorem}
\begin{proof}
Precomposition with $j$ gives a function
\begin{align*}
\StrictTele(\Optic_\C, \T) \to \StrictSymmMonCat(\C, \T_d)
\end{align*}
and the previous proposition states that this is a isomorphism. This is automatically natural in $\T$. Naturality in $\C$ follows by Lemma~\ref{lem:iota-commute-with-opticf}. 
\end{proof}

\begin{remark}
The above theorem and its proof have much in common with~\cite[Proposition 5.2]{JoyalStreetVerity}, which gave a similar universal property for their $\mathrm{Int}$ construction on traced monoidal categories.
\end{remark}

Working with strict monoidal categories made it significantly easier to prove the universal property. There is likely to be a 2-categorical universal property of $\Optic$ for non-strict monoidal categories, so long as we restrict our attention to the sub-2-category $\SymmMonCat_\mathrm{homcore}$ of $\SymmMonCat$ that only contains natural isomorphisms. We leave this to future work:

\begin{definition}
  A \emph{teleological natural isomorphism} $\alpha : F \Rightarrow G$ is a monoidal natural isomorphism whose components are all dualisable and that is additionally compatible with the dualisation:
  \[
  \begin{tikzcd}
    {(FX)}^* \ar[r, "\cong"]  \ar[d, "(\alpha_X)^*", swap] & F(X^*) \ar[d, "\alpha_{X^*}"] \\
    {(GX)}^* \ar[r, "\cong", swap] & G(X^*)
  \end{tikzcd}
  \]
  There is a (strict) 2-category $\Tele$ consisting of teleological categories, functors and natural isomorphisms.
\end{definition}

\begin{conjecture}
\[ \Optic : \SymmMonCat_\mathrm{homcore} \to \Tele \] is left biadjoint to \[(-)_d : \Tele \to \SymmMonCat_\mathrm{homcore}\]
\end{conjecture}

\subsection{Optics for a Monoidal Action}

To capture more of the optic variants available in the Haskell \lenslib{} library, we generalise to the case of a monoidal action of one category on another.

\begin{definition}
  Let $\C$ be a category and $(\M, \otimes, I)$ a monoidal category. An \emph{action of $\M$ on $\C$} is a monoidal functor $a : \M \to [\C, \C]$. For two objects $M \in \M$ and $A \in \C$, the action $a(M)(A)$ is abbreviated $M \act A$. 
\end{definition}

Given such an action, we define
\begin{align*}
  \Optic_\M((S, S'), (A, A')) := \int^{M \in \M} \C(S, M \act A) \times \C(M \act A', S')
\end{align*}
This subsumes the earlier definition, taking $\M = \C$ and having $\C$ act on itself via left-tensor:
\begin{align*}
a : \C &\to [\C, \C] \\
X &\mapsto X \otimes -
\end{align*}
We henceforth write this case as $\Optic_\otimes$, to emphasise the action on $\C$ that is used.

\begin{proposition}
  We have a category $\Optic_\M$ and a functor $\iota : \C \times \C^\op \to \Optic_\M$ defined analogously to Propositions~\ref{prop:optic-is-cat} and~\ref{prop:iota-functor}. \qed
\end{proposition}

\begin{definition}
Given two categories equipped with monoidal actions $(\M, \C)$ and $(\N, \D)$, a \emph{morphism of actions} is a monoidal functor $F^\bullet : \M \to N$ and a functor $F : \C \to \D$ that commutes with the actions, in the sense that there exists a natural isomorphism
  \begin{align*}
  \phi_{M,A} &: F(M \act A) \to (F^\bullet M) \act (F A)
  \end{align*}
satisfying conditions analogous to those for a monoidal functor.

%A morphism of actions is \emph{lax} if $\phi$ and $\phi_{M,A}$ are not necessarily invertible, and \emph{oplax} if they face the other direction. \todo{Do I need laxness in $F^\bullet$?}
\end{definition}

\begin{proposition}\label{prop:change-of-action}
If $F : (\M, \C) \to (\N, \D)$ is a morphism of actions, there is an induced functor $\Optic(F) : \Optic_\M \to \Optic_\N$. \qed
\end{proposition}

For the remainder of the paper we work in this more general setting. 

\section{Lawful Optics}\label{sec:lawful-optics}
Typically we want our optics to obey certain laws. The `constant-complement' perspective suggests declaring an optic $\rep{l}{r}$ to be lawful if $l$ and $r$ are mutual inverses. There are a couple of issues with this definition. Firstly, it is not invariant under the coend relation, so the condition holding for one representative is no guarantee that it holds for any other. Still, we might say that an optic is lawful if it has \emph{some} representative that consists of mutual inverses. In our primary example of an optic variant, lenses in $\Set$, this does indeed correspond to the concrete lens laws. However, this fact relies on some extra structure possessed by $\Set$: the existence of pullbacks, and that all objects (other than the empty set) have a global element.

In this section we make a different definition of lawfulness that at first seems strange, but which in the case of lenses corresponds \emph{exactly} to the three concrete lens laws with no additional assumptions on $\C$ required. As further justification for this definition, in Section~\ref{sec:profunctor-optics} we will see an interpretation of (unlawful) optics as maps between certain comonoid objects. Lawfulness in our sense corresponds exactly to this map being a comonoid homomorphism.

The optic laws only make sense for optics of the form $p : (S,S) \hto (A, A)$. In this section we will abbreviate $\Optic_\M((S, S), (A, A))$ as $\Optic_\M(S, A)$ and $p : (S, S) \hto (A, A)$ as $p : S \hto A$.

\begin{remark}
We use $;$ to denote composition of $\C$ in diagrammatic order. The reason for this is that the coend relation can be applied simply by shifting the position of $\mid$ in a representative:
\begin{align*}
\rep{l;(\phi \act A)}{r} = \rep{l}{(\phi \act A);r}
\end{align*}
%\todo{Should I switch the entire paper into using diagrammatic order?}
\end{remark}

Let $\Twoptic_\M(S, A)$ denote the set \[ \int^{M_1, M_2 \in \M} \C(S, M_1 \act A) \times \C(M_1 \act A, M_2 \act A) \times \C(M_2 \act A, S). \]
Using the universal property of the coend, we define three maps:
\begin{align*}
  \outside &: \Optic_\M(S, A) \to \C(S, S) \\
  \once, \twice &: \Optic_\M(S, A) \to \Twoptic_\M(S, A)
\end{align*}
by
\begin{align*}
  \outside(\rep{l}{r}) &= l;r \\
  \once(\rep{l}{r}) &= \repthree{l}{\id_{M\act A}}{r} \\
  \twice(\rep{l}{r}) &= \repthree{l}{r;l}{r}
\end{align*}

\begin{definition}
  An optic $p : S \hto A$ is \emph{lawful} if
  \begin{align*}
    \outside(p) &= \id_S \\
    \once(p) &= \twice(p)
  \end{align*}
\end{definition}

Returning to ordinary lenses, we can show that this is equivalent to the laws we expect.

\begin{proposition}\label{prop:lawful-lens-laws}
  A concrete lens described by $\fget$ and $\fput$ is lawful (in our sense) iff it obeys the three concrete lens laws.
\end{proposition}
\begin{proof}
  We begin by giving $\Twoptic_\times(S, A)$ the same treatment as we did $\Optic_\times(S, A)$. Using the universal property of the product and Yoneda reduction twice each, we have:
  \begin{align*}
  \Twoptic_\times(S, A)
  &= \int^{M_1, M_2 \in \C} \C(S, M_1 \times A) \times \C(M_1 \times A, M_2 \times A) \times \C(M_2 \times A, S) \\
  &\cong \int^{M_1, M_2 \in \C} \C(S, M_1) \times \C(S, A) \times \C(M_1 \times A, M_2 \times A) \times \C(M_2 \times A, S) \\
  &\cong \int^{M_2 \in \C} \C(S, A) \times \C(S \times A, M_2 \times A) \times \C(M_2 \times A, S) \\
  &\cong \int^{M_2 \in \C} \C(S, A) \times \C(S \times A, M_2) \times \C(S \times A, A) \times \C(M_2 \times A, S) \\
  &\cong \C(S, A) \times \C(S \times A, A) \times \C(S \times A \times A, S)
  \end{align*}
%  Call this last set $\conctwice_\times(S, A)$. Written equationally, the isomorphism $\Phi : \Twoptic_\times(S, A) \to \conctwice_\times(S, A)$ is given by:
  Written equationally, the isomorphism $\Phi : \Twoptic_\times(S, A) \to \C(S, A) \times \C(S \times A, A) \times \C(S \times A \times A, S)$ is given by:
  \begin{align*}
    \Phi(\repthree{l}{c}{r}) = (\quad&l;\pi_2, \\
                                       &(l;\pi_1 \times A);c;\pi_2, \\
                                       &((l;\pi_1 \times A);c;\pi_1 \times A);r \quad )
  \end{align*}
  Now suppose we are given a lens $p$ that corresponds concretely to $(\fget, \fput)$, so $p = \rep{[\id_S, \fget]}{\fput}$. Evaluating $\outside$ on this gives:
  \begin{align*}
    \outside(\rep{[\id_S, \fget]}{\fput}) = [\id_S, \fget];\fput 
  \end{align*}
  so requiring $\outside(p) = \id_S$ is precisely the $\fget\fput$ law.

  We now have to slog through evaluating $\Phi(\once(p))$ and $\Phi(\twice(p))$.
  \begingroup
  \allowdisplaybreaks
  \begin{alignat*}{3}
    \Phi(\once(\rep{[\id_S, \fget]}{\fput})) &=
    \Phi(&&\repthree{[\id_S, \fget]}{\id_{S \times A}}{\fput}) \\
    &= (&& [\id_S, \fget];\pi_2, \\
    &&& ( [\id_S, \fget];\pi_1 \times A);\id_{S \times A};\pi_2, \\
    &&& (( [\id_S, \fget];\pi_1 \times A);\id_{S \times A};\pi_1 \times A) ; \fput \quad) \\
    &= (&&\fget, \\
    &&& \pi_2, \\
    &&& \pi_{1,3} ;\fput \quad) \\
    %%%%
    %%%%
    \Phi(\twice(\rep{[\id_S, \fget]}{\fput})) &=
    \Phi(&&\repthree{[\id_S, \fget]}{\fput;[\id_S, \fget]}{\fput}) \\
    &= (&& [\id_S, \fget];\pi_2, \\
    &&& ([\id_S, \fget];\pi_1  \times A);\fput;[\id_S, \fget];\pi_2, \\
    &&& (([\id_S, \fget];\pi_1  \times A);\fput;[\id_S, \fget];\pi_1 \times A);\fput \quad) \\
    &= (&&\fget, \\
    &&& (\id_S \times A);\fput;\fget, \\
    &&& ((\id_S \times A);\fput  \times A);\fput \quad) \\
    &= (&&\fget, \\
    &&& \fput;\fget, \\
    &&& (\fput \times A);\fput \quad)
  \end{alignat*}
  \endgroup
  So comparing component-wise, $\Phi(\once(p))$ being equal to $\Phi(\twice(p))$ is exactly equivalent to the $\fput\fget$ and $\fput\fput$ laws holding.
\end{proof}

We can also check when some other of our basic optics are lawful.

\begin{proposition}
If $p = \rep{l}{r} : S \hto A$ is an optic such that $l$ and $r$ are mutual inverses, then $p$ is lawful.
\end{proposition}
\begin{proof}
The conditions are easy to check:
\begin{align*}
\outside(\rep{l}{r}) &= l;r = \id_S \\
\twice(\rep{l}{r})
&= \repthree{l}{r;l}{r} \\
&= \repthree{l}{\id_{M\act A}}{r} \\
&= \once(\rep{l}{r})
\end{align*}
\end{proof}

\begin{corollary}\label{cor:iota-lawful}
  If $f : S \to A$ and $g : A \to S$ are mutual inverses, then $\iota(f, g) : S \hto A$ is a lawful optic, so $\iota$ restricts to a functor $\iota : \mathrm{Core}(\C) \to \Lawful_\M$. \qed
\end{corollary}

\begin{corollary}\label{cor:tautological-lawful}
   For any two objects $A \in \C$ and $M \in \M$, the tautological optic $M \act A \hto A$ is lawful. \qed
\end{corollary}

\begin{proposition}
A costate $p : (S, S) \hto (I, I)$ corresponding to a morphism $f : S \to S$ via Proposition~\ref{prop:costates} is lawful iff $f = \id_S$.
\end{proposition}
\begin{proof}
The first law states that $\outside(p) = \id_S$, so if $p$ is lawful we have
\[  \id_S = \outside(\rep{\rho_S^{-1}}{\rho_S;f}) =  \rho_S^{-1};\rho_S;f = f \]
On the other hand, if $f  = \id_S$ then $\rep{\rho_S^{-1}}{\rho_S}$ is lawful because its components are mutual inverses.
\end{proof}

%\begin{remark}
%  Jeremy Gibbons notes (originally in the case of ordinary lenses) that the first law is equivalent to requiring that $p$ composed with the connector $c_A$ is equal to the connector $c_S$. This is an appealing description from a string diagram standpoint, but it is not clear whether there is a similar description for the second law. % \todo{Something to do with embedding $\Optic_\otimes$ into a compact closed category then using the unit to compose the lenses vertically? Should I draw a picture of this?}
%\end{remark}

\begin{proposition}\label{prop:lawful-category}
  There is an subcategory $\Lawful_\M$ of $\Optic_\M$ given by objects of the form $(S, S)$ and lawful optics between them.
\end{proposition}
\begin{proof}
  This will follow from our description of lawful profunctor optics later, but we give a direct proof. The identity optic is lawful as by definition it has a representative $\rep{\lambda_S^{-1}}{\lambda_S}$ consisting of mutual inverses. We just have to show that lawfulness is preserved under composition.

  Suppose we have two lawful optics $\rep{l}{r} : R \hto S$ and $\rep{l'}{r'} : S \hto A$ with residuals $M$ and $N$ respectively. We must show that $\rep{l;(M\act l')}{(M\act r');r}$ is also a lawful optic. Showing the first law is straightforward:
  \begin{align*}
    \outside(\rep{l; (M\act l')}{(M\act r') ; r})
    &= l ; (M \act l') ; (M \act r') ; r \\
    &= l ; (M \act l'r') ; r \\
    &= l ; (M \act \id_{N \act A}) ; r \\
    &= l ; r \\
    &= \id_S
  \end{align*}

  For the second law, we must show that
  \[ \repthree{ l;(M\act l')}{(M\act r'); r;l;(M\act l')}{(M\act r') ; r} = \repthree{l;(M\act l')}{\id_{M \act N \act A}}{(M\act r') ; r}. \]
  The idea is that, by the lawfulness of $\rep{l}{r}$ and $\rep{l'}{r'}$, there are chains of coend relations that prove
  \begin{align*}
  \repthree{l}{r;l}{r} &= \repthree{l}{\id_{M \act S}}{r} \\
  \repthree{l'}{r';l'}{r'} &= \repthree{l'}{\id_{N \act A}}{r'}
  \end{align*}
  The result is achieved by splicing these chains of relations together in the following way.

  Consider one of the generating relations $\repthree{l;(\phi \act S)}{c}{r} = \repthree{l}{(\phi \act S) ; c}{r}$ in $\Twoptic_\M(R, S)$, where $\phi : M \to M'$. By the functoriality of the action, we calculate:
  \begin{align*}
    &\repthree{l;(\phi \act S);(M' \act l')}{(M'\act r'); c ;(M\act l')}{(M\act r') ; r} \\
    &= \repthree{l;(M \act l');(\phi \act N \act A)}{(M' \act r'); c ;(M \act l')}{(M \act r') ; r} && \text{(functoriality)} \\
    &= \repthree{l;(M \act l')}{(\phi \act N \act A);(M' \act r'); c ;(M \act l')}{(M \act r') ; r} && \text{(coend relation)} \\
    &= \repthree{l;(M \act l')}{(M \act r');(\phi \cdot S); c ;(M \act l')}{(M \act r') ; r} && \text{(functoriality)}
  \end{align*}
  And similarly for the other generating relation, $\repthree{l}{c;(\phi \act S)}{r} = \repthree{l}{c}{(\phi \act S); r}$.
  
  So indeed by replicating the same chain of relations that proves $\repthree{l}{r;l}{r} =\repthree{l}{\id_{M \act S}}{r}$, we see
  \begin{align*}
    \repthree{l;(M \act l')}{(M \act r');r;l;(M \act l')}{(M \act r') ; r} 
    &= \repthree{l;(M \act l')}{(M \act r');\id_{M \act S};(M \act l')}{(M \act r') ; r} \\
    &= \repthree{l;(M \act l')}{(M \act r';l')}{(M \act r') ; r}.
  \end{align*}

  Now that the $r;l$ in the center has been cleared away, we turn to the chain of relations proving $\repthree{l'}{r';l'}{r'} = \repthree{l'}{\id_{N\act A}}{r'}$. A generating relation $\repthree{l';(\psi \act A)}{c'}{r'} = \repthree{l'}{(\psi \act A) ; c'}{r'}$ in $\Twoptic_\M(S, A)$ implies that
  \begin{align*}
    \repthree{l;(M \act l';\psi \act A)}{M \act c'}{(M \act r') ; r}
    &= \repthree{l;(M\act l');(M \act \psi \act A)}{M \act c'}{(M \act r') ; r} \\
    &= \repthree{l;(M \act l')}{(M \act \psi \act A) ; (M \act c')}{(M\act r') ; r} \\
    &= \repthree{l;(M \act l')}{M \act ((\psi \act A);c')}{(M \act r') ; r}
  \end{align*}
  And similarly for the generating relation on the other side. So again we can replicate the chain of relations proving $\repthree{l'}{r';l'}{r'} = \repthree{l'}{\id_{M\act A}}{r'}$ to show that
  \begin{align*}
    \repthree{l;(M \act l')}{(M \act r';l')}{(M \act r') ; r} 
    &= \repthree{l;(M \act l')}{M(\id_{N\act A})}{(M \act r') ; r} \\
    &= \repthree{l;(M \act l')}{\id_{M \act N \act A}}{(M \act r') ; r}
  \end{align*}
as required. We conclude that composition preserves lawfulness, so $\Lawful_\M$ is indeed a subcategory of $\Optic_\M$.
\end{proof}

\begin{proposition}
  In the case that $\C$ is symmetric monoidal and $\M = \C$ acts by left-tensor, then $\Lawful_\otimes$ is symmetric monoidal with the unswitched tensor.
\end{proposition}
This would of course make no sense with the switched tensor, as the tensor of two objects would typically no longer be of the form $(X, X)$.
\begin{proof}
Due to Corollary~\ref{cor:iota-lawful}, the structure maps of $\Optic_\M$ are all lawful. We just have to check that $\otimes : \Optic_\M \times \Optic_\M \to \Optic_\M$ restricts to a functor on $\Lawful_\M$. 

Given two lawful optics $p : S \hto A$ and $q : T \hto B$, the first law for $p \otimes q$ follows immediately from the first law for $p$ and $q$. To prove the second law, we follow the same strategy as used in the previous proposition: the two chains of relations proving $p$ and $q$  lawful can be combined to prove the law for $p \otimes q$.

\end{proof}

\begin{proposition}
  Suppose $F : (\M, \C) \to (\N, \D)$ is a morphism of actions. Then $\Optic(F) : \Optic_\M \to \Optic_\N$ restricts to a functor $\Lawful_\M \to \Lawful_\N$.
\end{proposition}
\begin{proof}
  If $p = \rep{l}{r}$ is lawful, then verifying the first equation is easy:
  \begin{align*}
  \outside(\Optic(F)(\rep{l}{r}))
  &= \outside\left(\rep{(Fl);\phi^{-1}_{M,A}}{\phi_{M,A'};(Fr)}\right) \\
  &= (Fl);\phi^{-1}_{M,A};\phi_{M,A'};(Fr)\\
  &= (Fl);(Fr)\\
  &= \id_{FS}
  \end{align*}
  where $\phi_{M,A'} : (F^\bullet M) \act (FA) \to F(M \act A)$ is the structure map that commutes $F$ with the actions.

  For the second equation, consider a generating relation $\repthree{l;(\psi \act A)}{c}{r} = \repthree{l}{(\psi \act A) ; c}{r}$. We can use the naturality of $\phi$ to show
  \begin{align*}
  \repthree{F(l;\psi \act A);\phi^{-1}_{M,A}}{\phi_{M,A};Fc;\phi^{-1}_{N,A}}{\phi_{N,A};Fr}
  &= \repthree{Fl;F(\psi \act A);\phi^{-1}_{M,A}}{\phi_{M,A};Fc;\phi^{-1}_{N,A}}{\phi_{N,A};Fr} \\
  &= \repthree{Fl;\phi^{-1}_{M',A};(F^\bullet \psi) \act A}{\phi_{M,A};Fc;\phi^{-1}_{N,A}}{\phi_{N,A};Fr} \\
  &= \repthree{Fl;\phi^{-1}_{M',A}}{(F^\bullet \psi) \act A;\phi_{M,A};Fc;\phi^{-1}_{N,A}}{\phi_{N,A};Fr} \\
  &= \repthree{Fl;\phi^{-1}_{M',A}}{\phi_{M',A};F(\psi \act A);Fc;\phi^{-1}_{N,A}}{\phi_{N,A};Fr} \\
  &= \repthree{Fl;\phi^{-1}_{M',A}}{\phi_{M',A};F(\psi \act A;c);\phi^{-1}_{N,A}}{\phi_{N,A};Fr}
  \end{align*}
  Similarly,
  \begin{align*}
    \repthree{Fl;\phi^{-1}_{M,A}}{\phi_{M,A};Fc;\phi^{-1}_{N,A}}{\phi_{N,A};F((\psi \act A);r)}
    &= \repthree{Fl;\phi^{-1}_{M,A}}{\phi_{M,A};F(c;(\psi \act A));\phi^{-1}_{N,A}}{\phi_{N,A};Fr}
  \end{align*}

  If $\rep{l}{r}$ is lawful, we can therefore replicate the chain of relations proving $\twice(\rep{l}{r}) = \once(\rep{l}{r})$ to show:
  \begin{align*}
  \twice(\Optic(F)(\rep{l}{r}))
  &= \repthree{Fl;\phi^{-1}_{M,A}}{\phi_{M,A};F(r;l);\phi^{-1}_{M,A}}{\phi_{M,A};Fr}  \\
  &= \repthree{Fl;\phi^{-1}_{M,A}}{\phi_{M,A};F(\id_{M \act A});\phi^{-1}_{M,A}}{\phi_{M,A};Fr}  \\
  &= \repthree{Fl;\phi^{-1}_{M,A}}{\id_{(F^\bullet M) \act (FA)}}{\phi_{M,A};Fr}  \\
  &= \once(\Optic(F)(\rep{l}{r}))
  \end{align*}
\end{proof}

We end with some commentary on the optic laws. The requirement that $\once(p) = \twice(p)$ is mysterious, but there are sufficient conditions that are easier to verify.

\begin{proposition}
  Let $\rep{l}{r} : S \hto A$ be an optic. If $l;r = \id_S$ and $r;l = \phi \act A$ for some $\phi : M \to M$ in $\M$, then $\rep{l}{r}$ is lawful.
\end{proposition}
\begin{proof}
  The statement $\outside(\rep{l}{r}) = l;r = \id_S$ is exactly the first law. And for the second, we verify:
  \begin{align*}
    \twice(\rep{l}{r})
    &= \repthree{l}{r;l}{r} \\
    &= \repthree{l}{\phi \act A}{r} && \text{($r;l = \phi \act A$)}\\
    &= \repthree{l ; (\phi \act A)}{\id_{M\act A}}{r} && \text{(coend relation)} \\
    &= \repthree{l;r;l}{\id_{M\act A}}{r} && \text{($r;l = \phi \act A$ again)}\\
    &= \repthree{l}{\id_{M\act A}}{r}  && \text{($l;r = \id_S$)}\\
    &= \once(\rep{l}{r})
  \end{align*}
\end{proof}

Even if $r;l = \phi \act A$ for some $\phi$, the same is not necessarily true for other representatives of the same optic. Let $\inside : \Optic_\M(S, A) \to \int^{M \in \M} \C(M \act  A, M \act  A)$ be the map induced by $\inside(\rep{l}{r}) = \langle r ; l \rangle$. We might ask that instead of requiring $r;l = \phi \act A$ exactly, we have $\langle r ; l \rangle = \langle \phi \act A \rangle$ in $\int^{M \in \M} \C(M \act  A, M \act  A)$. In fact, this is equivalent:

\begin{proposition}\label{prop:onthenose}
  Suppose $p : S \hto A$ satisfies $\outside(p) = \id_S$ and $\inside(p) = \langle \phi \act A \rangle$. Then there exists a representative $\rep{l}{r}$ such that $r;l = \psi \act A$ on the nose for some (possibly different) $\psi : M \to M$.
\end{proposition}
\begin{proof}
  The generating relation for $\int^{M \in \M} \C(M \act A, M \act  A)$ is
  \[ \langle f; (\phi \act A) \rangle = \langle (\phi \act A); f \rangle \]
  whenever $f : N \act A \to M \act A$ and $\phi : M \to N$. This relation $f; (\phi \act A) \rightsquigarrow (\phi \act A); f$ is not likely to be symmetric or transitive in general. Note that if $f; (\phi \act A) \rightsquigarrow (\phi \act A); f$ then $f ;(\phi \act A) ;f; (\phi \act A) \rightsquigarrow (\phi \act A) ;f ;(\phi \act A); f$. More generally, if $f \rightsquigarrow g$ then $f^n \rightsquigarrow g^n$ for any $n$.

  Now let $\rep{l}{r}$ be a representative for $p$, so $l;r = \id_S$ and $\langle r;l \rangle = \langle \psi \act A\rangle$. There therefore exists a finite chain of relations $r;l = u_1 \leftrightsquigarrow \dots \leftrightsquigarrow u_n = \psi \act A$. Suppose the first relation faces rightward, so there exists a $k$ and $\phi$ with $r;l = (\phi \act A);k$ and $u_2 = k;(\phi \act A)$. Define $l' = l;(\phi \act A)$ and $r' = k;r$. Then:
\begin{align*}
\rep{l'}{r'}
&= \rep{l ; (\phi \act A)}{k ; r} \\
&= \rep{l}{(\phi \act A) ; k ; r} \\
&= \rep{l}{r ; l; r} \\
&= \rep{l}{r}
\end{align*}
This new representative satisfies
\begin{align*}
l';r' &= l;(\phi \act A);k;r = l;r;l;r = \id_S \\
r';l' &= k;r;l;(\phi \act A) = k;(\phi \act A);k;(\phi \act A) = u_2^2
\end{align*}

  A symmetric argument shows that if instead the relation faces leftward, so $lr \leftsquigarrow u_2$, there again exists $l'$ and $r'$ so that $\rep{l}{r} = \rep{l'}{r'}$, and both $l';r' = \id_S$ and $r';l' = u_2^2$.

  We can now inductively apply the above argument to the shorter chain \[l'r' = u_1^2 \leftrightsquigarrow \dots \leftrightsquigarrow u_n^2 \leftrightsquigarrow {(\psi \act A)}^2 = \psi^2 \act A,\] obtained by squaring each morphism in the original chain, until we are left with a representative $\rep{l^*}{r^*}$ such that $r^*;l^* = \psi^N \act A$, for some $N>0$. This pair $\rep{l^*}{r^*}$ is the required representative.
\end{proof}

The above argument has a similar form to those that appear in~\cite{OnTheTrace}, which considered (among other things) coends of the form $\int^{c \in \C} \C(c, Fc)$ for an endofunctor $F : \C \to \C$.

\section{Examples}\label{sec:examples}

The general pattern is as follows. Once we choose a particular monoidal action $\M \to [\C, \C]$, we find an isomorphism between the set of optics $(S, S') \hto (A, A')$ and a set $\conc((S, S'), (A, A'))$ that is easier to describe. We follow~\cite{ProfunctorOptics} (and others) in calling elements of $\conc$ \emph{concrete optics}. There is no canonical choice for this set; our primary goal is to find a way to eliminate the coend so that we no longer have to deal with equivalence classes of morphisms.

Ideally, we then also find a simplified description $\conctwice(S, A)$ for the set $\Twoptic_\C(S, A)$. We can then ``read off'' what conditions on are needed on a concrete optic to ensure that the corresponding element of $\Optic_\M(S, A)$ is lawful. We will call these conditions the \emph{concrete laws}.

It is worth emphasising that once a monoidal action has been chosen and a concrete description of the corresponding optics found, no further work is needed to show that the result forms a category with a subcategory of lawful optics. This is especially useful when devising new optic variants, as we do later.

\subsection{Lenses}

The founding example, that of lenses, has already been discussed in the previous sections. We add a couple of remarks.

\begin{remark}\label{lens-iota-not-faithful}
  For the category of sets, the functor $\iota : \Set \times \Set^\op \to \Optic_\times$ is not faithful. The problem is the empty set: the functor $0 \times (-)$ is not faithful. Any pair of maps $f : 0 \to A$, $g : A' \to S'$ yield equivalent optics $\iota(f, g)$, as the corresponding $\fget$ and $\fput$ functions must be the unique maps from $0$.
\end{remark}

\begin{remark}
  In the case that $\C$ is cartesian closed, $\Optic_\times$ is monoidal closed via the astonishing formula
  \begin{align*}
    [(S, S'), (A, A')] := (\homC(S, A) \times \homC(S \times A', S'), \, S \times A')
  \end{align*}
  where $\homC(-, -)$ denotes the internal hom. For a proof see~\cite[Section 1.2]{DialecticaCategories}. This cannot be extended to non-cartesian closed categories, the isomorphism
  \begin{align*}
    \Lens((S, S') \otimes (T, T'), (A, A')) \cong \Lens((S, S'),  [(T, T'), (A, A')])
  \end{align*}
  uses the diagonal maps of $\C$ in an essential way.
\end{remark}

If we ask more of our category $\C$, we can show that a lens $S \hto A$ implies the existence of a complement $\C$ with $S \cong C \times A$.  This doesn't appear to follow purely from the concrete lens laws---an argument that a definition of lawfulness based on constant complements is not the correct generalisation. For completeness we include a proof in our notation.

\begin{proposition}[{Generalisation of~\cite[Corollary 13]{AlgebrasAndUpdateStrategies}}]
  Suppose $\C$ has pullbacks and that there is a morphism $x : 1 \to A$. If $p : S \hto A$ is lawful then there exists $C \in \C$ and mutual inverses $l : S \to C \times A$ and $r : C \times A \to S$ so that $p = \rep{l}{r}$.
\end{proposition}
\begin{proof}
  Set $C$ to be the pullback of $\fget$ along $x$, so there is a map $i : C \to S$ with $\fget \, i = x !_C$. There is also a map $j : S \to C$ induced by the following diagram:
  \[
    \begin{tikzcd}
      S \ar[ddr, bend right = 20] \ar[dr, "j", dashed] \ar[r, "{[\id_S, x !_S]}"] & S \times A \ar[dr, "\fput"] & \\
      & C \ar[r, "i"] \ar[d] \arrow[dr, phantom, "\lrcorner", very near start] & S \ar[d, "\fget"] \\
      & 1\ar[r, "x", swap] & A
    \end{tikzcd}
  \]
  which commutes by the $\fput\fget$ law. Note that $ji = \id_C$ by the universal property of pullbacks.

  Now take $l : [j,\fget] : S \to C \times A$  and $r : \fput (i \times A) : C \times A \to S$. That they are mutual inverses is easily checked:
  \begin{align*}
    \fput (i \times A)[j,\fget] &= \fput [ij,\fget] \\
                                &= \fput [\fput [\id_S, x!_S],\fget] && \text{(by definition of $j$)} \\
                                &= \fput [\id_S,\fget] && \text{(by $\fput\fput$)} \\
                                &= \id_S && \text{(by $\fget\fput$)}
  \intertext{and}
    [j,\fget]\fput (i \times A) &= [j\fput (i \times A),\fget\,\fput (i \times A)] && \text{(by universal property of product)} \\
                                &= [j\fput (i \times A), \pi_2 (i \times A)] && \text{(by $\fput\fget$)} \\
                                &= [j\fput (i \times A), \pi_2] && \\
                                &= [jij\fput (i \times A), \pi_2] && \\
                                &= [j\fput [\id_S,x !_S] \fput (i \times A), \pi_2] && \\
                                &= [j\fput [\id_S, x !_S] \pi_1 (i \times A), \pi_2] && \text{(by $\fput\fput$)}\\   
                                &= [jij \pi_1 (i \times A), \pi_2] && \\
                                &= [jiji \pi_1, \pi_2] && \\
                                &= [\pi_1, \pi_2] && \\
                                &= \id_{C \times A}
  \end{align*}

  Finally, the coend relation gives that \[\rep{[j,\fget]}{\fput (i \times A)} = \rep{(i \times A)[j,\fget]}{\fput} = \rep{[\id_S, \fget]}{\fput}\] as elements of $\Optic_\M(S, A)$.
\end{proof}

\begin{remark}
Much of the work on bidirectional transformations~\cite{CombinatorsForBidirectionalTreeTransformations} considers lenses that are only `well-behaved', not `very well-behaved': they obey the $\fput\fget$ and $\fget\fput$ laws but not the $\fput\fput$ law.

For example, the ``change counter'' lens $\bN \times A \hto A$ from~\cite{AClearPictureOfLensLaws} has $\fput$ and $\fget$ given by:
  \begin{align*}
    \fget(n, a) &= a \\
    \fput((n, a), a') &= \begin{cases}
      (n, a) & \text{if } a = a' \\
      (n+1, a') & \text{otherwise}
    \end{cases}
  \end{align*}
  This example is typical of (merely) well-behaved lenses: there is metadata stored alongside the target of a lens that mutates as the lens is used.

Lenses that satisfy only the two laws correspond to pairs $\rep{l}{r}$ such that $rl = \id_S$ and $\pi_2lr = \pi_2$. This condition seems unavoidably tied to the product structure on $\C$; there is no obvious way generalise this to other optics variants.
\end{remark}

%\begin{example}
%  A minimal example of a lens that obeys $\fget\fput$ and $\fput\fget$ but not $\fput\fput$ is the following. Let $p : \{A,B,C\} \hto \{X, Y\}$ be the lens with
%  \begin{align*}
%    \fget : \{A,B,C\} &\to \{X, Y\} \\
%    A &\mapsto X \\
%    B &\mapsto X \\
%    C &\mapsto Y
%  \end{align*}
%  \begin{align*}
%    \fput : \{A,B,C\} \times \{X, Y\} &\to \{A,B,C\} \\
%    (A,X) &\mapsto A \\
%    (B,X) &\mapsto B \\
%    (C,X) &\mapsto B \\
%    (A,Y) &\mapsto C \\
%    (B,Y) &\mapsto C \\
%    (C,Y) &\mapsto C
%  \end{align*}
%\end{example}

\subsection{Prisms}
Prisms are dual to lenses:

\begin{definition}
  Suppose $\C$ has finite coproducts. The \emph{category of prisms} is the category of optics with respect to the coproduct $\sqcup$: $\Prism \defeq \Optic_\sqcup$.
\end{definition}

Just as optics for $\times$ correspond to a pair of maps $\fget : S \to A$ and $\fput : S \times A \to S$, optics for $\sqcup$ correspond to pairs of maps $\freview : A \to S$ and $\fmatching : S \to S \sqcup A$. These names are taken from the Haskell \lenslib{} library.
\begin{align*}
  \Prism((S, S'), (A, A')) &= \int^{M \in \C} \C(S, M \sqcup A) \times \C(M \sqcup A', S') \\
                                  &\cong \int^{M \in \C} \C(S, M \sqcup A) \times \C(M, S') \times \C(A, S') && \text{(universal property of coproduct)} \\
                                  &\cong \C(S, S' \sqcup A) \times \C(A', S') && \text{(Yoneda reduction)}
\end{align*}
If we are given a prism $\rep{l}{r} : (S, S') \hto (A, A')$ then associated $\freview$ and $\fmatching$ morphisms are given by $\freview = r \inr$ and $\fmatching = (r\inl \sqcup A)l$

The concrete laws for prisms are the obvious duals to the lens laws:
\begin{align*}
  \fmatching \; \freview &= \inr \\
  [\id_S, \freview] \fmatching &= \id_S \\
  (\fmatching \sqcup A) \fmatching &= \mathrm{in}_{1,3} \, \fmatching
\end{align*}

In the \lenslib{} library documentation the third law is missing, on account of following:

\begin{proposition}
  When $\C = \Set$, the third law is implied by the other two.
\end{proposition}
\begin{proof}
  The key is that for any map $f : X \to Y$ in $\Set$, the codomain $Y$ is equal to the union of $\im f$ and its complement. The first law implies that $\freview$ is injective, so $S \cong C \sqcup A$ for some complement $C$. Identifying $A$ with its image in $S$, the second law implies that if $a\in A \subset S$ then $\fmatching(a) = \inr(a)$ and if $c\in C \subset S$ then $\fmatching(c) = \inl(c)$. The third law can then be verified pointwise by checking both cases $a\in A \subset S$ and $c\in C \subset S$ separately.
\end{proof}

The following is then exactly the dual of Proposition~\ref{prop:lawful-lens-laws}.
\begin{proposition}\label{prop:lawful-prism-laws}
  If $p : S \hto A$ is a lawful prism then the associated $\fmatching$ and $\freview$ functions satisfy the concrete prism laws. \qed
\end{proposition}

\subsection{Isos}

For any category $\C$, there is a unique action of the terminal category $1$ on $\C$ that fixes every object.

\begin{proposition}
  The category of optics for this action is isomorphic to $\C \times \C^\op$.
\end{proposition}
\begin{proof}
  \begin{align*}
    \Optic_1((S, S'), (A, A')) &= \int^{M \in 1} \C(S, M \act A) \times \C(M \act A', S') \\
                               &\cong \C(S, \star \act A) \times \C(\star \act A', S') \\
                               &\cong \C(S, A) \times \C(A', S')
  \end{align*}
  where $\star$ denotes the object of $1$. Composition in $\Optic_1((S, S'), (A, A'))$ does indeed correspond to composition in $\C \times \C^\op$.
\end{proof}

\begin{proposition}
  An iso $\rep{l}{r}$ is lawful iff (as expected) $l$ and $r$ are mutual inverses.
\end{proposition}
\begin{proof}
$\Twoptic_\M(S, A)$ specialises in this case to just
  \[ \C(S, A) \times \C(A, A) \times \C(A, S) \]
  The condition $\outside(\rep{l}{r}) = \id_S$ is the claim that $rl = \id_S$, and $\once(\rep{l}{r}) = (l, \id_A, r)$ is equal to $\twice(\rep{l}{r}) = (l, lr, r)$ iff $lr = \id_A$.
\end{proof}

\subsection{Coalgebraic Optics}\label{sec:coalgebraic}

There is a common pattern in many of the examples to follow: for every object $A \in \C$, the evaluation-at-$A$ functor $- \act A : \M \to \C$ has a right adjoint, say $R_A : \C \to \M$. To fix notation, let \[\overrightarrow{(-)} : \C(F \act A, S) \to \M(F, R_{A} S) : \overleftarrow{(-)}\] denote the homset bijection, so the unit and counit are:
    \begin{equation*}
    \begin{aligned}[t]
  	\eta_F &: F \to R_A (F \act A) \\
    \eta_F &:= \overrightarrow{\id_{F \act A}}
    \end{aligned}
    \qquad\qquad\qquad
    \begin{aligned}[t]
    \varepsilon_{S} &: (R_A S) \act A \to S \\
    \varepsilon_{S} &:= \overleftarrow{\id_{R_{A} S}}
    \end{aligned}
\end{equation*}
It is shown in~\cite[Section 6]{ANoteOnActions} that, at least in the case $\M$ is right closed, to give such an action is equivalent to giving a ``copowered $\M$-category $\C$''. In most cases of interest to us, however, $\M$ is not right closed. %\todo{Pointed out to me years ago when I emailed Ross Street, but I didn't understand what he was saying}

When we have such a right adjoint, we can always find a concrete description of an optic:
\begin{align*}
\Optic_\M((S, S'), (A, A')) &= \int^{F \in \M} \C(S, F\act A) \times \C(F\act A', S') \\
&\cong \int^{F \in \M} \C(S, F \act A) \times \M(F, R_{A'} S') \\
&\cong \C(S,  (R_{A'} S') \act A)
\end{align*}
A concrete optic is therefore a map $\funzip : S \to (R_{A'} S') \act A$. The above isomorphism sends $\funzip$ to the element $\rep{\funzip}{\varepsilon_{S'}}$. In the other direction, given $\rep{l}{r}$ we have $\funzip = (\overrightarrow{r} \act A)l$.

\begin{theorem}\label{thm:optics-are-coalgebras}
A concrete optic $\funzip : S \to  (R_{A} S) \act A$ is lawful iff it is a coalgebra for the comonad $X \mapsto (R_{A} X) \act A$.
\end{theorem}
\begin{proof}
By adjointness and Yoneda reduction we have an isomorphism \[\Phi : \Twoptic_\M(S, A) \to \C(S, R_A (R_A S \act A) \act A)\] given by
\begin{align*}
&\Twoptic_\M(S, A) \\
&= \int^{M_1, M_2 \in \M} \C(S, M_1 \act A) \times \C(M_1 \act A, M_2 \act A) \times \C(M_2 \act A, S) \\
&\cong \int^{M_1, M_2 \in \M} \C(S, M_1 \act A) \times \M(M_1, R_A (M_2 \act A)) \times \M(M_2, R_A S) \\
&\cong \int^{M_1, M_2 \in \M} \C(S, M_1 \act A) \times \M(M_1, R_A (M_2 \act A)) \times \M(M_2, R_A S) \\
&\cong \int^{M_1 \in \M} \C(S, M_1 \act A) \times \M(M_1, R_A (R_A S \act A)) \\
&\cong \C(S, R_A (R_A S \act A) \act A)
\end{align*}
  which evaluated on an element $\repthree{l}{c}{r}$ is
  \begin{align*}
    \Phi(\repthree{l}{c}{r}) &= (\overrightarrow{(\overrightarrow{r} \act A)c} \act A)l
  \end{align*}

So now interpreting the optic laws, we find
\[\outside(\rep{\funzip}{\varepsilon_{S} }) = \varepsilon_{S} \; \funzip = \id_S \] is exactly the coalgebra counit law, and equality of
  \begin{align*}
    \Phi(\once(\rep{\funzip}{\varepsilon_S}))
    &= \Phi(\repthree{\funzip}{\id_{(R_{A} S) \act A}}{\varepsilon_S }) \\
    &= (\overrightarrow{(\overrightarrow{\varepsilon_S} \act A)\id_{(R_{A} S) \act A}} \act A)\funzip \\
    &= (\overrightarrow{(\id_{R_A S} \act A)} \act A)\funzip \\
    &= (\overrightarrow{\id_{R_A S \act A}} \act A)\funzip \\
    &= (\eta_{R_A S} \act A) \funzip \\
    \Phi(\twice(\rep{\funzip}{\varepsilon_S }))
    &= \Phi(\repthree{\funzip}{\funzip \; \varepsilon_S}{\varepsilon_S }) \\
    &= (\overrightarrow{(\overrightarrow{\varepsilon_S} \act A)(\funzip \; \varepsilon_S)} \act A)\funzip \\
    &= (\overrightarrow{(\funzip \; \varepsilon_S)} \act A)\funzip \\
    &= (R_A (\funzip) \act A)\funzip
  \end{align*}
is exactly the coalgebra comultiplication law.
\end{proof}

\subsection{Setters}\label{sec:setters}

\begin{definition}
  The \emph{category of setters} $\Setter_\C$ is the category of optics for the action of $[\C, \C]$ on $\C$ by evaluation.
\end{definition}

To devise the concrete form of a setter, we use the following proposition. This is a generalisation of~\cite[Proposition 2.2]{SecondOrderFunctionals}, and helps to explain why the store comonad is so important in the theory of lenses.

\begin{proposition}
If $\C$ is powered over $\Set$ then the evaluation-at-$A$ functor $-A : [\C, \C] \to \C$ has a right adjoint given by $S \mapsto S^{\C(-, A)}$.

If $\C$ is copowered over $\Set$ then $-A : [\C, \C] \to \C$ has a left adjoint given by $S \mapsto \C(A, -) \bullet S$, where $\bullet$ denotes the copower.
\end{proposition}
\begin{proof}
For the first, we have
\begin{align*}
\C(FA, S)
&\cong \int_X \Set(\C(X, A), \C(FX, S)) \\
&\cong \int_X \C(FX, S^{\C(X, A)}) \\
&\cong [\C, \C](F, S^{\C(-, A)})
\end{align*}
and for the second,
\begin{align*}
\C(S, FA)
&\cong \int_X \Set(\C(A, X), \C(S, FX)) \\
&\cong \int_X \C(\C(A, X) \bullet S, FX) \\
&\cong [\C, \C](\C(A, -) \bullet S, F)
\end{align*}
\end{proof}

Recall that any category with coproducts is copowered over $\Set$ and any category with products is powered over $\Set$.

We could immediately use the previous section to give a coalgebraic description of setters and their laws, but with a little manipulation we get a form that looks more familiar:
\begin{align*}
  \Setter_\C((S, S'), (A, A')) &= \int^{F \in [\C, \C]} \C(S, FA) \times \C(FA', S') \\
                               &\cong \int^{F \in [\C, \C]} [\C, \C](\C(A, -) \bullet S, F) \times \C(FA', S') \\
                               &\cong \C(\C(A, A') \bullet S, S') \\
                               &\cong \Set(\C(A, A'), \C(S, S'))
\end{align*}

In the Haskell \lenslib{} library, the map $\C(A, A) \to \C(S,S)$ corresponding to a setter is called $\fover$: we think of a setter as allowing us to apply a morphism $A \to A$ over some parts of $S$. Tracing through the isomorphisms, the optic corresponding to $\fover$ is $\rep{l}{r}$ where $l : S \to \C(A, A) \bullet S$ is the inclusion with the identity morphism $\id_A$ and $r : \C(A, A') \bullet S \to S'$ is the transpose of $\fover$ along the adjunction defining the copower.

The laws for setters in this form are a kind of functoriality:

\begin{proposition}
A setter $p : S \hto A$ is lawful iff
\begin{align*}
\fover(\id_A) &= \id_S \\
\fover(f)\fover(g)&= \fover(fg)
\end{align*}
\end{proposition}
\begin{proof}
The key is concretely describing $\Twoptic_{[\C, \C]}(S, A)$ as \[ \Set( \C(A, A) \times \C(A, A), \C(S, S) ).\] We leave the verification that the conditions are equivalent to the reader. %\todo{Worth actually writing this out?}
%Because a concrete setter is a function in $\Set$, it is enough to verify the equations pointwise.
\end{proof}

This characterisation of setters is maybe a little odd, in that we have ended up with a function of $\Set$s, rather than a description internal to $\C$. If we modify our definition of $\Setter$, we do get an internal characterisation. Suppose $\C$ is cartesian closed and let $\Strong_\C$ be the category of \emph{strong functors} on $\C$.

\begin{definition}[\cite{StrongFunctors}]\label{def:strong-functor}
  A \emph{(left) strong functor} is a functor $F : \C \to \C$ equipped with a natural transformation called the \emph{strength}:
  \begin{align*}
    \theta_{A,B} : A \otimes F B \to F(A \otimes B)
  \end{align*}
  such that the strength commutes with the unitor:
  \[
    \begin{tikzcd}
      I \otimes F A \ar[r, "\theta_{1,A}"] \ar[d, "\cong" left]  & F(I \otimes A) \ar[d, "\cong" right] \\
      F A \ar[r, equals] & F A
    \end{tikzcd}
  \]
  and with associativity:
  \[
    \begin{tikzcd}
      (A \otimes B) \otimes F C \ar[rr, "\theta_{A \otimes B, C}"] \ar[d, "\alpha_{A,B,FC}" left]  && F((A \otimes B) \otimes C) \ar[d, "F\alpha_{A,B,FC}" right] \\
      A \otimes (B \otimes F C) \ar[r, "A \otimes \theta_{B,C}" below] & A \otimes F(B \otimes C) \ar[r, "\theta_{A, B\otimes C}" below] & F(A \otimes (B \otimes C))
    \end{tikzcd}
  \]
  A \emph{strong natural transformation} $\tau : (F,\theta) \Rightarrow (G,\theta')$ is a natural transformation that respects the strengths. There is an evident category $\Strong(\C)$ of strong endofunctors and strong natural transformations, and a forgetful functor $U : \Strong(\C) \to [\C, \C]$.
\end{definition}
Then, again, $\Strong_\C$ acts on $\C$ by evaluation. We leave it to the reader to verify there is a natural isomorphism \[\C(S, FA) \cong \Strong_\C(\homC(A, -) \times S, F),\] which we can use to describe optics for this action as elements of $\C(\homC(A, A'), \homC(S,S'))$.
%The use of strong functors here is due to the correspondence between tensorial strengths and ``$\C$-enrichments'' on a functor.

%\todo{This section might also work just with monoidal closed? I don't see where cartesianness was used.}

\subsection{Traversals}
In this section we work in the case $\C = \Set$. Traversals allow us to traverse through a data structure, accumulating applicative actions as we go. We begin by reviewing the definitions of applicative and traversable functors~\cite{AnInvestigationOfTheLawsOfTraversals}.

\begin{definition}
An \emph{applicative functor} $F : \C \to \C$ is a lax monoidal functor with a strength compatible with the monoidal structure, in the sense that
\[
\begin{tikzcd}[column sep = large]
A \otimes FB \otimes FC \ar[r, "{\theta_{A, B} \otimes FC}"] \ar[d, swap, "{A \otimes \phi_{B, C}}"] & F(A \otimes B) \otimes FC \ar[d, "{\phi_{A \otimes B, C}}"] \\
A \otimes F(B \otimes C) \ar[r, swap, "{\theta_{A, B \otimes C}}"] & F(A \otimes B \otimes C)
\end{tikzcd}
\]
commutes. An \emph{applicative natural transformation} is one that is both monoidal and strong. Applicative functors and natural transformations form a monoidal category $\App$ with the tensor given by functor composition.
\end{definition}

\begin{definition}
A \emph{traversable functor} is a functor $T : \C \to \C$ equipped with a distributive law $\delta_F : TF \to FT$ for $T$ over the action of $\App$ on $\C$ by evaluation.

Explicitly, this means that the diagrams
\[
  \begin{tikzcd}
    TF \ar[r, "\delta_F"] \ar[d, swap, "T\alpha"] & FT \ar[d, "T\alpha"] \\
    TG \ar[r, swap, "\delta_G"] & GT
  \end{tikzcd} \hspace{1cm}
  \begin{tikzcd}
    TFG \ar[dr, swap, "\delta_F G"] \ar[rr, "\delta_{FG}"] &  & FGT \\
    & FTG \ar[ur, swap, "F \delta_G"] &
  \end{tikzcd} \hspace{1cm}
  \begin{tikzcd}
    T\id_\C \ar[r, bend left, "\id_T"] \ar[r, bend right, swap, "\delta_{\id_\C}"] & \id_\C T
  \end{tikzcd}
\]
in $[\C, \C]$ commute.
\end{definition}

\begin{definition}
The category $\Traversal$ of traversals is the category of optics for the action of $\Traversable$ on $\Set$ given by evaluation. (Yes, the names $\Traversal$/$\Traversable$ are confusing!)
\end{definition}

%\todo{The rest of this section should probably be rewritten to use the free applicative functor instead of the confusing parameterised comonad business}

It is known that traversable functors correspond to coalgebras for a particular parameterised comonad. See~{\cite[Definitions 4.1 and 4.2]{SecondOrderFunctionals}}, also~\cite{AlgebrasForParameterisedMonads} for the relevant definitions of parameterised comonads and coalgebras.

\begin{proposition}[{\cite[Theorem 4.10, Proposition 5.4]{SecondOrderFunctionals}}]
Traversable structures on a functor $T : \Set \to \Set$ correspond to parameterised coalgebra structures
\begin{align*}
t_{A, B} : TA \to UR^*_{A, B}(T B)
\end{align*}
where $UR^*_{X,Y}$ is the parameterised comonad
\begin{align*}
UR^*_{X, Y} Z = \Sigma_{n\in \bN} X^n \times \Set(Y^n,Z)
\end{align*}
Moreover, this correspondence forms an isomorphism of categories between $\Traversable$ and the Eilenberg-Moore category of coalgebras for $UR^*_{-, -}$, which we denote $\E$. \qed
\end{proposition}

% Experts will recognise $UR^*_{-, -}$ as the free applicative functor \cite{FreeApplicativeFunctors} on the functor $R_{X,Y} Z = X \times Y \to Z$.

\begin{lemma}
  For any objects $A, B \in \Set$ and traversable functor $F$, \[\Set(FA, B) \cong \Traversable(F, \Sigma_n {(-)}^n \times \Set(A^n,B))\]
naturally in $B$ and $F$. In other words, the functor \[(B \mapsto \Sigma_n {(-)}^n \times \Set(A^n,B)) : \Set \to \Traversable\] is right adjoint to the evaluation-at-$A$ functor $-A : \Traversable \to \Set$.
\end{lemma}
\begin{proof}
By~\cite[Proposition 6]{AlgebrasForParameterisedMonads}, there is a parameterised adjunction $L_T \dashv R_T$, where
\begin{align*}
L_T : \Set \times \E &\to \Set \\
(X, (F, f)) &\mapsto FX \\
R_T : \Set \times \Set &\to \E \\
(Y, Z) &\mapsto (UR^*_{-, Y} Z, \varepsilon)
\end{align*}
where $\varepsilon$ is the counit of $UR^*_{-, -}$. Evaluating these with the fixed parameter $A$, we get an ordinary adjunction
\begin{align*}
L_T(A) : \E &\to \Set \\
(F, f) &\mapsto FA \\
R_T(A) : \Set &\to \E \\
Z &\mapsto (UR^*_{-, A} Z, \varepsilon)
\end{align*}
But this is exactly the adjunction we were trying to show.
\end{proof}

We can then use the coalgebraic pattern from earlier to reach the same concrete description of traversals as found in~\cite{ProfunctorOptics}.
\begin{align*}
\Traversal((S, S'), (A, A')) &\cong \Set(S, \Sigma_n A^n \times \Set(A'^n,S'))
\end{align*}

The concrete laws for this representation are the coalgebra laws. These laws, however, are not the ones usually presented for traversals. Instead, versions of the profunctor laws are used, see Section~\ref{sec:profunctor-optics}.

%\begin{remark}
%In~\cite[Section 2.3]{ProfunctorOptics}, it is further claimed that a traversal $S \hto A$ exhibits an isomorphism $S \cong \Sigma_n A^n \times \Set(A^n,S)$. This cannot possibly be true---consider the traversable functor $X \times -$. The claim appears to be a misreading of~\cite[Proposition 5.4]{SecondOrderFunctionals}.
%\end{remark}

%\begin{remark}
%A careful analysis would be needed to describe traversals in some other category, in particular, the description of $UR^*$ does not make sense if $\C$ is not locally cartesian closed.
%\end{remark}

%\begin{remark}
%Traversable functors can be described as a particular class of polynomial functors known as finitary containers. It may be possible to generalise this to other classes of polynomial functors.
%\end{remark}

\subsection{Polymorphic Optics}
Haskell's optics allow \emph{polymorphic updates}, where the type of the codomain of the lens may be changed by an update, causing a corresponding change in the type of the domain. As an example, we permitted to use a lens into the \mintinline{haskell}{first} entry of a tuple in the following way:
\begin{minted}{haskell}
set first (1, 5) "hello" == ("hello", 5)
\end{minted}
This has changed the type from \mintinline{haskell}{(Int, Int)} to \mintinline{haskell}{(String, Int)}.

Polymorphic optics can be captured by the coend formalism as follows. Any action of a monoidal category $\M \times \C \to \C$ can be extended to act object-wise on a functor category:
\begin{align*}
  \M \times [\D, \C] &\to [\D, \C] \\
  (M, F) &\mapsto  M \act (F-)
\end{align*}

So in the above example, we have the product $\times$ acting pointwise on the functor category $\Set \to \Set$. Our example \mintinline{haskell}{first} is then an optic $F \hto G$, where $F = (-) \times \mintinline{haskell}{Int}$ and $G$ is the identity functor.

Given such a polymorphic optic in $[\D, \C]$, we can always `monomorphise' to obtain an ordinary optic in $\C$.
\begin{proposition}
  There is a functor
  \begin{align*}
    \mathsf{mono} : \D \times \D^\op \times \Optic_{[\D, \C]} \to \Optic_\C
  \end{align*}
  that sends an object $(D, D') \in \D \times \D^\op$ and optic $\rep{l}{r} : (F, F') \hto (G, G')$ in $\Optic_{[\D, \C]}$ to the optic $\rep{l_D}{r_{D'}} : (FD, F'D') \hto (GD, G'D')$ in $\Optic_\C$. For fixed $(D, D) \in \D \times \D^\op$, this functor preserves lawfulness.
\end{proposition}
\begin{proof}
  On an object $(D, D') \in \D \times \D^\op$, that we get a functor $\Optic_{[\D, \C]} \to \Optic_\C$ is essentially the same proof as Proposition~\ref{prop:change-of-action} but with different functors on each side of the lens: the evaluation-at-$D$ functor $[\D, \C] \to \C$ on the left and evaluation-at-$D'$ on the right.

  For functoriality in $\D \times \D^\op$, given $(f, g) : (D_1, D'_1) \to (D_2, D'_2) \in \D \times \D^\op$ and an object $(F, F') \in \Optic_{[\D, \C]}$, there is an induced lens $\iota(Ff, F'g) : (FD_1, F'D'_1) \hto (FD_2, F'D'_2)$. Bifunctoriality of $\mathsf{mono}$ is ensured by the naturality of each $l$ and $r$ in the morphisms of $\Optic_{[\D, \C]}$.
\end{proof}

\subsection{Linear Lenses}\label{sec:linear-lenses}
\newcommand{\ev}{\mathsf{ev}}
\newcommand{\coev}{\mathsf{coev}}

If $\C$ is closed monoidal but not necessarily cartesian, we can still define the category of \emph{linear lenses} to be $\Optic_\otimes$. The internal hom provides a right adjoint to the evaluation-at-$A$ functor, so we have immediately
\begin{align*}
  \Optic_\otimes((S, S'), (A, A')) &\cong \C(S, \homC(A',S') \otimes A)
\end{align*}
where $\homC(A', S')$ denotes the internal hom. If $\C$ is cartesian, this is of course isomorphic to the set of $(\fget, \fput)$ functions discussed earlier.

We cannot possibly use the three $\fput$/$\fget$ style lens laws in this setting as we lack projections, but specialising the coalgebra laws gives us:

\begin{proposition}\label{prop:concrete-linear-lawful}
  A linear lens $p : S \hto A$ is lawful iff the following two concrete laws for $\funzip$ hold:
  \begin{align*}
    \ev_{A, S} \; \funzip &= \id_S && \textsc{(Rezip)} \\
    (\coev_{\homC(A, S), A} \otimes A)\funzip &= ((\funzip \circ -) \otimes A)\funzip && \textsc{(ZipZip)}
  \end{align*}
  where \[ \funzip \circ - : \homC(A, S) \to \homC(A, \homC(A, S) \otimes A) \] denotes internal composition and \[\coev_{\homC(A, S), A} : \homC(A, S) \to \homC(A, \homC(A, S) \otimes A)\] is coevaluation.
  %\todo{I obviously need better names for the laws} 
  \qed
\end{proposition}

We have essentially rederived the result given in~\cite[Section 3.2]{RelatingAlgebraicAndCoalgebraic} for ordinary lenses, but we note that cartesianness was not required.

\subsection{Effectful Optics}
\newcommand{\monact}{\rtimes}

Many proposed definitions of effectful lenses~\cite{ReflectionsOnMonadicLenses} have modified one or both of $\fget$ and $\fput$ to produce results wrapped in a monadic action. There are disadvantages to this approach: it is not obvious what the laws ought to be and there is no clear generalisation to other optic variants. The general definition of optic given in Section~\ref{sec:optics} suggests we instead work with the Kleisli category $\C_T$ of some monad $(T, \eta, \mu) : \C \to \C$.

\begin{definition}
The Kleisli category $\C_T$ of a monad $T$ has the same objects as $\C$, with morphisms $X \to Y$ in $\C_T$ given by morphisms $X \to TY$ in $\C$. Identity morphisms are given by the unit of $T$, and the composite of two morphisms $f : X \to Y$ and $g : Y \to Z$ in $\C_T$ is given by
\begin{align*}
X \xrightarrow{f} TY \xrightarrow{Tg} TTZ \xrightarrow{\mu_Z} TZ
\end{align*}

For $f : X \to Y$ in $\C_T$, we write $\underline{f} : X \to TY$ for its underlying morphism in $\C$.
\end{definition}

Working in a Kleisli category presents its own set of difficulties. The product in $\C$ is a monoidal product in a $\C_T$ only when the monad in question is \emph{commutative}, which rules out many monads of interest. A premonoidal structure~\cite{PremonoidalCategories} is not sufficient: composition of optics would in that case not be well defined.

But this does not preclude the existence of monoidal actions on $\C_T$. In fact, there is a monoidal action that has long been used under a different guise:

\begin{definition}[{\cite{NotionsOfComputationAndMonads}}]
A \emph{strong monad} $T : \C \to \C$ on a monoidal category $(\C, \otimes, I)$ is a monad that is strong as a functor (Definition~\ref{def:strong-functor}), and such that the strength commutes with the unit and multiplication:
\[
  \begin{tikzcd}
    A \otimes B \ar[d, swap, "A \times \eta_B"] \ar[dr, "\eta_{A \times B}"] & \\
    A \otimes TB \ar[r, swap, "\theta_{A, B}"] & T(A \otimes B)
  \end{tikzcd} \hspace{1cm}
  \begin{tikzcd}
    A \otimes T^2 B \ar[r, "\theta_{A, TB}"] \ar[d, swap, "A \otimes \mu_B"] & T(A \otimes TB) \ar[r, "T\theta_{A, B}"] & T^2(A \otimes B) \ar[d, "\mu_{A \otimes B}"] \\
    A \otimes TB \ar[rr, swap, "\theta_{A, B}"] & & T(A \times B)
  \end{tikzcd}
\]
\end{definition}

\begin{proposition}
If $T : \C \to \C$ is a strong monad then $\C$ acts on $\C_T$ by $X \act Y := X \otimes Y$.
\end{proposition}

The crucial difference between this and a monoidal structure on $\C_T$ is that we only demand $X$ be functorial with respect to \emph{pure functions} in $\C$, whereas $Y$ must be functorial with respect to \emph{computations} in $\C_T$. We will write this action as $X \monact Y$ to highlight the different roles played by $X$ and $Y$.

\begin{proof}
Suppose $T$ is a strong monad with strength $\theta_{A, B} : A \otimes T B \to T(A \otimes B)$. For $A \in \C$, we have a functor $A \otimes - : \C_T \to \C_T$ which on a morphism $f : X \to Y$ in $\C_T$ is defined to be the composite
\begin{align*}
A \otimes X \xrightarrow{A \otimes \underline{f}} A \otimes TY \xrightarrow{\theta_{A, Y}} T(A \otimes Y)
\end{align*}
For details, see~\cite[Theorem 4.2]{PremonoidalCategories}. Our goal is to show this extends to a monoidal functor $a : \C \to [\C_T, \C_T]$.

A morphism $f : A \to B$ in $\C$ induces a natural transformation $A \otimes - \Rightarrow B \otimes -$ of functors $\C_T \to \C_T$, with components $A \otimes X \to T(B \otimes X)$ given by composing $A \otimes X \to B \otimes X$ with the unit of the monad. Naturality follows by the naturality of the strength and the unit of $T$.
%Naturality is easy to check:
%\[\begin{tikzcd}
%A \otimes X \ar[r] \ar[d] & B \otimes X \ar[r] \ar[d] & T(B \otimes X) \ar[d] \\
%A \otimes TY \ar[r] \ar[d] & B \otimes TY \ar[r] \ar[d] & T(B \otimes TY) \ar[d] \\
%T(A \otimes Y) \ar[r] & T(B \otimes Y) \ar[r] & TT(B \otimes Y)
%\end{tikzcd}\]
%The upper left square commutes by functoriality of $\otimes$, lower left by naturality of the strength, the two right squares by naturality of the unit.

Monoidality of $a$ is shown exactly by the commutative diagrams in the definition of strong functor, i.e.\ that the strength commutes with associator and left unitor of $\C$.

%\todo{I suspect there is a 1-to-1 correspondence between strengths on $T$ and actions of $\C$ on $\C_T$ by $A \otimes B$ on objects, is it worth proving this?}
\end{proof}

Suppose $\C$ is a monoidal closed category and $T : \C \to \C$ is a strong monad. Then the evaluation-at-$A$ functor has a right adjoint:
\begin{align*}
\C_T(M \monact A', S')
&= \C(M \times A', T S') \\
&\cong \C(M, \homC(A', T S'))
\end{align*}
Using the coalgebraic description, we see that concrete effectful lenses consist of a single morphism in $\C$ \[\munzip : S \to T(\homC(A', T S') \otimes A).\] The optic laws in this case specialise to:

\begin{proposition}
A concrete effectful lens is lawful iff
  \begin{align*}
    \mu_S T(\ev_{A, TS}) \; \munzip &= \eta_S \\
    T(\eta_{\homC(A, TS) \otimes A}\coev_{\homC(A, TS), A} \otimes A)\munzip &= T((\munzip \circ_T -) \otimes A)\munzip
  \end{align*}
  where \[ \munzip \circ_T - : \homC(A, TS) \to \homC(A, T(\homC(A, TS) \otimes A)) \] denotes internal Kleisli composition and \[\coev_{\homC(A, TS), A} : \homC(A, TS) \to \homC(A, \homC(A, TS) \otimes A) \] is coevaluation. \qed
\end{proposition}
  Or, if you prefer do-notation, the two laws are:
\begin{multicols}{2}
\begin{minted}{haskell}
do (c, a) <- munzip s
   c a
==
return s
\end{minted}
~\columnbreak
\begin{minted}{haskell}
do (c, a) <- munzip s
   let f a' = do
     s' <- c a'
     munzip s'
   return (f, a)
==
do (c, a) <- munzip s
   let f a' = (c, a')
   return (f, a)
\end{minted}
\end{multicols}

The inclusion of $\C$ into $\C_T$ preserves the action of $\C$, so there is an induced inclusion $\Optic_\otimes \to \Optic_\monact$.

If we choose a specific monad, we can hope to simplify the description of a concrete effectful optic and its laws.

%\subsubsection{Partial Lenses}
%
%The simplest nontrivial monad we could try is the maybe monad.
%
%\begin{definition}
%The \emph{partiality} or \emph{maybe monad} is defined by
%\begin{align*}
%T X = X \sqcup 1
%\end{align*}
%with unit $\eta_X : X \to X \sqcup 1$ the inclusion, multiplication $\mu_X : (X \sqcup 1) \sqcup 1 \to X \sqcup 1$ given by the fold map $1 \sqcup 1 \to 1$, and strength $\theta_{A, B}$ given by the composite
%\[
%A \times (B \sqcup 1) \to (A \times B) \sqcup (A \times 1) \to (A \times B) \sqcup 1
%\]
%using the unique map $A \times 1 \to 1$.
%\end{definition}
%

\subsubsection{Writer Lenses}
We begin with a simple example. Suppose $\C$ has finite products.
\begin{definition}
The \emph{writer monad} for a monoid $W$ is defined by
\begin{align*}
T_W X = X \times W
\end{align*}
The unit, multiplication of $T_W$ are given by pairing with the unit and multiplication of $W$, and the strength is simply the associativity morphism.
\end{definition}

We can find a more explicit description of concrete effectful lenses for this monad.
\begin{align*}
\Optic((S, S'), (A, A'))
&= \int^{M \in \C} \C_{T_W}(S, M \monact A) \times \C_{T_W}(M \monact A', S') \\
&= \int^{M \in \C} \C(S, M \times A \times W) \times \C_{T_W}(M \monact A', S') \\
&\cong \int^{M \in \C} \C(S, M) \times \C(S, A\times W) \times \C_{T_W}(M \monact A', S') \\
&\cong \C_{T_W}(S, A) \times \C_{T_W}(S \times A', S')
\end{align*}
Fortunately, concrete writer lenses correspond to $\fget$ and $\fput$ functions in the Kleisli category of $T_W$.

\subsubsection{Stateful Lenses}

Suppose $\C$ is cartesian closed.

\begin{definition}
The \emph{state monad} with state $Q$ is defined by
\begin{align*}
T_Q X = \homC(Q, X \times Q)
\end{align*}
%This is the monad for the tensor-hom adjunction. There are a pair of useful morphisms in the Kleisli category:
%\begin{align*}
%\textsc{GetState} &: 1 \to Q \\
%\textsc{PutState} &: Q \to 1
%\end{align*}
%where $\textsc{GetState}$ is the transpose of the diagonal map $1 \times Q \cong Q \to Q \times Q$, and $\textsc{PutState}$ is the transpose of the second projection $Q \times Q \to Q \cong 1 \times Q$.
\end{definition}

We call optics for the action $\monact : \C \times \C_{T_Q} \to \C_{T_Q}$ \emph{stateful lenses}. We can find a concrete description that is closer to that for ordinary lenses:
\begin{align*}
\Optic_\monact((S, S'), (A, A'))
&= \int^{M \in \C} \C_{T_Q}(S, M \monact A) \times \C_{T_Q}(M \monact A', S') \\
&= \int^{M \in \C} \C(S, \homC(Q, M \times A \times Q)) \times \C_{T_Q}(M \monact A', S') \\
&\cong \int^{M \in \C} \C(S \times Q, M \times A\times Q) \times \C_{T_Q}(M \monact A', S') \\
&\cong \int^{M \in \C} \C(S \times Q, M) \times \C(S \times Q, A\times Q) \times \C_{T_Q}(M \monact A', S') \\
&\cong \C(S \times Q, A \times Q) \times \C_{T_Q}((S \times Q) \monact A', S') \\
&\cong \C_{T_Q}(S, A) \times \C_{T_Q}(S \times Q \times A', S')
\end{align*}
By analogy with ordinary lenses, let us call these maps $\mget$ and $\mput$. 
The induced composition of effectful lenses is a little intricate, and is possibly best explained in code. The composite $\mget$ is straightforward, just the composite of $\mget_1$ and $\mget_2$ in the Kleisli category. For $\mput$ however, there is some curious plumbing of the state into different places. Tracing through the isomorphism, two stateful lenses $(\mget_1, \mput_1) : (T, T') \hto (S, S')$ and $(\mget_2, \mput_2) : (S, S') \hto (A, A')$ compose as follows. 
\begin{minted}{haskell}
mget t = do
  s <- mget1 t
  mget2 s

mput t q a = do 
  start <- getState
  s <- mget1 t
  q' <- getState
  putState start

  s' <- mput1 s q' a
  mput2 t q s'
\end{minted}

\begin{proposition}
A stateful lens given by
\begin{minted}{haskell}
mget :: s -> State q a
mput :: s -> q -> a -> State q s
\end{minted}
is lawful iff the following three laws hold: \\
\begin{minipage}{\textwidth}
\begin{multicols}{3}
\begin{minted}{haskell}
do
  q <- getState
  a <- mget s
  mput s q a
==
return s
\end{minted}
~\columnbreak
\begin{minted}{haskell}
do s' <- mput s q a
   mget s'
==
return a
\end{minted}
~\columnbreak
\begin{minted}{haskell}
let (s', q')
  = runState (mput s q1 a1) q2
in mput s' q' a2
==
mput s q1 a2
\end{minted}
\end{multicols}
\end{minipage}
By analogy we call these the $\fget\fput$, $\fput\fget$ and $\fput\fput$ laws.
\end{proposition}

%\begin{proof}
%We similarly calculate
%\begin{align*}
%&\int^{M, N \in \C} \C_{T_Q}(S, M \monact A) \times \C_{T_Q}(M \monact A, N \monact A) \times \C_{T_Q}(N \monact A, S) \\
%&\cong \int^{M, N \in \C} \C(S, \homC(Q, M \times A \times Q)) \times \C(M \times A, \homC(Q, N \times A \times Q)) \times \C_{T_Q}(N \monact A, S) \\
%&\cong \int^{M, N \in \C} \C(S \times Q, M \times A \times Q) \times \C(M \times A \times Q, N \times A \times Q) \times \C_{T_Q}(N \monact A, S) \\
%&\cong \int^{M, N \in \C} \C(S \times Q, M) \times \C(S \times Q,  A \times Q) \times \C(M \times A \times Q, N \times A \times Q) \times \C_{T_Q}(N \monact A, S) \\
%&\cong \int^{N \in \C} \C_{T_Q}(S, A) \times \C(S \times Q \times A \times Q, N \times A \times Q) \times \C_{T_Q}(N \monact A, S) \\
%&\cong \int^{N \in \C} \C_{T_Q}(S, A) \times \C(S \times Q \times A \times Q, N) \times \C(S \times Q \times A \times Q, A \times Q) \times \C_{T_Q}(N \monact A, S) \\
%&\cong \C_{T_Q}(S, A)\times \C_{T_Q}((S \times Q) \monact A, A) \times \C_{T_Q}((S \times Q \times A \times Q) \monact A, S)
%\end{align*}
%\todo{This is a mess}
%\end{proof}

Of course, this notion of effectful lens may not be useful! It is hard to get intuition for the meaning of the laws, but they seem to suffer from the same deficiency that other attempts at effectful lenses do: they are too strong. The $\fget\fput$ law here appears easier to satisfy than the $\mathsf{MGetPut_0}$ law of~\cite{ReflectionsOnMonadicLenses}, as $\mput$ is given access to the original state. However, our $\fput\fget$ law seems very restrictive: no matter what auxiliary state is provided, $\fput$ting then $\fget$ting must leave the state unchanged.

\subsection{Further Examples}
The dedicated reader may enjoy deriving the concrete representation and laws for the following optic varieties:

\begin{itemize}
\item \emph{``Achromatic'' Lenses}~\cite[Section 5.2]{ProfunctorOpticsThesis} are lenses that also admit an operation $\fcreate : A \to S$. These are optics for the action of $\C$ on itself by $M \act A = (M \sqcup 1) \times A$, or equivalently, of the category of pointed objects of $\C$ on $\C$ by cartesian product. Concrete achromatic lenses $(S, S') \hto (A, A')$ are elements of the set \[\C(S, \homC(A', S') \sqcup 1) \times \C(S, A) \times \C(A', S').\]
%\begin{align*}
%  \Optic((S, S'), (A, A'))
%  &= \int^{M \in \C} \C(S, M \act A) \times \C(M \act A', S') \\
%  &= \int^{M \in \C} \C(S, (M \sqcup 1) \times A) \times \C((M \sqcup 1) \times A', S') \\
%  &\cong \int^{M \in \C} \C(S, (M \sqcup 1) \times A) \times \C((M \times A') \sqcup A', S') \\
%  &\cong \int^{M \in \C} \C(S, (M \sqcup 1) \times A) \times \C(M \times A', S') \times \C(A', S') \\
%  &\cong \int^{M \in \C} \C(S, (M \sqcup 1) \times A) \times \C(M, \homC(A', S')) \times \C(A', S') \\
%  &\cong \C(S, (\homC(A', S') \sqcup 1) \times A) \times \C(A', S') \\
%  &\cong \C(S, \homC(A', S') \sqcup 1) \times \C(S, A) \times \C(A', S')
%\end{align*}
\item \emph{Affine Traversals}~\cite{SecondOrderFunctionals} allow access to a target that may or may not be present. Suppose $\C$ is cartesian closed and has binary coproducts. Let $\mathsf{Aff}$ be the category $\C \times \C$, equipped with the monoidal structure
\begin{align*}
  (P', Q') \otimes (P, Q) &= (P' \sqcup (Q' \times P) , Q' \times Q)
\end{align*}
The category $\mathsf{Aff}$ acts on $\C$ by $(P, Q) \act A = P \sqcup (Q \times A)$, in fact, $\mathsf{Aff}$ is cooked up to act on $\C$ exactly by the closure of the actions $- \times A$ and $- \sqcup A$ under composition. A concrete affine traversal is an element of \[\C(S, S' \sqcup (\homC(A', S') \times A)).\]

Affine traversals are described in the folklore as pairs of maps $\C(S, A \sqcup S') \times \C(S\times A', S')$. Such a pair does determine an affine traversal, but gives more information than is necessary: the right-hand map need not be defined at all $S$.

%\begin{align*}
%  (P', Q') \act (P, Q) \act A
%  &= (P', Q') \act  (P \sqcup (Q \times A)) \\
%  &= P' \sqcup (Q' \times (P \sqcup (Q \times A))) \\
%  &\cong P' \sqcup (Q' \times P) \sqcup (Q' \times Q \times A) \\
%  &= (P' \sqcup (Q' \times P) , Q' \times Q) \act A \\
%  &= ((P', Q') \otimes (P, Q)) \act A
%\end{align*}
%\begin{remark}
%  It is important here that the morphisms in $\mathsf{Aff}$ are only those that arise from pairs of morphisms $P \to P'$ and $Q \to Q'$, although in principle there may be other natural transformations between the corresponding functors $(P, Q) \act -$ and $(P', Q') \act -$.
%\end{remark}
%\todo{This feels similar to taking some sort of `compositum' of the two actions $\times$ and $\sqcup$, both embed in this category. Marco suggests taking the pushout of the projections into the pullback, calculated in the 2-cat of monoidal categories.}
%Now the set of optics $(S, S') \hto (A, A')$ is:
%\begin{align*}
%  \Optic_{\mathsf{Aff}}((S, S'), (A, A'))
%  &= \int^{M \in \mathsf{Aff}} \C(S, M \act A) \times \C(M \act A', S') \\
%  &\cong \int^{P,Q \in \C} \C(S, (P,Q) \act A) \times \C((P,Q) \act A', S') \\
%  &= \int^{P,Q \in \C} \C(S, P \sqcup (Q \times A)) \times \C(P \sqcup (Q \times A'), S') \\
%  &\cong \int^{P,Q \in \C} \C(S, P \sqcup (Q \times A)) \times \C(P,S') \times \C(Q \times A', S') \\
%  &\cong \int^{Q \in \C} \C(S, S' \sqcup (Q \times A)) \times \C(Q \times A', S') \\
%  &\cong \int^{Q \in \C} \C(S, S' \sqcup (Q \times A)) \times \C(Q, \homC(A', S')) \\
%  &\cong \C(S, S' \sqcup (\homC(A', S') \times A))
%\end{align*}

\item \emph{Grates}~\cite{GratesPost} are optics for the contravariant action of a monoidal closed category $\C$ on itself by $X \act A \mapsto \homC(X, A)$. Concretely these correspond to morphisms \[ \C(\homC(\homC(S, A), A'), S'). \]
\end{itemize}

\section{The Profunctor Encoding}\label{sec:profunctor-optics}
To use optics in practice, one could take the definition of the optic category and translate it almost verbatim into code---using an existential type in place of the coend. In Haskell syntax, lenses would be defined as:
\begin{minted}{haskell}
data Lens s s' a a' = forall m. Lens {
  l :: s -> (m, a),
  r :: (m, a') -> s'
}
\end{minted}
This not the approach usually taken by implementations! Instead the somewhat indirect \emph{profunctor encoding} is used. (This is not quite true for the Haskell \lenslib{} library, for a few reasons \lenslib{} uses the closely related \emph{van Laarhoven encoding}, see Section~\ref{sec:van-laarhoven}. The Purescript \texttt{purescript-profunctor-lenses} library~\cite{PurescriptLibrary} does use the profunctor encoding directly.)

The equivalence between the profunctor encoding and optics as described earlier has been explored in~\cite{ProfunctorOptics} and~\cite{ProfunctorOpticsPost}. We begin by reviewing this equivalence from a categorical perspective before investigating how the optic laws manifest in this setting.

\subsection{Tambara Modules}
Let $I = \C(-,{=}) : \C \hto \C$ be the identity profunctor and $\odot$ be profunctor composition, written in diagrammatic order. The following section generalises definitions that first appeared in~\cite[Section 3]{Doubles} for monoidal categories to the more general case of a monoidal action.

\begin{definition}
  Suppose a category $\C$ is acted on by $(\M, \otimes, I)$ and let $P \in \Prof(\C, \C)$ be a profunctor. A \emph{Tambara module structure for $\M$ on $P$} is a family of maps:
  \begin{align*}
    \zeta_{A,B,M} : P(A,B) \to P(M \act A, M\act B)
  \end{align*}
  natural in $A$ and $B$, dinatural in $M$, and such that $\zeta$ commutes with the action of $\M$:
  \[
    \begin{tikzcd}
      P(A,B) \ar[r, "\zeta_{A,B,M}"] \ar[d, "\zeta_{A, B, N\otimes M}" swap] & P(M \act A, M \act B) \ar[d, "\zeta_{M \act A, M \act B, N}" right] \\
      P((N\otimes M) \act A), (N\otimes M) \act B) \ar[r, "\alpha_{N, M, A}" swap] & P(N\act (M\act A), N\act (M \act B))
    \end{tikzcd}
    \qquad
    \begin{tikzcd}
      P(A,B) \ar[r, "\zeta_{A,B,I}"] \ar[dr, equal] & P(I\act A, I\act B) \ar[d, "{P(\lambda_A^{-1}, \lambda_B)}" right] \\
      & P(A, B)
    \end{tikzcd}
  \]
  for all $A, B \in \C$ and $N, M \in \M$.
\end{definition}

Note that the identity profunctor $I$ has a canonical Tambara module structure $\zeta_{A, B, M} : \C(A, B) \to \C(M \act A, M \act B)$ for any $\M$, given by functoriality.

If $P, Q \in \Prof(\C, \C)$ are equipped with module structures $\zeta$ and $\xi$ respectively, there is a canonical module structure on $P \odot Q$. Given $M \in \M$ and $A,B \in \C$, the structure map ${(\zeta \odot \xi)}_{A,B,M}$ is induced by
\begin{align*}
  &P(A,C) \times Q(C,B)  \\
  \xrightarrow{\zeta_{A,C,M} \times \xi_{C,B,M}} \quad& P(M\act A, M\act C) \times Q(M\act C, M\act B) \\
  \xrightarrow{\copr_{M\act C}} \quad&\int^{C \in \C} P(M\act A, C) \times Q(C, M\act B) \\
  = \quad&(P \odot Q)(M\act A, M\act B)
\end{align*}

\begin{definition}
  There is a category $\Tamb_\M$ of Tambara modules and natural transformations that respect the module structure, in the sense that for any $l : P \to Q$, the diagram
  \[
    \begin{tikzcd}
      P(A,B) \ar[r, "\zeta_{A,B,M}"] \ar[d, "l_{A,B}" left] & P(M\act A, M\act B) \ar[d, "l_{M\act A, M\act B}" right] \\
      Q(A,B) \ar[r, "\xi_{A,B,M}" swap] & Q(M \act A, M \act B)
    \end{tikzcd}
  \]
  commutes.
\end{definition}

This category is monoidal with respect to $\odot$ as given above with monoidal unit $I$. There is an evident forgetful functor $U : \Tamb_\M \to \Prof(\C, \C)$ that is strong monoidal. This forgetful functor has both a left and right adjoint; important for us is the left adjoint: (The right adjoint to $U$ is described in~\cite{NotionsOfComputationAsMonoids}, used there to investigate Haskell's \mintinline{haskell}{Arrow} typeclass.)

\begin{definition}[{\cite[Section 5]{Doubles}}]
  Let $\Pastro_\M : \Prof(\C, \C) \to \Tamb_\M$ be the functor:
  \begin{align*}
    \Pastro_\M(P) := \int^{M \in \M}  \C(-, M\act {=}) \odot P \odot \C(M\act -, {=})
  \end{align*}
  Or, in other words,
  \begin{align*}
    \Pastro_\M(P)(A,B) := \int^{M \in \M} \int^{C,D \in \C} \C(A, M\act C) \times P(C,D) \times  \C(M \act D, B)
  \end{align*}
  The module structure $\zeta_{A,B,M} : \Pastro_\M P(A,B) \to \Pastro_\M P (M\act A, M\act B)
  $ is induced by the maps
  \begin{align*}
    &\C(A, N\act C) \times P(C,D) \times  \C(N\act D, B) \\
    \xrightarrow{\text{functoriality}} \quad& \C(M\act A, M\act N\act C) \times P(C,D) \times  \C(M\act N\act D, M\act B) \\
    \xrightarrow{\copr_{M\otimes N}} \quad&\int^{N \in \M} \int^{C,D \in \C} \C(M\act A, N\act C) \times P(C,D) \times  \C(N\act D, M\act B) \\
    = \quad&\Pastro_\M P (M \act A, M \act B)
  \end{align*}
  for all $C, D \in \C$ and $N \in \M$. Equationally, this is $\zeta_{A,B,M}(\repthree{l}{p}{r} ) = \repthree{M\act l}{p}{M\act r} $.
\end{definition}

\begin{proposition}
  $\Pastro_\M : \Prof(\C, \C) \to \Tamb_\M$ is left adjoint to $U : \Tamb_\M \to \Prof(\C, \C)$.
\end{proposition}
\begin{proof}
  For any $P \in \Prof(\C, \C)$, there is a map $\eta : P \hto U \Pastro_\M P$, given by $\eta(p) = \repthree{\id_A}{p}{\id_B}$. Suppose we have an element $\repthree{l}{p}{r} \in \Pastro_\M P(A,B)$. One can check that this element is equal to
  \begin{align*}
    \repthree{l}{p}{r} = (\Pastro_\M P(l, r)) \zeta_{A, B, M} \; \eta(p)
  \end{align*}
  where $\zeta_{A, B, M}$ is the module structure map for $\Pastro_\M P$.

  If $T \in \Tamb_\M$ is a Tambara module with structure map $\xi$, we would like to show that for any map $f : P \hto UT$ there exists a unique $\hat f : \Pastro_\M P \hto T$ so that $f$ factors as \[P \xrightarrow{\eta} U \Pastro_\M P \xrightarrow{U\hat f} UT. \]

  The data of such a map $\hat f : \Pastro_\M P \hto T$ is a natural transformation between the underlying profunctors. For the factorisation property to hold we must have that $\hat{f}\eta(p) = f(p)$ for any $p \in P(A,B)$, but then the action on the remainder of $\Pastro_\M P(A, B)$ is fixed:
  \begin{align*}
    \hat{f}(\repthree{l}{p}{r}) 
    &= \hat{f}(\Pastro_\M P(l, r) \zeta_{A, B, M} \; \eta(p)) \\
    &=T(l, r) \; \xi_{A, B, N} \; f(p)
  \end{align*}
  This establishes uniqueness. It remains to show that $\hat{f}$ so defined is actually a Tambara module morphism, but this is easy:
  \begin{align*}
    \hat{f}\zeta_{A,B,N}(\repthree{l}{p}{r})
    &= \hat{f}(\repthree{N\act l}{p}{N\act r}) && \text{(definition of $\zeta$)}\\
    &= T(N\act l, N\act r) \; \xi_{A, B, N \otimes M} \; f(p) && \text{(definition of $\hat{f}$)}\\
    &= T(N\act l, N\act r) \xi_{M\act A,M\act B,N} \; \xi_{A, B, M} \; f(p) && \text{($\xi$ commutes with tensor in $\M$)} \\
    &= \xi_{A,B,N} T(l, r) \; \xi_{A, B, M} \; f(p) && \text{(naturality of $\xi$)} \\
    &= \xi_{A,B,N} \hat{f} (\repthree{l}{p}{r}) && \text{(definition of $\hat{f}$)}
  \end{align*}
\end{proof}

\begin{corollary}
  $\Pastro_\M$ (and therefore also $U \Pastro_\M$) is oplax monoidal.
\end{corollary}
\begin{proof}
  This follows from abstract nonsense as $\Pastro_\M$ is the left adjoint of a strong monoidal functor, see~\cite{Kelly1974}.
\end{proof}

\subsection{Optics}
\begin{definition}
  For a pair of objects $A, A' \in \C$, the \emph{exchange profunctor} $E_{A, A'}$ is defined to be $\C(-, A) \times \C(A', {=})$.
\end{definition}

Given a profunctor, or indeed a Tambara module, we can evaluate it at any two objects of $\C$. This process is functorial in the choice of Tambara module, giving a functor $(U-)(A,A') : \Tamb_\M \to \Set$.

\begin{lemma}\label{lemma-rep}
  The functor $(U-)(A,A') : \Tamb_\M \to \Set$ is representable: there is a isomorphism $(U-)(A,A') \cong \Tamb_\M(\Pastro_\M E_{A, A'}, -)$
\end{lemma}
\begin{proof}
  We have the chain of isomorphisms:
  \begin{align*}
    &(U-)(A,A') \\
    \cong \;&\int_{X,Y \in \C} \Set(\C(X,A) \times \C(A',Y), (U-)(X,Y)) && \text{(by Yoneda reduction twice)} \\
    =\;&\int_{X,Y \in \C} \Set(E_{A, A'}(X,Y), (U-)(X,Y)) && \text{(by definition)}\\
    \cong \;&\Prof(E_{A, A'}, U-) && \text{(natural transformations as ends)} \\
    \cong \;&\Tamb_\M(\Pastro_\M E_{A, A'}, -) && \text{(by adjointness)}
  \end{align*}
\end{proof}

Note that the value of $\Pastro_\M E_{A, A'}$ at $(X,Y)$ is precisely the set of optics $(X, Y) \hto (A, A')$:
\begin{align*}
\Pastro_\M E_{A, A'} (X, Y)
&= \int^{M \in \M} \int^{C,D \in \C} \C(X, M\act C) \times E_{A, A'}(C,D) \times  \C(M\act D, Y) \\
&= \int^{M \in \M} \int^{C,D \in \C} \C(X, M\act C) \times \C(C, A) \times \C(A', D) \times \C(M\act D, Y) \\
&\cong \int^{M \in \M} \C(X, M\act A) \times \C(M\act A', Y)
\end{align*}
For convenience we identify $\Pastro_\M E_{A, A'}(X,Y)$ with $\Optic_\M((X, Y), (A, A'))$.

We can now show that profunctor optics are precisely optics in the ordinary sense.

\begin{proposition}[Profunctor Optics are Optics]\label{prop:profunctor-optics-are-optics}
  \begin{align*}
    [\Tamb_\M, \Set]((U-)(A,A'),(U-)(S,S')) &\cong \Optic_\M((S, S'), (A, A'))
  \end{align*}
\end{proposition}
\begin{proof}
  We have the chain of isomorphisms:
  \begin{align*}
    &[\Tamb_\M, \Set]((U-)(A,A'),(U-)(S,S')) \\
    \cong \;&[\Tamb_\M, \Set](\Tamb_\M(\Pastro_\M E_{A, A'}, -), (U-)(S,S'))  && \text{(by Lemma~\ref{lemma-rep})}\\
    \cong \;&(U\Pastro_\M E_{A, A'})(S,S')  && \text{(by Yoneda)} \\
    = \;&\Optic_\M((S, S'), (A, A'))
  \end{align*}
\end{proof}

For $p : (S, S') \hto (A, A')$, let $\tilde{p} : (U-)(A,A') \Rightarrow (U-)(S,S')$ denote the corresponding natural transformation under this isomorphism, and for $t : (U-)(A,A') \Rightarrow (U-)(S,S')$, let $\hat{t} : (S, S') \hto (A, A')$ be the corresponding optic.

\begin{corollary}
  A profunctor optic $t$ is determined by its component at $\Pastro_\M E_{A, A'}$, and furthermore, this component is determined by its value on $\rep{\lambda_A^{-1}}{\lambda_{A'}} \in (U \Pastro_\M E_{A, A'})(A, A')$.
\end{corollary}
\begin{proof}
  This is the content of the first two isomorphisms above. Explicitly, suppose $p = \rep{l}{r}$ with $l : S \to M\act A$ and $r : M\act A' \to S'$. Then for any Tambara module $P$, the component of $\tilde{p}$ at $P$ is
  \begin{align*}
    \tilde{p}_P = (UP)(l,r) \zeta_{A,A',M}
  \end{align*}
  where $\zeta$ is the module structure for $P$. In particular,
  \[
    \tilde{p}_{\Pastro_\M E_{A, A'}}(\rep{\lambda_A^{-1}}{\lambda_{A'}}) = \rep{l}{r}
  \]
\end{proof}

We finish with one final isomorphic description of an optic:

\begin{proposition}
  $\Optic_\M((S, S'), (A, A'))$ is isomorphic to $\Tamb_\M(\Pastro_\M E_{S, S'}, \Pastro_\M E_{A, A'})$.
\end{proposition}
\begin{proof}
This follows from the previous two propositions and the Yoneda lemma:
\begin{align*}
  &\Optic_\M((S, S'), (A, A')) \\
  &\cong [\Tamb_\M, \Set]((U-)(A,A'),(U-)(S,S')) \\
  &\cong [\Tamb_\M, \Set](\Tamb_\M(\Pastro_\M E_{A, A'}, -),\Tamb_\M(\Pastro_\M E_{S, S'}, -)) \\
    &\cong \Tamb_\M(\Pastro_\M E_{S, S'}, \Pastro_\M E_{A, A'})
\end{align*}
Explicitly, an optic $p = \rep{l}{r}$ corresponds to the natural transformation with components:
\begin{align*}
t_{X, Y} : \Pastro_\M E_{S, S'}(X, Y) \to \Pastro_\M E_{A, A'}(X, Y) \\
t_{X, Y}(\rep{f}{g}) = \rep{(M\act l)f}{g(M\act r)}
\end{align*}
where $M$ is the residual for the representative $\rep{f}{g}$.
This is exactly the formula for optic composition!
\end{proof}

\subsection{Lawful Profunctor Optics}

The next goal is to characterise the profunctor optics that correspond to lawful optics.

The exchange profunctor $E_{A, A}$, hereafter abbreviated to $E_A$, has a comonoid structure, where the comultiplication $\Delta : E_A \to E_A \odot E_A$ and counit $\varepsilon : E_A \to \C$ are given by
\begin{align*}
  \Delta_{X, Y} : (E_A)(X, Y) &\to (E_A \odot E_A)(X, Y) \\
  \Delta_{X, Y}(\rep{f}{g}) &= \repthree{f}{\id_A}{g}  \\
  \varepsilon_{X, Y} : (E_A)(X, Y) &\to \C(X, Y) \\
  \varepsilon_{X, Y}(\rep{f}{g}) &= gf
\end{align*}
respectively. Here we have identified $E_A \odot E_A$ with the profunctor $\C(-, A) \times \C(A, A) \times \C(A, =)$, via the isomorphism
\begin{align*}
E_A \odot E_A 
&= \int^{Z \in \C} E_A(-, Z) \times E_A(Z, =) \\
&= \int^{Z \in \C} \C(-, A) \times \C(A, Z) \times \C(Z, A) \times \C(A, =) \\
&\cong \C(-, A) \times \C(A, A) \times \C(A, =) 
\end{align*}

Because $\Pastro_\M$ is oplax monoidal, the Tambara module $\Pastro_\M E_A$ has an induced comonoid structure, in this case given by
\begin{align*}
  \Delta_{X, Y} : (\Pastro_\M E_A)(X, Y) &\to (\Pastro_\M E_A \odot \Pastro_\M E_A)(X, Y) \\
  \Delta(\rep{l}{r}) &= \repthree{l}{\id_{M\act A}}{r} \\
  \varepsilon_{X, Y} : (\Pastro_\M E_A)(X, Y) &\to \C(X, Y) \\
  \varepsilon(\rep{l}{r}) &= rl
\end{align*}

The connection with lawfulness is hopefully now evident!

\begin{proposition}\label{prop:lawful-if-homomorphism}
  An optic $p : S \hto A$ is lawful iff the corresponding natural transformation $\Pastro_\M E_S \rightarrow \Pastro_\M E_A$ is a comonoid homomorphism.
\end{proposition}
\begin{proof}
For $t : \Pastro_\M E_S \rightarrow \Pastro_\M E_A$ to be a comonoid homomorphism means that the following diagrams commute for every $X, Y \in \C$:
  \[
    \begin{tikzcd}
    (\Pastro_\M E_S)(X, Y) \ar[r, "t_{X, Y}"] \ar[d, "\varepsilon_{X,Y}" swap] & (\Pastro_\M E_A)(X, Y) \ar[d, "\varepsilon_{X,Y}"] \\
    \C(X, Y) \ar[r, equals] & \C(X, Y)
    \end{tikzcd}
    \quad
    \begin{tikzcd}
      (\Pastro_\M E_S)(X, Y) \ar[r, "t_{X, Y}"] \ar[d, "\Delta_{X, Y}" swap] & (\Pastro_\M E_A)(X, Y) \ar[d, "\Delta_{X, Y}"] \\
      (\Pastro_\M E_S \odot \Pastro_\M E_S)(X, Y) \ar[r, "(t \odot t)_{X, Y}" swap] & (\Pastro_\M E_A \odot \Pastro_\M E_A)(X, Y)
    \end{tikzcd}
  \]
  Suppose $t$ corresponds to an optic with representative $\rep{l}{r}$ with residual $M$ and we have an element $\rep{f}{g} : (\Pastro_\M E_S)(X, Y)$ with residual $N$. The left diagram requires that 
  \begin{align*}
  g(Nr)(Nl)f = gf,
  \end{align*}
  as an element of $\C(X, Y)$. This is certainly true as $rl = \id_S$. The right diagram claims that 
  \begin{align*}
  \repthree{(N\act l)f}{\id_{N\act M\act A}}{g(N \act r)} = \repthree{(N\act l)f}{(N\act r)(N\act l)}{g(N\act r)}
  \end{align*}
  But this holds by exactly the same argument as used in Proposition~\ref{prop:lawful-category} to show that the composite of lawful optics is lawful: by transplanting the relations showing the second optic law for $\rep{l}{r}$

  For the backward direction, consider the above diagrams specialised to $X = Y = S$. Tracing the element $\rep{\lambda_S^{-1}}{\lambda_S} \in (\Pastro_\M E_S)(S, S)$ around the commutative diagrams yields precisely the first and second optic laws respectively.
\end{proof}

All that is needed to complete the connection with profunctor optics is the following standard result in category theory.

\begin{lemma}
  For an object $X$ in a monoidal category $(\C, \otimes, I)$, a comonoid structure $(X,\Delta,\varepsilon)$ is equivalent to a lax monoidal structure on the functor $\C(X, -) : \C \to \Set$, considering $\Set$ as a monoidal category with respect to $\times$. 
  
  Further, a morphism $(X_1,\Delta_1,\varepsilon_1) \to (X_2,\Delta_2,\varepsilon_2)$ is a comonoid homomorphism iff the induced natural transformation $\C(X_2, -) \Rightarrow \C(X_1, -)$ is monoidal.
\end{lemma}
\begin{proof}
This is a follow-your-nose result!
%  For a comonoid $(X,\Delta,\varepsilon)$, define a lax monoidal structure
%  \begin{align*}
%    \phi &: 1 \to \C(X, I) \\
%    \phi_{A, B} &: \C(X, A) \times \C(X, B) \to \C(X, A \otimes B)
%  \end{align*}
%  by $\phi = \varepsilon$ and $\phi_{A, B}(f, g) = (f \otimes g) \Delta$. The coherences follow straightforwardly using the comonoid axioms and functoriality of $\otimes$.
%
%  In the other direction,  we recover the comonoid maps by $\varepsilon = \phi$ and $\Delta = \phi_{X, X}(\id_X, \id_X)$.
\end{proof}

\begin{theorem}
  $p : S \hto A$ is a lawful optic iff the associated natural transformation $\tilde{p} : (U-)(A,A) \Rightarrow (U-)(S,S)$ is monoidal with respect to the canonical lax monoidal structures on $(U-)(A,A)$ and $(U-)(S,S)$.
\end{theorem}
\begin{proof}
\begin{align*}
& p : S \hto A \text{ is lawful} \\
\Leftrightarrow\; & \Pastro_\M E_S \to \Pastro_\M E_A \text{ is a comonoid homomorphism} \\
\Leftrightarrow\; & \Tamb_\M(\Pastro_\M E_A, -) \Rightarrow \Tamb_\M(\Pastro_\M E_S, -) \text{ is a monoidal natural transformation} \\
\Leftrightarrow\; & (U-)(A, A) \Rightarrow (U-)(S, S) \text{ is a monoidal natural transformation}
\end{align*}
\end{proof}

\subsection{Implementation}
We review quickly how the profunctor encoding is translated into code in the Haskell~\cite{LensLibrary} and Purescript~\cite{PurescriptLibrary} libraries. We define a typeclass for profunctors:
\begin{minted}{haskell}
class Profunctor p where
  dimap :: (a -> b) -> (c -> d) -> p b c -> p a d
\end{minted}
To be considered a valid instance of \mintinline{haskell}{Profunctor}, the function \mintinline{haskell}{dimap} must behave functorially. Now, for each optic variant we wish to define, we create a typeclass for the corresponding Tambara module. In the case of $\Lens$es, this typeclass is named \mintinline{haskell}{Strong}:
\begin{minted}{haskell}
class Profunctor p => Strong p where
  second :: p a b -> p (c, a) (c, b)
\end{minted}
This \mintinline{haskell}{second} function is the equivalent of the structure map $\zeta$ for the Tambara module. We require this map to satisfy the Tambara module coherences, but as with any definition in Haskell, these equations must be checked manually.

Now the type of lenses $(S, S') \hto (A, A')$ is the direct translation of the set of natural transformations $(U-)(A,A') \Rightarrow (U-)(S,S')$:
\begin{minted}{haskell}
type Lens s s' a a' = forall p. Strong p => p a a' -> p s s'
\end{minted}
where we use parametricity in \mintinline{haskell}{p} as a proxy for naturality. A profunctor lens 
\begin{minted}{haskell}
l :: forall p. Strong p => p a a -> p s s
\end{minted}
is lawful if it is monoidal as a natural transformation. In code this is:
\begin{minted}{haskell}
l id == id
l (Procompose p q) == Procompose (l p) (l q)
\end{minted}
where
\begin{minted}{haskell}
data Procompose p q d c where
  Procompose :: p x c -> q d x -> Procompose p q d c
\end{minted}
denotes profunctor/Tambara module composition, once equipped with appropriate \mintinline{haskell}{Profunctor} and \mintinline{haskell}{Strong} instances.

%The profunctor encoding has a number of benefits. The primary benefit is that the typeclass system allows optics to automatically degrade from one variant to another as needed. For example, if we wish to also encode $\Traversal$s, we would define the corresponding class of Tambara modules
%\begin{minted}{haskell}
%class Strong p => Wandering p where
%  wander :: Traversable f => p a b -> p (f a) (f b)
%\end{minted}
%Note that \mintinline{haskell}{Strong} is a superclass. An instance of \mintinline{haskell}{Wandering} is required to behave the same way 
%So given a 
%
%
%Another benefit is that optics compose via ordinary function composition

\subsection{The van Laarhoven Encoding}\label{sec:van-laarhoven}

Some optic variants can be encoded in a profunctor-like style without requiring the full complexity of profunctors. Chronologically this development came before profunctor optics, and was first introduced by Twan van Laarhoven~\cite{VanLaarhovenPost}.

The van Laarhoven encoding for \mintinline{haskell}{Lens}es, \mintinline{haskell}{Traversal}s and \mintinline{haskell}{Setter}s is:
\begin{minted}{haskell}
type Lens s a      = forall f. Functor f     => (a -> f a) -> (s -> f s)
type Traversal s a = forall f. Applicative f => (a -> f a) -> (s -> f s)
type Setter s a    = forall f. Settable f    => (a -> f a) -> (s -> f s)
\end{minted}
What allows such a description to work for these particular optic variants is that the Tambara module that characterises them, $\Pastro_\M E_A$, can be written in the form $\C(-, \mintinline{haskell}{f}=)$ for some \mintinline{haskell}{f} that is an instance of the corresponding typeclass. This is possible in particular for the optic variants that admit a coalgebraic description; the ones for which the evaluation-at-$A$ functor has a right adjoint.

No expressive power is lost by defining an optic to operate only on functions of the shape \mintinline{haskell}{a -> f a'}, as the entire concrete description of the optic can be extracted from its value on that particular Tambara module. The same is not true for other optic variants, and indeed in the Haskell \lenslib{} library, \mintinline{haskell}{Prism}s and \mintinline{haskell}{Review}s take a form much closer to the profunctor encoding. (The \lenslib{} library does not use \emph{precisely} the profunctor encoding even here, for backwards compatibility reasons.)

A consequence is that the laws typically given for $\Traversal$s actually only need to be checked for the applicative functor we earlier called $UR^*$. In Haskell this functor is implemented as \mintinline{haskell}{FunList}~\cite{FunListPost} or \mintinline{haskell}{Bazaar}~\cite{LensLibrary}.

\section{Future Work}

There are many avenues for future exploration!

\subsection{Mixed Optics}

One can generalise the definition of $\Optic$ so that the two halves lie in different categories. Suppose $\C_L$ and $\C_R$ are categories that are acted on by a common monoidal category $\M$. Write these actions as $\actL : \M \to [\C_L, \C_L]$ and $\actR : \M \to [\C_R, \C_R]$ respectively.

\begin{definition}
  Given two objects of $\C_L \times \C_R^\op$, say $(S, S')$ and $(A, A')$, a \emph{mixed optic} $p : (S, S') \hto (A, A')$ for $\actL$ and $\actR$ is an element of the set
  \begin{align*}
    \Optic_{\actL, \actR}((S, S'), (A, A')) := \int^{M \in \M} \C_L(S, M \actL A) \times \C_R(M \actR A', S')
  \end{align*}
\end{definition}
$\Optic_{\actL, \actR}$ forms a category. It is not so clear what notion of lawfulness is appropriate in this setting.

Examples of mixed optics include the \emph{degenerate optics} of the \lenslib{} library: \mintinline{haskell}{Getter}s, \mintinline{haskell}{Review}s and \mintinline{haskell}{Fold}s. The mixed optic formalism also appears able to capture \emph{indexed optics} such as \mintinline{haskell}{IndexedLens}es and \mintinline{haskell}{IndexedTraversal}s~\cite{ProfunctorOpticsPost}.

\subsection{Monotonic Lenses}
In the bidirectional transformation community, the $\fput\fput$ law is often considered too strong. In particular we have seen that in $\Set$, together with the other laws, it implies that $\fget$ must be a projection from a product.

To overcome this we work in $\Cat$, so that the objects under consideration have internal morphisms that we think of as updates. We modify $\fput$ so that instead of accepting an object $a$ of $A$ to overwrite the original in $S$ with, it requires a morphism in $A$ of the form $\fget(s) \to a$. In this way we are restricted in what updates we may perform. This is captured in the following definition:

\begin{definition}[{\cite[Definition 4.1]{LensesFibrationsAndUniversalTranslations}}]
A \emph{c-lens} $S \hto A$ in $\Cat$ is a pair of functors
\begin{align*}
\fget &: S \to A \\
\fput &: (\fget \downarrow \id_{A}) \to S
\end{align*}
such that a version of the three lens laws hold, where $(\fget \downarrow \id_{A})$ denotes the comma category construction.
\end{definition}

We can rewrite this in a form that gives hope for a correspondence with some optic category:

%\todo{I don't know if this is known:}
\begin{theorem}
The data of a c-lens $S \hto A$ corresponds to a functor \[ S \to \int [(- / A), S] \] where $(-/A)$ denotes the slice category and $\int$ denotes (confusingly!) the Grothendieck construction.

Furthermore, a c-lens is lawful iff it is a coalgebra for the comonad of the adjunction
\[
\begin{tikzcd}[column sep = large]
{[A^\op, \Cat]} \ar[r, bend left, "\int"] \ar[r, phantom, "\bot" pos = 0.4] & \Cat \ar[l, bend left, "{X \mapsto [(-/A), X]}" below]
\end{tikzcd}
\]
\qed
\end{theorem}

It is not clear whether there is an action on $\Cat$ that generates this description as its concrete optics. There doesn't seem to be a natural place for an $A'$ to appear! We remain optimistic:

\begin{conjecture}
c-lenses are the lawful (possibly mixed) optics for some action on $\Cat$.
\end{conjecture}

\subsection{Functor and Monad Transformer Lenses}
These were considered by Edward Kmett~\cite{MonadTransformerLensesTalk} as a method for embedding pieces of a monad transformer stack into the whole. There is some debate about the correct categorical description of monad transformers~\cite{MonadTransformersAsMonoidTransformers, CalculatingMonadTransformersCategoryTheory}, so we do not attempt to say anything precise, but the perspective given here could help in a couple of ways.

Kmett considers optics for the operation of composing two monad transformers. The primary test-case was to embed \mintinline{haskell}{ReaderT} actions into \mintinline{haskell}{StateT} actions, but from the constant-complement perspective, this is impossible: \mintinline{haskell}{StateT} does not factor as the composite of \mintinline{haskell}{ReaderT} with some other monad transformer. In this setting the constant-complement laws may be asking too much, the optic laws given here might be the correct notion of lawfulness for monad transformers.

Also, instead of considering optics within a category of monad transformers, we could instead look at optics for the action of monad transformers on monads. One can indeed define an optic $\mintinline{haskell}{State} \hto \mintinline{haskell}{Reader}$ that uses residual $\mintinline{haskell}{StateT}$. Whether this is lawful or useful is not clear!

\subsection{Learners}
A recent paper in applied category theory~\cite{BackpropAsFunctor} describes a compositional approach to machine learning, with a category whose morphisms describe learning algorithms.

\begin{definition}[{\cite[Definition 2.1]{BackpropAsFunctor}}]
For $A$ and $B$ sets, a \emph{learner} $A \hto B$ is a tuple $(P, I, U, r)$ where $P$ is a set, and $I$, $U$, and $r$ are functions of shape:
\begin{align*}
I &: P \times A \to B \\
U &: P \times A \times B \to P \\
r &: P \times A \times B \to A
\end{align*}
\end{definition}
To form a category, one must consider learners up to an equivalence relation on the sets $P$. There is an alternate slick description of the set of learners $A \hto B$, that goes as follows. Note that the data of a learner is describing an element of the coend
\begin{align*}
\int^{P, Q \in \Set} \Set(P \times A, Q \times B) \times \Set(Q \times B, P \times A)
\end{align*}
via the isomorphisms
\begin{align*}
&\int^{P, Q \in \Set} \Set(P \times A, Q \times B) \times \Set(Q \times B, P \times A) \\
&\cong \int^{P, Q \in \Set} \Set(P \times A, Q) \times \Set(P \times A, B) \times \Set(Q \times B, P \times A) \\
&\cong \int^{P \in \Set} \Set(P \times A, B) \times \Set(P \times A \times B, P \times A) \\
&\cong \int^{P \in \Set} \Set(P \times A, B) \times \Set(P \times A \times B, P) \times \Set(P \times A \times B, A)
\end{align*}
Composition of learners can be defined analogously to composition for optics. This perspective explains the slight fussing around required in dealing with equivalence classes of learners, and suggests a generalisation to other monoidal categories.

%\section{Conclusion}
%
%\todo{Maybe some of the long equational manipulations could be done by commutative diagram, using a pair of dashed lines to signal the start and end value?}
%
%\todo{Move some proofs to an appendix?}
%
%\todo{Find an example where the lens laws are definitely not equivalent}

\bibliographystyle{alpha}
\bibliography{optics.bib}
\end{document}